\documentclass[11pt]{article}

\usepackage{lmodern} 
\usepackage{hyperref} 
\usepackage{amssymb,amsmath,accents}
\usepackage{cases} 
\usepackage{amscd}
\usepackage{amsfonts,amsthm}
\usepackage{mathrsfs}
\usepackage{setspace} 
\usepackage{graphicx} 
\usepackage{graphicx} 
\usepackage{subcaption} 
\usepackage{soul} 
\usepackage{relsize} 
\usepackage{verbatim}
\usepackage{stackengine} 

\usepackage{color}

\numberwithin{equation}{section}
\newtheorem{thm}{Theorem}

\newtheorem{lem}[thm]{Lemma}
\newtheorem{prop}[thm]{Proposition}

\newtheorem{definition}[thm]{Definition}
\newtheorem{rem}[thm]{Remark}

\DeclareMathOperator*{\argmin}{arg\,min} 
\DeclareMathOperator*{\esslim}{ess\,lim} 

\begin{document}
\title{The fourth-order total variation flow in $\mathbb{R}^n$\\
\vspace{7mm}
\small{Dedicated to Professor Neil Trudinger on the occasion of his 80th birthday}}
\author{Yoshikazu Giga\footnote{The work of the first author was partly supported by the Japan Society for the Promotion of Science (JSPS) through the grants Kakenhi: No.\,19H00639, No.\,18H05323, No.\,17H01091, and by Arithmer Inc.\ and Daikin Industries, Ltd.\ through collaborative grants.}, 
Hirotoshi Kuroda\footnote{The work of the second author was partly supported by JSPS through the grant Kakenhi No.\,18H05323.} and Micha{\l} {\L}asica\footnote{This work was created during the last author's JSPS Postdoctoral Research Fellowship at the University of Tokyo. The last author was partly supported by the Kakenhi Grant-in-Aid No.\,21F20811.}}
\date{}

\maketitle
\thispagestyle{empty}

\abstract{
		We define rigorously a solution to the fourth-order total variation flow equation in $\mathbb{R}^n$. If $n\geq3$, it can be understood as a gradient flow of the total variation energy in $D^{-1}$, the dual space of $D^1_0$, which is the completion of the space of compactly supported smooth functions in the Dirichlet norm. However, in the low dimensional case $n\leq2$, the space $D^{-1}$ does not contain characteristic functions of sets of positive measure, so we extend the notion of solution to a larger space. We characterize the solution in terms of what is called the Cahn-Hoffman vector field, based on a duality argument. This argument relies on an approximation lemma which itself is interesting.
		
		We introduce a notion of calibrability of a set in our fourth-order setting. This notion is related to whether a characteristic function preserves its form throughout the evolution. It turns out that all balls are calibrable. However, unlike in the second-order total variation flow, the outside of a ball is calibrable if and only if $n\neq2$. If $n\neq2$, all annuli are calibrable, while in the case $n=2$, if an annulus is too thick, it is not calibrable.
		
		We compute explicitly the solution emanating from the characteristic function of a ball. We also provide a description of the solution emanating from any piecewise constant, radially symmetric datum in terms of a system of ODEs.}

\vspace{17pt}
\noindent
\small{2020 Mathematics Subject Classification, Primary 35K67 (Singular parabolic equations), Secondary 35K25 (Higher-order parabolic equations), 47J35 (Nonlinear evolution equations).}


\section{Introduction} \label{S1} 

We consider the fourth-order total variation flow equation in $\mathbb{R}^n$ of the form
	\begin{equation} \label{E1}
		u_t = -\Delta \operatorname{div} \frac{\nabla u}{|\nabla u|}.
	\end{equation}
We aim to give explicit description of its solutions emanating from piecewise constant radial data. However, it turns out that the definition of a solution is itself non-trivial since $-\Delta$ does not have a bounded inverse on $L^2(\mathbb{R}^n)$. Our first goal is thus to provide a rigorous definition of a solution. Our second goal is to find explicit formula for the solution to \eqref{E1} when the initial datum $u(0,x)=u_0(x)$ is the characteristic function of a ball or an annulus.
	In other words,
	\[
	u_0 = a_0 \mathbf{1}_{B_{R_0}} \quad\text{or}\quad
	u_0 = a_0 \mathbf{1}_{A^{R^1_0}_{R^0_0}} \quad
	a_0 \in \mathbb{R},
	\]
	where $\mathbf{1}_K$ is the characteristic function of a set $K\subset\mathbb{R}^n$, i.\,e.,
	\begin{equation*}
		\mathbf{1}_K(x)= \left \{
		\begin{array}{ll}
			1,& x \in K \\
			0,& x\in\mathbb{R}^n \backslash K.
		\end{array}
		\right.
	\end{equation*}
Here $B_R$ denotes the open ball of radius $R$ centered at $0\in\mathbb{R}^n$ and $A^{R^1}_{R^0}$ denotes the annulus defined by $A^{R^1}_{R^0}=B_{R_1}\setminus\overline{B_{R_0}}$.
Our major concern is whether or not the solution remains a characteristic function throughout the evolution.
For example, in the case $u_0=a_0\mathbf{1}_{B_{R_0}}$, whether or not the solution $u$ of \eqref{E1} is of the form
\[
	u(t,x) = a(t) \mathbf{1}_{B_{R(t)}}
\]
with a function $a=a(t)$. In other words, we are asking whether the speed $u_t$ on the ball $B_{R(t)}$ and on its complement are constant in the spatial variable. As in the second-order problem \cite{ACM} (see also \cite{GP}), this leads to the notion of calibrability of a set. In the case of the second-order problem $u_t=\operatorname{div}\left(\nabla u/|\nabla u|\right)$, a ball and its complement are always calibrable and $R(t)\equiv R_0$, i.\,e.\ the ball does not expand nor shrink \cite{ACM}. In our problem, $R(t)$ may be non-constant.

We first note that the definition of a solution itself is non-trivial.
The fourth-order total variation flow has been mainly studied in the periodic setting \cite{GK}, \cite{GG} or in a bounded domain with some boundary conditions \cite{GKM}.
Formally, it is a gradient flow of the total variation functional
\[
TV(u) := \int_\Omega | \nabla u |
\]
with respect to the inner product
\[
(u,v)_{-1} = \int_\Omega u (-\Delta)^{-1} v
\]
when $\Omega$ is a domain in $\mathbb{R}^n$ or a flat torus $\mathbb{T}^n$. In the periodic setting, i.\,e.\ $\Omega=\mathbb{T}^n$ as in \cite{GK}, \cite{GG}, it is interpreted as a gradient flow in $H^{-1}_{av}$ which is the dual space of $H^1_{av}$, the space of average-free $H^1$ functions equipped with the inner product
\[
(u,v)_1 = \int_\Omega \nabla u \cdot \nabla v.
\] 
For the homogeneous Dirichlet boundary condition with bounded $\Omega$, $H^{-1}_{av}$ is replaced by $D^{-1}$, the dual space of $D^1_0=D^1_0(\Omega)$, which is the completion of $C^\infty_c(\Omega)$ in the norm associated with the inner product $(u,v)_1$; here $C^\infty_c(\Omega)$ denotes the space of all smooth functions compactly supported in $\Omega$. By the Poincar\'e inequality, both $H^1_{av}$ and $D^1_0(\Omega)$ can be regarded as subspaces of $L^2(\Omega)$. However, if $\Omega$ equals $\mathbb{R}^n$, the situation is more involved. If $n\geq3$, $D^1_0(\mathbb{R}^n)$ is continuously and densely embedded in $L^{2^*}(\mathbb{R}^n)$, where $p^*=np/(n-p)$ so that $2^* = 2n/(n-2)$, by the Sobolev inequality.
In fact,
\[
D^1_0(\mathbb{R}^n) = D^1(\mathbb{R}^n) \cap L^{2^*}(\mathbb{R}^n),\quad
D^1(\mathbb{R}^n)= \left\{ u \in L^1_{loc}(\mathbb{R}^n)  \bigm| \nabla u \in L^2(\mathbb{R}^n) \right\}
\]
see e.\,g.\ \cite{Gal}.
On the other hand, if $n\leq2$, $D^1_0$ is isometrically identified with the quotient space $\dot{D}^1(\mathbb{R}^n):=D^1(\mathbb{R}^n)/\mathbb{R}$, when $D^1(\mathbb{R}^n)$ is equipped with inner product $(u,v)_1$ \cite{Gal}.
Thus, we need to be careful when $n\leq2$ because an element of $D^1_0(\mathbb{R}^n)$ is determined only up to a constant. 
In any case, $D^1_0(\mathbb{R}^n)$ is a Hilbert space with the scalar product 
	\[\left(u,v\right)_{D^1_0(\mathbb{R}^n)} = \int_{\mathbb{R}^n} \nabla u \cdot \nabla v.\]
	Therefore, we can identify $D^1_0(\mathbb{R}^n)$ with its dual space by means of the isometry 
	\[ -\Delta \colon u \mapsto (u, \cdot)_{D^1_0(\mathbb{R}^n)}.\]
	On the other hand, let us define a subspace $\widetilde{D}^{-1}(\mathbb{R}^n) \subset D^1_0(\mathbb{R}^n)'$ by    
	\[\widetilde{D}^{-1}(\mathbb{R}^n) = \left\{w \mapsto \int_{\mathbb{R}^n} u w\colon  u \in C_c^\infty(\mathbb{R}^n) \right\} \quad \text{if } n \geq 3,\]
	\[\widetilde{D}^{-1}(\mathbb{R}^n) = \left\{w \mapsto \int_{\mathbb{R}^n} u w\colon u \in C_{c,av}^\infty(\mathbb{R}^n) \right\} \quad \text{if } n = 1 \text{ or } n=2,\]
	where 
	\[C_{c,av}^\infty(\mathbb{R}^n) = \left\{u \in C_c^\infty(\mathbb{R}^n) \colon \int_{\mathbb{R}^n} u = 0\right\}.\]   
	Then the closure $D^{-1}(\mathbb{R}^n)$ of $\widetilde{D}^{-1}(\mathbb{R}^n)$ coincides with $D^1_0(\mathbb{R}^n)'$ \cite{Gal}. Note that the restriction to $C_{c,av}^\infty(\mathbb{R}^n)$ in the definition of $\widetilde{D}^{-1}(\mathbb{R}^n)$ in $n=1,2$ is necessary for the functionals to be well-posed on $D^1(\mathbb{R}^n)/\mathbb{R}$. 
	In any case, since (by definition) the space of test functions $\mathcal D(\mathbb{R}^n)$ is continuously and densely embedded in $D_0^1(\mathbb{R}^n)$, we also have a continuous embedding $D^{-1}(\mathbb{R}^n) = D^1_0(\mathbb{R}^n)' \subset \mathcal D'(\mathbb{R}^n)$. Throughout the paper, we will often drop $(\mathbb{R}^n)$ in the notation for spaces of functions on $\mathbb{R}^n$, e.\,g.\ $D^{-1} = D^{-1}(\mathbb{R}^n)$.  

	In the first step, we give a rigorous definition of the total variation functional $TV$ on $D^{-1}$. Then we calculate the subdifferential of $TV$ in $D^{-1}$ space. Since it is a homogeneous functional, we are able to apply a duality method \cite{ACM} to characterize the subdifferential, provided that $TV$ is well approximated by nice functions in $D^{-1}$. We know that $C^\infty_{c,av}(\mathbb{R}^n)$ is dense in $D^{-1}$ for $n\leq2$; see e.\,g.\ \cite{Gal}.
	However, it is not immediately clear whether $TV$ is simultaneously approximable. Fortunately, it turns out that for any $w\in D^{-1}$, there is a sequence $w_k\in C^\infty_{c,av}(\mathbb{R}^n)$ which converges to $w$ in $D^{-1}$ and $TV(w_k)\to TV(w)$ as $k\to\infty$. This approximation part is relatively involved since we have to use an efficient cut-off function. Using the approximation, we are able to characterize the subdifferential $\partial_{D^{-1}}TV$ of $TV$ in $D^{-1}$ by adapting the argument in \cite{ACM}.
	Namely, we have
	\[
	\partial_{D^{-1}}TV(u) = \left\{ v = \Delta\operatorname{div}Z \bigm|
	Z \in L^\infty(\mathbb{R}^n),\ |Z| \leq 1,\ 
	-\langle u, \operatorname{div} Z \rangle = TV(u) \right\},
	\]
	where $\langle\ ,\ \rangle$ denotes the canonical paring of $D^{-1}$ and $D^1_0$. A vector field $Z$ corresponding to an element of the subdifferential is often called a \emph{Cahn-Hoffman vector field}.
	The equation \eqref{E1} should be interpreted as the gradient flow of $TV$ in $D^{-1}$, i.\,e.
	\begin{equation} \label{E2}
		u_t \in -\partial_{D^{-1}} TV(u),
	\end{equation}
	and its unique solvability for any initial datum $u_0\in D^{-1}$ is guaranteed by the classical theory of maximal monotone operators (\cite{Ko}, \cite{Br}). By our characterization of the subdifferential, we are able to give a more explicit definition of a solution which is consistent with that proposed in \cite{GKM}.
	Namely, for $u_0\in D^{-1}$ with $TV(u_0)<\infty$, a function $u\in C\left([0,T[,D^{-1}\right)$ is a solution to \eqref{E2} with $u(0)=u_0$ if and only if there exists $Z\in L^\infty\left(]0,T[\times\mathbb{R}^n\right)$ satisfying $\operatorname{div}Z \in L^2 \left(0,T; D^1_0(\mathbb{R}^n)\right)$ such that
	\begin{align}
		& u_t = -\Delta\operatorname{div}Z \quad\text{in}\quad D^{-1}(\mathbb{R}^n) \label{E4.0} \\
		& \left| Z (t,x) \right| \leq 1 \quad\text{for a.\,e. } x \in \mathbb{R}^n \label{E5.0} \\
		& \langle u, \operatorname{div}Z \rangle = -TV(u) \label{E6.0}
	\end{align}
	for a.\,e.\ $t\in]0,T[$.
	This is convenient for calculating explicit solutions.

	Unfortunately, in $n\leq2$, for a compactly supported square integrable function $u_0$, we know that $u_0\in D^{-1}$ if and only if $u_0$ is average-free, i.\,e.\ $\int_{\mathbb{R}^n}u_0=0$ (see Lemma \ref{4DS}). Thus, the characteristic function of any bounded, measurable set of positive measure does not belong to $D^{-1}$. We have to extend a class of initial data $u_0$ such that $u_0=\psi+w_0$ with $w_0\in D^{-1}$ while $\psi$ is a fixed compactly supported $L^2$ function. We consider a gradient flow $u_t\in-\partial_{D^{-1}}TV(u)$ in the affine space $\psi + D^{-1}$. Since $\partial_{D^{-1}}$ is a directional partial derivative in the direction of $D^{-1}$, it is more convenient to consider solutions to evolutionary variational inequality
	\begin{equation} \label{EVI} 
	\frac12 \frac{\mathrm{d}}{\mathrm{d}t} \left\|u(t)-g\right\|^2_{D^{-1}}
	\leq TV(g) - TV\left(u(t)\right) \quad \text{for a.\,e. }t>0
	\end{equation}
	for any $g \in \psi +D^{-1}$ \cite{AGS}. In the case $\psi=0$, it is easy to show that the evolutionary variational inequality is equivalent to \eqref{E2}. Indeed, by definition of the subdifferential, \eqref{E2} is equivalent to
	\[
	\left(-u_t, g-u(t)\right)_{D^{-1}}
	\leq TV(g) - TV\left(u(t)\right)
	\]
	for any $g\in D^{-1}$.
	The left-hand side equals $(\mathrm{d}/\mathrm{d} t)\left(\|u-g\|^2/2\right)$.
	Thus, the equivalence follows if $\psi=0$.
	From now on we assume that $\int_{\mathbb{R}^n}\psi \neq0$. 
	
	It is easy to check that there is at most one solution to the evolutionary variational inequality \eqref{EVI}.
	The solution $u$ is constructed by solving
	\[
	w_t \in -\partial_{D^{-1}}TV(w+\psi) \quad\text{with}\quad
	w(0) = w_0 = u_0 - \psi
	\]
	and setting $u=w+\psi$. Characterization of the (directional) subdifferential is more involved since $w\mapsto TV(w+\psi)$ is no more positively one-homogeneous.
	We identify the one dimensional space $\{c\psi|c\in\mathbb{R}\}$ with $\mathbb{R}$ and consider the Hilbert space $E^{-1}$ defined as the orthogonal sum $D^{-1}\oplus\mathbb{R}$. We calculate the subdifferential by the duality method since $TV$ is now positively one-homogeneous on $E^{-1}$.
	We then project this subdifferential onto $D^{-1}$ to get a characterization of a (directional) subdifferential $\partial_{D^{-1}}TV$.
	We end up with a characterization of solution to \eqref{E2} similar to \eqref{E4.0}--\eqref{E6.0}, with \eqref{E6.0} adjusted in a suitable way. If we also denote $E_0^1 = D^1$ in $n \leq 2$, $E_0^1 = D_0^1$, $E^{-1} = D^{-1}$ in $n \geq3$ and 
	\begin{equation} 
		\langle u, v \rangle_E = \left\{\begin{array}{l} 
			\langle u, v\rangle \text{ if } n \geq 3,  \\
			\langle w, [v]\rangle + c \int \psi v, \text{ where } u = w + c \psi, \ w \in D^{-1} \text{ if } n \leq 2,   
		\end{array}\right.
	\end{equation} 
	for $u \in E^{-1}$, $v \in E^1_0$, we end up with the following definition of solution 
	\begin{definition} \label{DEF0}
		Assume that $u_0 \in E^{-1}$.
		We say that $u\in C([0,\infty[, E^{-1})$ with $u_t\in L^2_{loc}(]0,\infty[, D^{-1})$ is a solution to \eqref{E1} with initial datum $u_0$ if there exists $Z \in L^\infty(]0,\infty[\times \mathbb{R}^n)$ with $\operatorname{div} Z(t,\cdot) \in E_0^1$ for a.\,e.\ $t>0$ such that
	\begin{align}
	& u_t = -\Delta\operatorname{div}Z \quad\text{in}\quad D^{-1}(\mathbb{R}^n) \label{E4} \\
	& \left| Z (t,x) \right| \leq 1 \quad\text{for a.\,e. } x \in \mathbb{R}^n \label{E5} \\
	& - \langle u, \operatorname{div}Z \rangle_E = TV(u) \label{E6}
	\end{align}
	for a.\,e.\ $t>0$. 
	\end{definition} \noindent
	and associated well-posedness result 
	\begin{thm} \label{MAIN0} 
		Let $u_0 \in E^{-1}$. There exists a unique solution to \eqref{E1} with initial datum $u_0$. 
	\end{thm} 

	Our next problem is whether or not the speed of a characteristic function of a set is spatially constant inside and outside of the set.
	By the general theory (\cite{Br}, \cite{Ko}), the speed is determined by the minimal section (canonical restriction) $\partial^0_{D^{-1}}TV$ of $\partial_{D^{-1}}TV$.
	In other words, $\partial^0_{D^{-1}}TV(u)=v_0$ minimizes $\|v\|_{D^{-1}}$ for $v\in\partial_{D^{-1}}TV(u)$, i.\,e.,
	\[
	\partial^0_{D^{-1}}TV(u) := \argmin\left\{\|v\|_{D^{-1}}
	\bigm| v \in \partial_{D^{-1}} TV(u) \right\}.
	\]
	To motivate the notion of calibrability, we consider a smooth function $u$ such that
	\[
	\overline{U} = \left\{ x \in \mathbb{R}^n
	\bigm| u(x) = 0 \right\},
	\]
	where $U$ is a smooth open set.
	Outside $\overline{U}$, we assume that $\nabla u\neq0$.
	To fix the idea, we assume that $\partial U$ has negative signature (orientation) in the sense that $u<0$ outside $\overline{U}$.
	By our specification of $u$, we see that
	\[
	\partial^0_{D^{-1}}TV(u) = \argmin\left\{\|\operatorname{div}Z\|_{D^1_0}
	\Bigm| |Z| \leq 1\ \text{in}\ U,\ 
	Z = \nabla u/|\nabla u|\ \text{in}\ \overline{U}^c,\ 
	\operatorname{div}Z \in D^1_0 \right\}.
	\]
	Since $\operatorname{div}Z$ is locally integrable, $Z\cdot\nu$ does not jump across $\partial U$, where $\nu$ is the exterior unit normal of $\partial U$.
	In this case,
	\begin{equation} \label{E7}
		Z\cdot\nu = Z\cdot\nabla u/|\nabla u| = -1
		\quad\text{on } \partial U.
	\end{equation}
	Since $\nabla\operatorname{div}Z$ does not have a singular part, $\operatorname{div}Z$ does not jump across $\partial U$.
	In this case,
	\begin{equation} \label{E8}
		\operatorname{div}Z = -\operatorname{div}\nu
		\quad\text{on } \partial U.
	\end{equation}
	However, $v=\Delta\operatorname{div}Z$ may have a non-zero singular part concentrated on $\partial U$ even if $v=v_0$, i.\,e., $v$ is the minimizer.
	This phenomenon is observed in \cite{Ka1}, \cite{Ka2}, \cite{GG} in a one-dimensional periodic setting.
	Different from the second-order problem, this causes expansion or shrinking of the ball when the solution $u$ is of the form $u(t,x)=a(t)\mathbf{1}_{B_{R(t)}}(x)$.
	If $u>0$ outside $U$, the minus in \eqref{E7} \eqref{E8} should be replaced by the plus.

	If $\partial^0_{D^{-1}}TV(u)$ is constant on $B_{R(t)}$ and $\left(\overline{B_{R(t)}}\right)^c$, this property is preserved under the evolution, which leads us to definition of calibrability.
	Note that the value of $\partial^0_{D^{-1}}TV(u)$ on $U$ is determined by $U$ and its signature does not depend on particular value of $u$. We say that $U$ (with negative signature) is \emph{calibrable} if there exists $Z_0 \in L^\infty(U, \mathbb{R}^n)$ such that $\nabla \operatorname{div} Z_0 \in L^2(U, \mathbb{R}^n)$, $Z_0$ satisfies \eqref{E7}, \eqref{E8}, $|Z_0|\leq 1$ a.\,e.\ in $U$ and $\Delta\operatorname{div} Z_0$ is constant on $U$. We call any such $Z_0$ a \emph{calibration} for $U$.
	
	Recall that in the case of the second-order problem, we say that $U$ (with negative signature) is calibrable if there exists $\widetilde{Z}_0 \in L^\infty(U, \mathbb{R}^n)$ satisfying \eqref{E7}, $|\widetilde{Z}_0|\leq 1$ a.\,e.\ in $U$ and $-\operatorname{div}\widetilde{Z}_0$ is a  constant function on $U$. This is formally equivalent to \emph{$-$calibrability} in \cite{BCN}, \cite{ACM}. It can be shown that $\widetilde{Z}_0$ is a calibration for $U$ if and only if
	\[
	\widetilde{Z}_0 \in \argmin \left\{ \|\operatorname{div}Z\|_{L^2(U)} \bigm|
	z\ \text{satisfies \eqref{E7} and}\ |Z|\leq 1 \text{ a.\,e.\ in } U \right\}
	\]
	and $-\operatorname{div}\widetilde{Z}_0$ is a  constant function on $U$, which is the definition of calibrability in \cite{GP}.

Going back to our fourth-order problem, if $Z_0$ is a calibration for $U$, then $w=\operatorname{div}Z_0$ must satisfy
	\begin{align}
		-\Delta w &= \lambda \quad\text{in } U  \label{E9} \\
		w &= -\operatorname{div}\nu \quad\text{on } \partial U \label{E10}
	\end{align}
	with some constant $\lambda$.
	If $U$ is bounded, $\lambda$ is determined by \eqref{E7} since
	\begin{equation} \label{E11}
		\int_U w \,\mathrm{d} \mathcal{L}^n = \int_U\operatorname{div} Z_1 \,\mathrm{d} \mathcal{L}^n
		= \int_{\partial U} Z_1\cdot\nu \,\mathrm{d}\mathcal{H}^{n-1} = -\mathcal{H}^{n-1}(\partial U),
	\end{equation}
	where $\mathcal{H}^{n-1}$ denotes the $n-1$ dimensional Hausdorff and $\mathcal{L}^n$ denotes the Lebesgue measure in $\mathbb{R}^n$.
	Using this fact, in section \ref{SCA} we prove that if $Z_0$ is a calibration for a bounded $U$, then 
	\begin{equation} \label{SZ0}
		Z_0 \in \argmin \left\{ \|\nabla\operatorname{div}Z\|_{L^2(U)} \bigm|
		Z\ \text{satisfies \eqref{E7}, \eqref{E8} and}\ |Z|\leq 1 \text{ a.\,e.\ in } U \right\}. 
	\end{equation}	
	Moreover, we obtain an "explicit" formula for the constant $\lambda$ in terms of the \emph{Saint-Venant problem} in $U$.

	In the radially symmetric setting, it is not difficult to show that $Z_0$ in \eqref{SZ0} can be chosen in the form $z\left(|x|\right)\frac{x}{|x|}$. Indeed, if $Z_0$ is belongs to the set of minimizers \eqref{SZ0}, then its rotational average $\overline{Z_0}$ belongs to \eqref{SZ0} as well, because averaging preserves \eqref{E7}, \eqref{E8} and the inequality $|Z|\leq1$. Since the angular part of $\overline{Z_0}$ does not contribute to the divergence, it is possible to delete this part (Lemma \ref{symmetry}).
 We thus conclude that there is an element of \eqref{SZ0} of form $Z(x)=z(|x|)\frac{x}{|x|}$. Thus, the equation \eqref{E9} can be written as the third-order ODE of the form
	\begin{equation} \label{E12}
		- r^{1-n} \left(r^{n-1} \left(r^{1-n} (r^{n-1}z)' \right)' \right)'=\lambda
	\end{equation}
	since $\operatorname{div}Z=r^{1-n}(r^{n-1}z)'$.
	If $U$ is $B_R$ with negative signature, condition \eqref{E7} implies
	\begin{equation} \label{E13}
		z(R) = -1.
	\end{equation}
	Since $\operatorname{div}Z=z'+(n-1)z/r$, condition \eqref{E8} implies that
	\begin{equation} \label{E14}
		z'(R)= 0.
	\end{equation}
	Solving \eqref{E12} under the assumption that $z$ is smooth near zero under conditions \eqref{E13}, \eqref{E14}, we eventually get a unique solution \eqref{E12}--\eqref{E14} of the form
	\[
	z(r) = \frac12 \left(\frac{r}{R}\right)^3 - \frac32\frac{r}{R}, \quad
	\lambda = -\frac{n(n+2)}{R^3}
	\]
	for all $n\geq1$. It is easy to see that $Z(x)=z(|x|)\frac{x}{|x|}$ satisfies the constraint $|Z|\leq1$ in $B_R$.
	We conclude that all balls are calibrable. More careful argument is necessary, but we are able to discuss calibrability of an annulus as well as a complement of a ball.
	\begin{thm} \label{MAIN1}
		\begin{enumerate}
			\item[(i)] All balls are calibrable for all $n\geq1$.
			\item[(i\hspace{-1.5pt}i)] All complement of balls are calibrable except $n=2$.
			\item[(i\hspace{-1.5pt}i\hspace{-1.5pt}i)] If $n=2$, all complement of balls are not calibrable.
			\item[(i\hspace{-1.5pt}v)] All annuli (with definite signature) are calibrable except in $n=2$.
			\item[(v)] For $n=2$, there is $Q_*>1$ such that an annulus (with definite signature) is calibrable if and only if the ratio of the exterior radius over the interior radius is smaller than or equal to $Q_*$.
			In other words, $A^{R_1}_{R_0}$ is calibrable if and only if $R_1/R_0\leq Q_*$.
		\end{enumerate}
	\end{thm}
	Theorem \ref{MAIN1}(v) is consistent with (i\hspace{-1pt}i\hspace{-1pt}i) since $R_1 \to\infty$ implies $A^{R_1}_{R_0}$ converges to $\overline{B_{R_0}}^c$, a complement of the closure of the ball $B_{R_0}$.
	Note that in the case of an annulus, there is a possibility we take a signature which is different on the exterior boundary $\partial B_{R_1}$ and the interior boundary $\partial B_{R_0}$. We also study such indefinite cases.

	We now calculate an explicit solution of \eqref{E1} starting from $u_0=a_0\mathbf{1}_{B_{R_0}}$.
	We first discuss the case $n\neq2$.
	Since a ball and its complement is calibrable, the solution is of the form
	\begin{equation} \label{E15}
		u(t,x) = a(t) \mathbf{1}_{B_{R(t)}}.
	\end{equation}
	We take the (radial) calibration $Z_{in}$ in $B_{R(t)}$ and $Z_{out}$ in $\mathbb{R}^n\backslash\overline{B_{R(t)}}$ and set
	\begin{equation*}
		Z(x,t) = \left \{
		\begin{array}{ll}
			Z_{in}(x), & x \in B_{R(t)} \\
			Z_{out}(x), & x \in \mathbb{R}^n\backslash\overline{B_{R(t)}}.
		\end{array}
		\right.
	\end{equation*}
	Here $Z_{out}(x)=z_{out}(|x|)\frac{x}{|x|}$ can be calculated as
	\[
	z_{out}(r) = - \frac{n-1}{2} \left(\frac{r}{R}\right)^3
	+ \frac{n-3}{2} \left(\frac{r}{R}\right)^{1-n}
	\]
	while, as we already discussed, $z_{in}$ for $Z_{in}(x)=z_{in}(|x|)\frac{x}{|x|}$ is of the form 
	\[
	z_{in}(r) = \frac{1}{2} \left(\frac{r}{R}\right)^3
	- \frac{3}{2} \frac{r}{R}.
	\]
	This $Z$ satisfies \eqref{E5} and \eqref{E6}, and moreover $\operatorname{div}Z\in D^1_0$ for any $t>0$.
	Moreover, $\operatorname{div}Z$ is continuous across $\partial B_{R(t)}$.
	However, $\nabla\operatorname{div}Z$ may jump across $\partial B_{R(t)}$.
	Actually,
	\[
	-\Delta\operatorname{div}Z = \lambda\mathbf{1}_{B_{R(t)}} + \nu\cdot(\nabla\operatorname{div}Z_{in}-\nabla\operatorname{div}Z_{out})\delta_{\partial B_{R(t)}},
	\]
	where $\delta_\Gamma(\varphi)=\int_\Gamma\varphi\,d\mathcal{H}^{n-1}$ or $\delta_\Gamma=\mathcal{H}^{n-1}\,\,\raisebox{-.127ex}{\reflectbox{\rotatebox[origin=br]{-90}{$\lnot$}}}\,\Gamma$ for a hypersurface $\Gamma$ and $\nu$ is the exterior unit normal of $\partial B_{R(t)}$, i.\,e., $\nu=x/R(t)$.
	Here $\lambda=-\frac{n(n+2)}{R^3}$.
	By a direct calculation, the quantity $\nu\cdot (\nabla\operatorname{div}Z_{in}-\nabla\operatorname{div}Z_{out})=-\frac{n(n-4)}{R^2}$.
	Since $u_t=-\Delta\operatorname{div}Z$, by 
	\[
	\partial_t\left(a\mathbf{1}_{B_R}\right)
	= \frac{\mathrm{d}a}{\mathrm{d}t}\mathbf{1}_{B_R}
	+ a \frac{\mathrm{d}R}{\mathrm{d}t}\delta_{\partial B_R},
	\]
	we conclude that
	\[
	\frac{\mathrm{d}a}{\mathrm{d}t} = - \frac{n(n+2)}{R^3}, \quad
	\frac{\mathrm{d}R}{\mathrm{d}t} = - \frac{n(n-4)}{aR^2}.
	\]
	Since
	\[
	\frac{\mathrm{d}}{\mathrm{d}t} (aR^3) = - n(n+2) - 3n(n-4) = -n(4n-10),
	\]
	an explicit form of a solution is given as
	\[
	a(t) = a_0 \left(1 - \frac{n(4n-10)}{a_0 R^3_0}t\right)^{\frac{n+2}{4n-10}}, \quad
	R(t) = R_0 \left(1 - \frac{n(4n-10)}{a_0 R^3_0}t\right)^{\frac{n-4}{4n-10}}.
	\]

	As noticed earlier, in the case $n=2$, the complement of the disk is not calibrable.
	If $u$ is a radially strictly decreasing function outside $B_R$, we expect $Z_{out}(x)=-x/|x|$ for $|x|>R(t)$.
	In \cite{GKM}, it is proposed that a solution $u$ to \eqref{E1} must satisfy
	\[
	u_t = -\Delta\operatorname{div}Z_{out}.
	\]
	Since $\operatorname{div}Z_{out}=-(n-1)/|x|^2$ and $\nabla\operatorname{div}Z_{out}=\frac{(n-1)x}{|x|^3}$, this implies
	\begin{equation} \label{E16}
		u_t(t,x) = -\frac{(n-1)(n-3)}{|x|^3}, \quad
		x \in \left(\overline{B_{R(t)}}\right)^c = \mathbb{R}^n\backslash\overline{B_{R(t)}}.
	\end{equation}
	In the case $n=2$, $\nabla\operatorname{div}Z_{out}\in L^2\left(\left(\overline{B_{R(t)}}\right)^c\right)$ so $Z_{out}$ is a Cahn-Hoffman vector field.

	If we start with $u_0=a_0\mathbf{1}_{B_{R_0}}$ with $a_0>0$ for $n=2$, the expected form of a solution is
	\begin{equation} \label{E17}
		u(t,x) = a(t)\mathbf{1}_{B_{R(t)}}(x) + \frac{t}{|x|^3}\mathbf{1}_{\overline{B_{R(t)}}^c}(x),
	\end{equation}
	where
	\begin{equation} \label{E18}
		\frac{\mathrm{d}a}{\mathrm{d}t} = -\frac{2\cdot 4}{R^3}, \quad
		\left( a(t) - \frac{t}{R(t)^3} \right) \frac{\mathrm{d}R}{\mathrm{d}t} = \frac{2\cdot 2}{R^2}.
	\end{equation}
	Analyzing this ODE system, we can deduce qualitative properties of the solution. Summing up our results yields
	\begin{thm} \label{MAIN2} 
		Let $u_0=a_0\mathbf{1}_{B_{R_0}}$ with $a_0>0$.
		
		If $n \geq 3$, then the solution $u$ to \eqref{E1} with initial datum $u_0$ is of the form
		\[
		u(t,x) = a(t)\mathbf{1}_{B_R(t)} \quad\text{for}\quad
		t < t_* = a_0 R^3_0\bigm/\left(n(4n-10)\right)
		\]
		and $u(t,x)\equiv0$ for $t\geq t_*$.
		(The time $t_*$ is called the extinction time.)
		Moreover, $a(t)$ is decreasing and $a(t)\to0$ as $t\uparrow t_*$.
		\begin{enumerate}
			\item[(i)] $R(t)$ is increasing and $R(t)\to\infty$ as $t\uparrow t_*$ for $n=3$.
			\item[(i\hspace{-1.5pt}i)] $R(t)=R_0$ for $n=4$.
			\item[(i\hspace{-1.5pt}i\hspace{-1.5pt}i)] $R(t)$ is decreasing and $R(t)\to0$ as $t\uparrow t_*$ for $n\geq5$.
		\end{enumerate}
		
		If $n=2$, then the solution is not a characteristic function for $t>0$.
		It is of the form \eqref{E16} and moves by \eqref{E18}.
		In particular, there is no extinction time, $R(t)$ is increasing and $a(t)$ is decreasing.
		Moreover, $R(t)\to\infty$ and $a(t)\to0$ as $t\to\infty$.
		The gap $a(t)-\frac{t}{R(t)^3}$ is always positive. 
		
		If $n=1$, then the solution is of the form $u(t,x)=a(t)\mathbf{1}_{B_R(t)}$ for $t>0$.
		Moreover, $R(t)$ is increasing and $a(t)$ is decreasing with $R(t)\to\infty$ and $a(t)\to0$ as $t\to\infty$.
	\end{thm}
	We note that the infinite extinction time observed in $n \leq 2$ is related to the fact that $0$ is not an element of the affine space $u_0+D^{-1}$ where the flow lives if $\int u_0 \neq 0$.
	In \cite{GK}, finite time extinction for solution to \eqref{E1} is proved in a periodic setting for average zero initial data when the space dimension $n\leq4$
	Our result is unrelated to their result because we consider \eqref{E1} in $\mathbb{R}^n$.

	The formula \eqref{E16} does not give a solution to \eqref{E1} when $n\geq4$ since $\nabla\operatorname{div}Z_{out}$ does not belong to $L^2\left(\left(\overline{B_{R(t)}}\right)^c\right)$.
	In the case $n=3$, this formula is consistent with our definition.
	If we consider $u_0$ strictly radially decreasing for $|x|>R_0$ and $u_0(x)=u_*$ for $|x|\leq R_0$, then $u_0$ does not belong to the domain of $\partial_{D^{-1}}TV$ for $n\geq4$.
	In other words, there is no Cahn-Hoffman vector field.

	These results contrast with the second-order total variation flow
	\[u_t=\operatorname{div}\left(\nabla u/|\nabla u|\right).\]
	In the second-order problem, a ball and an annulus are always calibrable with their complements, see e.\,g.\ \cite{ACM} or \cite[Section 5]{GP}. Furthermore, $u_t(t,\cdot)$ is a locally integrable function without singular part for $t>0$ . Thus, for example, the solution starting from $u_0=a_0\mathbf{1}_{B_{R_0}}$ ($a_0>0$) must be $u(t,x)=a(t)\mathbf{1}_{B_{R_0}}$ with $a(t)=-\lambda t+a_0$, where $\lambda$ is the Cheeger ratio, i.\,e.\ $\lambda=\mathcal{H}^{n-1}(\partial B_{R_0})/\mathcal{L}^n(B_{R_0})$. In particular, the extinction time $t_*$ equals $t_*=a_0/\lambda$.

	We conclude this paper by deriving a system of ODEs prescribing the solution in the case when the initial datum is a piecewise constant, radially symmetric function, which we call a stack. To be precise, we say that $w \in E^{-1}$ is a stack if it is of the form 
	\[w = a^0 \mathbf{1}_{B_{R^0}} + a^1 \mathbf{1}_{A_{R^0}^{R^1}} +\ldots + a^{N-1} \mathbf{1}_{A_{R^{N-2}}^{R^{N-1}}} + a^N \mathbf{1}_{\mathbb{R}^n \setminus B_{R^{N-1}}},\]
	$0<R^0<R^1<\ldots<R^{N-1}$, $a^k \in \mathbb{R}$. In particular, we obtain 
	\begin{thm} \label{MAIN3} 
		Let $n \neq2$ and let $u_0$ be a stack. If $u$ is the solution to \eqref{E1}, then $u(t,\cdot)$ is a stack for $t>0$. 
	\end{thm}   
	\noindent In the case $n = 2$, this result is no longer true, as evidenced by Theorem \ref{MAIN2}. However the solution can still be prescribed by a finite system of ODEs. 

	A total variation flow type equation
	\begin{equation} \label{E19}
		w_t = -\Delta \left( \operatorname{div}\left(\nabla w /|\nabla w|\right)+\beta\operatorname{div}\left(\nabla w|\nabla w|\right) \right)
	\end{equation} 
	was introduced by \cite{Sp} to describe the height of crystal surface moved by relaxation dynamics below the roughening temperature, where $\beta>0$.
	For this equation, characterization of the subdifferential of the corresponding energy was given by Y.\ Kashima in periodic setting \cite{Ka1}, \cite{Ka2} and under Dirichlet condition on a bounded domain \cite{Ka2}.
	The speed of a facet (a flat part of the graph) is calculated for $n=1$ in \cite{Ka1} and for a ball with the Dirichlet condition under radial symmetry \cite{Ka2}.
	Different from the second-order problem, the speed of a facet is determined not only by the shape of facet.
	Also it has been already observed in \cite{Ka1}, that the minimal section may have a delta part although the behavior of the corresponding solution was not studied there.
	A numerical computation was given in \cite{KV}.
	The equation \eqref{E19} was derived as a continuum limit of models describing motion of steps on crystal surface as discussed in \cite{Od}, where numerical simulation was given; see also \cite{Koh}.

	In \cite{CRCT}, a crystalline diffusion flow was proposed and calculated numerically.
	In a special case, it is of the form $w_t=-\partial^2_x\left(W'(w_x)\right)$, where $W$ is a piecewise linear convex function, when the curve is given as the graph of a function.
	This equation was analyzed in \cite{GG2} in a class of piecewise linear (in space) solutions. 

	Fourth-order equations of type \eqref{E1} were proposed for image denoising as an improvement over the second-order total variation flow.
	For example, the equation
	\[
	w_t = -\Delta\operatorname{div}\left( \nabla w/|\nabla w| \right) + \lambda(f-w),
	\]
	where $f$ is an original image which is given and $\lambda>0$, corresponds to the Osher-Sol\'e-Vese model \cite{OSV}.
	The well-posedness of this equation was proved by using the Galerkin method by \cite{EIS}.
	
	For \eqref{E1}, an extinction time estimate was given in \cite{GK} for $n=1,2,3,4$ in the periodic setting. It was extended to the Dirichlet problem in a bounded domain by \cite{GKM}.
	In the review paper \cite{GG}, it was proved that the solution $u$ of \eqref{E1} in $n=1$ may become discontinuous instantaneously even if the initial datum is Lipschitz continuous, because the speed may have a delta part.
	
	There are a few numerical studies for \eqref{E1} in the periodic setting. A duality-based numerical scheme which applies the forward-backward splitting has been proposed in \cite{GMR}. A split Bregman method was adjusted to \eqref{E1} and also \eqref{E19} in \cite{GU}. In these methods, the singularity of the equation at $\nabla u=0$ is not regularized. However, all above studies deal with either periodic, Dirichlet or Neumann boundary condition for a bounded domain. It has never been rigorously studied in $\mathbb{R}^n$, although in \cite{GKM} there are some preliminary calculations for radial solution in $\mathbb{R}^n$.

	This paper is organized as follows.
	In Section \ref{S2}, we discuss basic properties of the total variation on $D^{-1}$, notably we show strict density of $C^\infty_{c,av}$.
	In Section \ref{S3}, we give a rigorous definition of a solution to \eqref{E1} and obtain a verifiable characterization of solutions.
	In Section \ref{SSD}, we extend the results of the previous section to include initial data with non-zero average in $n=1,2$.
	In Section \ref{SCA}, we introduce the notion of calibrability.
	In Section \ref{examples}, we discuss calibrability of rotationally symmetric sets in $\mathbb{R}^n$. 
	In Section \ref{SDS}, we study solutions emanating from piecewise constant, radially symmetric data.

\section{The total variation functional on $D^{-1}$} \label{S2} 
In this section, we give a rigorous definition of the total variation $TV$ on $D^{-1}$ and relate it to the usual total variation defined on $L^1_{loc}$. The main tool that we use here as well as in the following section is an approximation lemma, which for a given $w \in D^{-1}$ produces a sequence of nice functions $w_k \in D^{-1}$ that converges to $w$ in $D^{-1}$ and $TV(w_k) \to TV(w)$. 

Let us denote 
\[X_1 = \left\{ \psi \in C^\infty_c(\mathbb{R}^n, \mathbb{R}^n), \ \|\psi\|_{L^\infty(\mathbb{R}^n, \mathbb{R}^n)} \leq 1\right\}.\]
We define $TV \colon D^{-1}(\mathbb{R}^n) \to [0, \infty]$ by 
\[TV(u) = \sup_{\psi \in X_1} \langle u, \operatorname{div}\, \psi\rangle.  \]
Let us compare this definition with the usual total variation, which we denote here by $\overline{TV}\colon L^1_{loc}(\mathbb{R}^n) \to [0, \infty]$, defined by 
\[\overline{TV}(u) = \sup_{\psi \in X_1} \int_{\mathbb{R}^n} u \,\operatorname{div}\, \psi.\]
First of all, as in the case $\overline{TV}$, we easily check that $TV$ is lower semicontinuous with respect to the weak-* (and, a fortiori, strong) convergence in $D^{-1}(\mathbb{R}^n)$. Indeed, if $v_k \mathrel{\ensurestackMath{\stackon[0pt]{\rightharpoonup}{\scriptstyle\ast}}} v$ in $D^{-1}(\mathbb{R}^n)$,   
\begin{equation*}TV(v) = \sup_{\psi \in X_1} \{\langle v, \operatorname{div}\,\psi \rangle\} = \sup_{\psi \in X_1} \liminf_{k\to \infty} \{\langle v_k, \operatorname{div}\,\psi \rangle\}  \leq  \liminf_{k\to \infty} \sup_{\psi \in X_1} \{\langle v_k, \operatorname{div}\,\psi \rangle\} = \liminf_{k\to \infty} TV(v_k).
\end{equation*}  
In fact, we have
\begin{lem} \label{tvxtv} 
	We have $D(TV) \subset L^1_{loc}$, and so $D(TV) \subset D(\overline{TV})$ with $TV$ and $\overline{TV}$ coinciding on $D(TV)$.   
	In particular, if $n \geq 2$, $D(TV) \subset L^{1^*}(\mathbb{R}^n)$. If $n=1$, \[D(TV) \subset L^\infty_0(\mathbb{R})  
	= \left\{w \in L^\infty(\mathbb{R}) \colon \esslim\limits_{x \to \pm \infty} w(x) = 0 \right\}.\]
\end{lem}
The proof of this fact is a consequence of the lemma below and we postpone it.    

\begin{lem} \label{approx}
	For any $w \in D^{-1}(\mathbb{R}^n)$ there exists a sequence $w_k \in C^\infty_{c,av}(\mathbb{R}^n)$ 	such that 
	\[w_k \to w \text{ in } D^{-1}(\mathbb{R}^n)\]
	and
	\[TV(w_k) \to TV(w).\]
\end{lem} 

To prove it, we will use a special choice of cut-off function and associated variant of the Sobolev-Poincar\'e inequality. For $R>0$, let us denote by $\vartheta_R$ the element of minimal norm in $D^1_0(\mathbb{R}^n)$ among those $w \in D^1_0(\mathbb{R}^n)$ that satisfy $w(x) = 1$ if $|x|\leq \frac{R}{2}$, $w(x) = 0$ if $|x| \geq R$. It is an easy exercise to show that for $\frac{R}{2} \leq |x| \leq R$
\begin{equation*}\label{vartheta_form}
	\vartheta_R(x) = \left(2^{n-2}-1\right)^{-1}\left(\left(\frac{|x|}{R}\right)^{2-n} -1 \right) \text{ if } n \neq 2, \quad \vartheta_R(x) = \frac{\log\frac{R}{|x|}}{\log 2} \text{ if } n = 2.
\end{equation*}  
In either case,
\begin{equation}\label{nabla_theta}
	\nabla \vartheta_R(x) = C_n \frac{|x|^{-n}x}{R^{2-n}} \text{ if } \frac{R}{2} \leq |x| \leq R.  
\end{equation} 
\begin{lem}\label{poincare} 
	If $p \in [1,n[$ and $q \in [1, p^*]$, then for all $w \in C^1(\mathbb{R}^n)$, $R>0$ there holds 
	\begin{equation} \label{poincare_ineq}\left\| w - \frac{\int \vartheta_R w}{\int \vartheta_R}\right\|_{L^{q}(B_R)} \leq C R^{1+\frac{n}{q} - \frac{n}{p}} \|\nabla w\|_{L^p(B_R)} 
	\end{equation} 
	with $C=C(n,p)$ and 
	\begin{equation} \label{vartheta_norm}
		\|\nabla \vartheta_R\|_{L^p(\mathbb{R}^n)} =  C R^{-\frac{1}{p}(n-1)(2-p)}
	\end{equation} 
	with a different $C=C(n,p)$.  
\end{lem} 
\begin{proof} 
	Let $v \in C^1(\mathbb{R}^n)$. Following the proof of the standard Poincar\'e inequality by contradiction using Rellich-Kondrachov theorem, we obtain
	\[\left\| v - \frac{\int \vartheta_1 v}{\int \vartheta_1}\right\|_{L^{p}(B_1)} \leq C \|\nabla v\|_{L^p(B_1)}. \]
	Applying the Sobolev inequality in $B_1$ to the function $v - \frac{\int \vartheta_1 v}{\int \vartheta_1}$, we upgrade this to
	\begin{equation}\label{poinc_unit}
		\left\| v - \frac{\int \vartheta_1 v}{\int \vartheta_1}\right\|_{L^{q}(B_1)} \leq C \|\nabla v\|_{L^p(B_1)}
	\end{equation}
	Next, let $v(x) = w(Rx)$ for a given $w \in C^1(\mathbb{R}^n)$. We observe that 
	\[\vartheta_1(x) = \vartheta_R(Rx) \quad \text{for } x \in \mathbb{R}^n\]
	and so, by a change of variables $x = y/R$, 
	\[\int \vartheta_1 = \frac{1}{R^n}\int \vartheta_R, \quad \int \vartheta_1 v = \frac{1}{R^n}\int \vartheta_R w.\]
	Applying the same change of variables to both sides of \eqref{poinc_unit}
	we conclude the proof of \eqref{poincare_ineq}. 
	
	The proof of \eqref{vartheta_norm} is a matter of direct calculation.   
\end{proof}  
Let us now return to the proof of the approximation lemma. 
\begin{proof}[Proof of Lemma \ref{approx}] 
	Given $w \in D^{-1}$, let 
	\[w_{\varepsilon, R} = \left(\varrho_\varepsilon * w - \frac{\int \vartheta_R \varrho_\varepsilon * w}{\int \vartheta_R}\right) \vartheta_R.\]
	Equivalently, for $\varphi \in D_0^1(\mathbb{R}^n)$, 
	\[\langle w_{\varepsilon, R}, \varphi \rangle = \left\langle  w, \varrho_\varepsilon * \left(\left(\varphi - \frac{\int \vartheta_R \varphi}{\int \vartheta_R}\right) \vartheta_R\right)\right\rangle. \]
	Denoting $\widetilde{w} = (-\Delta)^{-1} w$, 
	\begin{multline*}
		\langle w_{\varepsilon, R} - w, \varphi \rangle = \int \nabla \varrho_\varepsilon * \widetilde{w} \cdot \nabla \left(\left(\varphi - \frac{\int \vartheta_R \varphi}{\int \vartheta_R}\right) \vartheta_R - \varphi\right) + \int \left(\nabla \varrho_\varepsilon * \widetilde{w} - \nabla \widetilde{w}\right)\cdot \nabla \varphi\\ = \int \nabla \varrho_\varepsilon * \widetilde{w} \cdot (\vartheta_R - 1) \nabla \varphi  + \int \nabla \varrho_\varepsilon * \widetilde{w} \cdot  \left(\varphi - \frac{\int \vartheta_R \varphi}{\int \vartheta_R}\right) \nabla \vartheta_R + \int \left(\nabla \varrho_\varepsilon * \widetilde{w} - \nabla \widetilde{w}\right)\cdot \nabla\varphi.
	\end{multline*} 
	We estimate the second term on the r.\,h.\,s.\ using the Poincar\'e inequality from Lemma \ref{poincare}, taking into account that the support of the integrand is contained in $\overline{A}_R$, where $A_R=B_R \setminus \overline{B}_{R/2}$,
	\begin{multline*}
		\left|\int \nabla \varrho_\varepsilon * \widetilde{w} \cdot  \left(\varphi - \frac{\int \vartheta_R \varphi}{\int \vartheta_R}\right) \nabla \vartheta_R\right| \leq C \|\nabla \varrho_\varepsilon * \widetilde{w}  \,\mathbf 1_{A_R}\|_{L^2(\mathbb{R}^n)} \left\|\varphi - \frac{\int \vartheta_R \varphi}{\int \vartheta_R}\right\|_{L^2(\mathbb{R}^n)} \|\nabla \vartheta_R\|_{L^\infty(\mathbb{R}^n)} \\ \leq C \|\nabla \varrho_\varepsilon * \widetilde{w}  \,\mathbf 1_{A_R}\|_{L^2(\mathbb{R}^n)} \|\nabla \varphi\|_{L^2(\mathbb{R}^n)}.
	\end{multline*} 
	Thus, 
	\begin{multline*}
		\|w_{\varepsilon, R} - w\|_{D^{-1}(\mathbb{R}^n)} = \sup_{\|\varphi\|_{D^1_0(\mathbb{R}^n) }\leq 1} \langle w_{\varepsilon, R} - w, \varphi \rangle \\ \leq \|(1 - \vartheta_R) \nabla \varrho_\varepsilon * \widetilde{w}\|_{L^2(\mathbb{R}^n)} + C \|\nabla \varrho_\varepsilon * \widetilde{w}  \,\mathbf 1_{A_R}\|_{L^2(\mathbb{R}^n)} + \|\nabla \varrho_\varepsilon * \widetilde{w}  - \nabla \widetilde{w}\|_{L^2(\mathbb{R}^n)}
	\end{multline*}
	and so 
	\begin{equation}\label{D-1_conv}\lim_{\varepsilon\to 0^+} \lim_{R \to \infty} \|w_{\varepsilon, R} - w\|_{D^{-1}(\mathbb{R}^n)} = 0 . 
	\end{equation} 
	
	Next we estimate $TV(w_{\varepsilon, R})$. Due to lower semicontinuity of $TV$, we can assume without loss of generality that $TV(w)<\infty$.  First, we note that $\varrho_\varepsilon * w \in D^{-1}(\mathbb{R}^n) \cap C^\infty(\mathbb{R}^n)$ (\cite{Sch}, \cite{Ru}) and 
	\[\|\varrho_\varepsilon * \psi\|_{L^\infty(\mathbb{R}^n, \mathbb{R}^n)} \leq \|\psi\|_{L^\infty(\mathbb{R}^n, \mathbb{R}^n)}.\]
	Thus, for any $\psi \in L^\infty(\mathbb{R}^n, \mathbb{R}^n)$,
	\[ TV(\varrho_\varepsilon * w) = \sup_{\psi \in X_1} \left\{ \langle w, \operatorname{div}\, \varrho_\varepsilon *\psi\rangle \right\} \leq TV(w).\]
	In particular, this implies that $\nabla \varrho_\varepsilon * w \in L^1(\mathbb{R}^n)$ for $\varepsilon >0$ and $\int |\nabla \varrho_\varepsilon * w| = TV(\varrho_\varepsilon * w) \leq TV(w)$. By Lemma \ref{poincare},
	\begin{multline}\label{TV_bound}
		TV(w_{\varepsilon, R}) = \int|\nabla w_{\varepsilon, R}| \leq  \int\vartheta_R\left|\nabla \varrho_\varepsilon * w \right| +  \int\left|\left(\varrho_\varepsilon * w - \frac{\int \vartheta_R \varrho_\varepsilon * w}{\int \vartheta_R}\right) \nabla \vartheta_R\right| \\ \leq  \int\left|\nabla \varrho_\varepsilon * w \right| + C \left\|\varrho_\varepsilon * w - \frac{\int \vartheta_R \varrho_\varepsilon * w}{\int \vartheta_R}\right\|_{L^{1^*}(\mathbb{R}^n)} \|\nabla \vartheta_R\|_{L^n(\mathbb{R}^n)} \\ \leq \|\nabla \varrho_\varepsilon * w\|_{L^1(\mathbb{R}^n)} + C \|\nabla \varrho_\varepsilon * w\|_{L^1(\mathbb{R}^n)} R^{-\frac{(n-1)(n-2)}{n}} \leq \left(1 + CR^{-\frac{(n-1)(n-2)}{n}}\right) TV(w).
	\end{multline}  
	If $n\geq 3$, together with lower semicontinuity of $TV$, this yields 
	\[\lim_{(\varepsilon,R) \to (0, \infty)} TV(w_{\varepsilon, R})=TV(w).	\]
	Taking into account \eqref{D-1_conv}, by a diagonal procedure we can select sequences $(\varepsilon_k)$, $(R_k)$ such that $w_k := w_{\varepsilon_k, R_k}$ satisfies both requirements in the assertion. On the other hand, if $n=1$ or $n=2$, \eqref{TV_bound} only implies uniform boundedness of $TV(\nabla w_{\varepsilon, R})$. 
	
	Let now $n=2$. Since $w_{\varepsilon, R}$ have compact support, uniform boundedness of $TV(\nabla w_{\varepsilon, R})$ implies uniform bound on $w_{\varepsilon, R}$ in $L^2(\mathbb{R}^n)$ by the Sobolev inequality. As $w_{\varepsilon, R}$ converges to $\varrho_\varepsilon * w$ in $D^{-1}(\mathbb{R}^n)$, we have $\varrho_\varepsilon * w \in L^2(\mathbb{R}^n)$. This allows us to improve \eqref{TV_bound}: 
	\begin{multline}\label{TV_bound2} 
		TV(w_{\varepsilon, R}) = \int|\nabla w_{\varepsilon, R}| \leq  \int\vartheta_R\left|\nabla \varrho_\varepsilon * w \right| +  \int\left|\varrho_\varepsilon * w \nabla \vartheta_R\right| + \frac{\int \vartheta_R \varrho_\varepsilon * w}{\int \vartheta_R} \int|\nabla \vartheta_R| \\ \leq  \|\nabla \varrho_\varepsilon * w \|_{L^1(\mathbb{R}^2)} + C\|\varrho_\varepsilon * w\, \mathbf 1_{A_R}\|_{L^2(\mathbb{R}^2)} \|\nabla \vartheta_R\|_{L^2(\mathbb{R}^2)}  + C\frac{\|\vartheta_R\|_{D_0^1(\mathbb{R}^2)}\| \varrho_\varepsilon * w\|_{D^{-1}(\mathbb{R}^2)}}{\|\vartheta_R\|_{L^1(\mathbb{R}^2)}} \|\nabla\vartheta_R\|_{L^1(\mathbb{R}^2)}  \\ \leq TV(w) + C\|\varrho_\varepsilon * w\, \mathbf 1_{A_R}\|_{L^2(\mathbb{R}^2)} + \frac{C}{R}\| \varrho_\varepsilon * w\|_{D^{-1}(\mathbb{R}^2)}.
	\end{multline}  
	The r.\,h.\,s.\ of \eqref{TV_bound2} converges to $TV(w)$ as $R\to \infty$ and we conclude as before. 
	
	Next, consider $n=1$. In this case, finiteness of $TV(w)$ implies that $\varrho_\varepsilon * w \in L^\infty(\mathbb{R}^n)$ and there exist $g_\varepsilon^\pm \in \mathbb{R}$ such that 
	\[\lim_{x \to \pm \infty} \varrho_\varepsilon * w(x) = g_\varepsilon^\pm. \]
	Now, let $\eta^\pm_R$ be the element of minimal norm in $D_0^1(\mathbb{R})$ under constraints 
	\[\eta^\pm_R(\pm x) = 1 \text{ if } x \in [2R,3R], \quad \eta^\pm_R(\pm x) = 0 \text{ if } x \not \in ]R,4R[ .\]
	(Clearly, $\eta^\pm_R$ is a continuous, piecewise affine function.) We have 
	\[|\langle \varrho_\varepsilon * w, \eta^\pm_R\rangle| \leq \|\varrho_\varepsilon * w\|_{D^{-1}(\mathbb{R})}\|\eta^\pm_R\|_{D_0^1(\mathbb{R})} \to 0 \text{ as } R \to  \infty. \]
	On the other hand, since $\eta^\pm_R$ are compactly supported and $\varrho_\varepsilon * w$ coincides as distribution with a locally integrable function, we can calculate 
	\[\langle \varrho_\varepsilon * w, \eta^\pm_R\rangle = \int \varrho_\varepsilon * w\, \eta^\pm_R \to \infty \cdot g^\pm_\varepsilon \text{ as } R \to \infty,\]
	so $g^\pm_\varepsilon = 0$. Therefore, we can estimate
	\begin{multline}\label{TV_bound1} 
		TV(w_{\varepsilon, R}) = \int|\nabla w_{\varepsilon, R}| \leq  \int\vartheta_R\left|\nabla \varrho_\varepsilon * w \right| +  \int\left|\varrho_\varepsilon * w \nabla \vartheta_R\right| + \frac{\int \vartheta_R \varrho_\varepsilon * w}{\int \vartheta_R} \int|\nabla \vartheta_R| \\ \leq  TV(w) + \frac{2}{R} \int_{A_R} |\varrho_\varepsilon * w| + 2 \frac{\int \vartheta_R \varrho_\varepsilon * w}{\int \vartheta_R}.
	\end{multline}
	Since we have shown that $\varrho_\varepsilon * w(x) \to 0$ as $x \to \pm \infty$, the averages on the r.\,h.\,s.\ converge to $0$ and we conclude as before. 
	
\end{proof} 

As a first application of the approximation lemma, we demonstrate Lemma \ref{tvxtv} announced before. 

\begin{proof}[Proof of Lemma \ref{tvxtv}] Let $w \in D(TV)$ and let  $(w_k) \subset C_{c,av}^\infty(\mathbb{R}^n)$ be the sequence provided by Lemma \ref{approx}. Let first $n>1$. Since $\nabla w_k$ is uniformly bounded in $L^1(\mathbb{R}^n, \mathbb{R}^n)$, by the Sobolev embedding $w_k$ is uniformly bounded in $L^{1^*}(\mathbb{R}^n)$.
 Therefore, $w \in L^{1^*}(\mathbb{R}^n)$. 
	
	In case $n=1$, since $\nabla w_k$ are compactly supported and uniformly bounded in $L^1(\mathbb{R})$, $w_k$ is uniformly bounded in $L^\infty(\mathbb{R})$. From these two bounds it follows that $w \in L^\infty(\mathbb{R})$ and that $\nabla w$ is a finite signed measure on $\mathbb{R}$. In particular, $w$ has essential limits at $\pm \infty$. Reasoning as in the final part of the proof of Lemma \ref{approx}, we show that these limits vanish. 
\end{proof}

\section{Existence and characterization of the flow on $D^{-1}$} \label{S3}

For a gradient flow of a convex functional, there is a general theory initiated by Y.\ K\=omura \cite{Ko} and developed by H.\ Br\'ezis \cite{Br} and others.
 It is summarized as follows.
\begin{prop}[\cite{Br}] \label{4GT}
Let $H$ be a real Hilbert space.
 Let $\mathcal{E}$ be a lower semicontinuous, convex functional on $H$ with values in $]\!-\!\infty,\infty]$.
 Assume that $D(\mathcal{E})$ is dense in $H$.
 Then, for any $u_0\in H$, there exists a unique solution $u\in C\left([0,\infty[\,, H\right)$ which is absolutely continuous in $(\delta,T)$ (for any $\delta<T<\infty$) satisfying
\begin{equation}\label{def_grad_flow}
	\left \{
\begin{array}{l}
	u_t \in -\partial\mathcal{E}(u) \quad\text{for a.\,e.}\quad t > 0 \\
	u(0) = u_0.
\end{array}
\right.
\end{equation}
Moreover,
\[
	\int^t_s \| u_t \|^2_H\, d\tau \leq \mathcal{E}\left(u(s)\right) - \mathcal{E}\left(u(t)\right)
	\quad\text{for all}\quad t \geq s > 0.
\]
If $\mathcal{E}(u_0)<\infty$, then $s=0$ is allowed.
 In particular, $u_t\in L^2(0,\infty;H)$.
\end{prop}

As in \cite{AGS}, this solution satisfies the evolutionary variational inequality
\[
	\frac12 \frac{\mathrm{d}}{\mathrm{d}t} \|u-f\|^2_H \leq \mathcal{E}(f) - \mathcal{E}(u)
	\quad\text{for a.\,e.\ } t>0
\]
for any $f\in H$.
 Indeed, by definition, 
$u_t\in-\partial\mathcal{E}(u)$ is equivalent to saying
\[
	\frac12 \frac{\mathrm{d}}{\mathrm{d}t} \|u-f\|^2_H = (-u_t, f-u)_H \leq \mathcal{E}(f) - \mathcal{E}(u)
\]
for any $f\in H$.
The evolutionary variational inequality is not only an equivalent formulation of the gradient flow $u_t\in-\partial\mathcal{E}(u)$, but also apply to a gradient flow of a metric space by replacing $\|u-f\|_H$ by distance between $u$ and $f$; see \cite{AGS} for the theory.

To be able to actually find solutions to \eqref{def_grad_flow}, we need to characterize the subdifferential of the total variation in the space $D^{-1}=D^{-1}(\mathbb{R}^n)$.
The basic idea of the proof is a duality argument, which has been carried out in the case of $L^2$ subdifferentials.
In the case of $L^2$ setting, the idea goes back to the unpublished note of F.\ Alter and a detailed proof is given in \cite{ACM}.
Let $\mathcal{E}$ be a functional on a real Hilbert space $H$ equipped with an inner product $(\cdot,\cdot)_H$.
The main idea is to characterize the subdifferential $\partial\mathcal{E}$ by the polar $\mathcal{E}^0$ of $\mathcal{E}:H\to[-\infty,\infty]$ which is defined by
\begin{equation*}
	\mathcal{E}^0(v) := \sup\left\{(u,v)_H \bigm| u\in H,\ \mathcal{E}(u)\leq 1 \right\}
	= \sup\left\{(u,v)/\mathcal{E}(u) \bigm| u\in D(\mathcal{E}),\ \mathcal{E}(u)\neq 0 \right\},
\end{equation*}
where $D(\mathcal{E})=\left\{u\in H \bigm| |\mathcal{E}(u)|<\infty\right\}$.
We first recall a lemma \cite[Lemma 1.7]{ACM}.
\begin{lem} \label{ACM} 
	Let $\mathcal{E}$ be convex.
	Assume that $\mathcal{E}$ is positively one-homogeneous, i.\,e.,
	\[
	\mathcal{E}(\lambda u) = \lambda\mathcal{E}(u)
	\]
	for all $\lambda>0$, $u\in H$.
	Then, $v\in\partial\mathcal{E}(u)$ if and only if $\mathcal{E}^0(v)\leq 1$ and $(u,v)_H=\mathcal{E}(u)$.
\end{lem}

\begin{rem} \label{R00} 
	By general theory of convex functionals, we know that
	\[
	(\mathcal{E}^0)^0 = \mathcal{E}
	\]
	if $\mathcal{E}$ is a non-negative, lower semicontinuous, convex, positively one-homogeneous functional \cite[Proposition 1.6]{ACM}.
\end{rem} 

This property is essential for the proof of 
\begin{thm} \label{DC}
	Let $\Psi\colon D^{-1} \to [0, \infty]$ by defined by
	\[
	\Psi(v) = \inf \left\{ \|Z\|_\infty \bigm|
	v = \Delta\operatorname{div}Z,\ Z \in L^\infty(\mathbb{R}^n),\ \operatorname{div}Z \in D^1_0 \right\}.
	\]
	Then $(TV)^0=\Psi$.
\end{thm}
\begin{rem} \label{DC1}
	\begin{enumerate}
		\item[(i)] By definition, $\Psi$ is a convex, lower semi-continuous, positively one-homogeneous function, so $(\Psi^0)^0=\Psi$.
		\item[(i\hspace{-1.5pt}i)] if $\Psi(v)<\infty$, the infimum is attained.
		Theorem \ref{DC} together with Lemma \ref{ACM} implies the following characterization of the subdifferential of $TV$.
	\end{enumerate}
\end{rem}
\begin{thm} \label{Ch}
	An element $v\in D^{-1}$ belongs to $\partial TV(u)$ if and only if there is $Z\in L^\infty(\mathbf{R}^n)$ with $\operatorname{div}Z\in D^1_0$ such that
	\begin{enumerate}
		\item[(i)] $|Z| \leq 1$
		\item[(i\hspace{-1.5pt}i)] $v = \Delta \operatorname{div}Z$
		\item[(i\hspace{-1.5pt}i\hspace{-1.5pt}i)] $-\langle u, \operatorname{div}Z \rangle = TV(u)$.
	\end{enumerate}
\end{thm}
\begin{proof}
	By Lemma \ref{ACM} and Theorem \ref{DC}, 
	\[
	v \in \partial TV(u)\ \iff\ \Psi(v) \leq 1 \text{ and } (v,u)_{D^{-1}} = TV(u).
	\]
	The property $\Psi(v)\leq1$ together with Remark \ref{DC1}(i\hspace{-1pt}i) implies (i), (i\hspace{-1pt}i) and $\operatorname{div}Z\in D^1_0$.
	\[
	(v,u)_{D^{-1}} = \left\langle u,(-\Delta)^{-1}v \right\rangle = -\langle u,\operatorname{div}Z \rangle.
	\]
	It is not difficult to check the converse.
\end{proof}
\begin{proof}[Proof of Theorem \ref{DC}]
	The inequality $TV^0\leq\Psi$:
	
	We take $v\in D^{-1}$ with $\Psi(v)<\infty$.
	By Remark \ref{DC1}(i\hspace{-1pt}i), there is $Z\in L^\infty(\mathbb{R}^n)$ with $v=\Delta\operatorname{div}Z$ with $\operatorname{div}Z\in D^1_0$ such that $\Psi(v)=\|Z\|_\infty$.
	By Lemma \ref{approx}, there is $u_k\in C^\infty_{c,av}$ such that $TV(u_k)\to TV(u)$, $u_k\to u$ in $D^{-1}$.
	We observe that
	\begin{align*}
		(u_k, v)_{D^{-1}} &=  \left\langle u_k,(-\Delta)^{-1}v \right\rangle = -\langle u_k,\operatorname{div}Z \rangle \\
		&= \int_{\mathbb{R}^n} Z \cdot \nabla u_k
		\leq \|Z\|_\infty TV(u_k).
	\end{align*}
	Sending $k\to\infty$, we conclude that
	\[
	(u,v)_{D^{-1}} \leq \|Z\|_\infty \quad\text{for all}\quad u \in D^{-1}
	\quad\text{with}\quad TV(u) \leq 1.
	\]
	By definition of $\Psi$, this implies $TV^0\leq\Psi$.
	
	The inequality $\Psi\leq TV^0$:
	
	By definition,
	\begin{align*}
		TV(u) &= \sup \left\{ \langle u,-\operatorname{div}Z \rangle \bigm| Z \in C^\infty_c(\mathbb{R}^n),\ |Z| \leq 1 \right\} \\
		&= \sup \left\{ \frac{\langle u,-\operatorname{div}Z \rangle}{\|Z\|_\infty} \biggm| Z \in C^\infty_c(\mathbb{R}^n),\ Z \neq 0 \right\}.
	\end{align*}
	Since
	\[
	\langle u, -\operatorname{div}Z \rangle
	= \left\langle u, (-\Delta)^{-1} \Delta\operatorname{div}Z \right\rangle
	= ( u, \Delta\operatorname{div}Z )_{D^{-1}},
	\]
	we observe that
	\begin{align*}
		TV(u) &= \sup \left\{ \frac{( u,\Delta\operatorname{div}Z )_{D^{-1}}}{\|Z\|_\infty} \biggm| Z \in C^\infty_c(\mathbb{R}^n),\ Z \neq 0 \right\} \\
		&\leq \sup \left\{  \frac{( u,\Delta\operatorname{div}Z )_{D^{-1}}}{\Psi(\Delta\operatorname{div}Z)} \biggm| Z \in C^\infty_c(\mathbb{R}^n),\ \Psi(\Delta\operatorname{div}Z) \neq 0 \right\} \\
		&\leq \Psi^0(u). 
	\end{align*}
	This implies that $TV^0\geq(\Psi^0)^0=\Psi$.
\end{proof}

Now that the subdifferential of $TV$ in $D^{-1}$ is  calculated, we are able to justify an explicit definition of a solution proposed in \cite{GKM}.
\begin{thm} \label{4UE}
Assume that $u\in C\left([0,\infty[\,, D^{-1}\right)$.
 Then $u$ is a solution of $u_t\in-\partial_{D^{-1}}TV(u)$ with $u_0=u(0)$ in the sense of Proposition \ref{4GT} if and only if there exists $Z\in L^\infty\bigl(]0,\infty[\times\mathbb{R}^n\bigr)$ satisfying
\[
	\operatorname{div}Z \in L^2 \left(\delta, \infty; D^1_0 (\mathbb{R}^n)\right)
	\quad\text{for any}\quad \delta > 0
\]
such that for a.\,e.\ $t\in]0,\infty[$ there holds
 \[ u_t = -\Delta\operatorname{div}Z \quad\text{in}\quad D^{-1} (\mathbb{R}^n), \]
\[ |Z| \leq 1 \quad \mathcal{L}^n\text{-a.\,e.} \]
and
\[
	\langle u, \operatorname{div}Z \rangle = -TV(u).
\]
(If $TV(u_0)<\infty$, $\delta=0$ is allowed.)
\end{thm}

	The Theorem essentially follows from Theorem \ref{Ch}. We only need to justify that a Cahn-Hoffman vector field $Z$ defined separately for every time instance by Theorem \ref{Ch} can be chosen to be jointly measurable, i.\,e.\ $Z\in L^\infty\bigl(]0,\infty[\times\mathbb{R}^n\bigr)$. As in the second-order case \cite[section 2.4]{ACM}, this can be done by recalling that Bochner functions can be well approximated by piecewise constant functions. Since our situation is slightly different, let us include the argument for completeness. We begin with a lemma which is a version of \cite[Lemma A.8]{ACM} suited to our needs. 
	
	\begin{lem} \label{riemann} 
	Let $v \in L^1_{loc}(]0,\infty[, X)$, where $X$ is a Banach space and let $N \subset ]0,\infty[$ have Lebesgue measure $0$. Then for each $\varepsilon > 0$ there exists a countable family $\mathcal{G}$ of disjoint closed intervals $I_k = \overline{B(t_k, r_k)}$, $k \in \mathbb{N}$, such that $t_k$ is a Lebesgue point for $v$, $t_k \not \in N$, 
\[ \mathcal{L}^1 \left(]0,\infty[\setminus \bigcup_{k \in \mathbb{N}} I_k\right) = 0 \]
	and 
	\[\int_0^\infty \|v - v^\varepsilon \|_X \leq \varepsilon, \quad \text{where } v^\varepsilon (t) = v(t_k) \quad \text{for } t \in I_k, \ k \in \mathbb{N}. \]
	\end{lem} 
	
	\begin{proof}
		Referring e.\,g.\ to \cite[p.\ 140]{Br}, $\mathcal{L}^1$-a.\,e.\ point $t\in ]0, \infty[$ is a \emph{Lebesgue point} for $v$, i.\,e.\ 
		\[\lim_{h \to 0^+} \frac{1}{2h} \int_{t-h}^{t+h} \|v - v(t)\|_X = 0.\]
		Let $A$ be the set of Lebesgue points of $v$ contained in $]0,\infty[\setminus N$ and let us take 
		\[\mathcal{F} = \Bigg\{\overline{B(t,r)}\ \Bigg|\ t \in A,\ r< \min(\varepsilon,t),\ \frac{1}{2r} \int_{t-r}^{t+r} \|v - v(t)\|_X \leq \varepsilon e^{-t} \Bigg\}.\]
		Using Besicovitch covering theorem \cite[Corollary 1 on p.\ 35]{EG} with $U = ]0, \infty[$ and (for example) $\mathrm{d}\mu(t) = e^{-t}\, \mathrm{d}t$, we obtain a candidate for the family $\mathcal{G}$. We check that indeed 
		\[\int_0^\infty \|v - v^\varepsilon\|_X = \sum_{k \in \mathbb{N}} \int_{I_k}\|v- v(t_k)\|_X \leq \varepsilon \sum_{k \in \mathbb{N}} 2r_k \, e^{-t_k} \leq \varepsilon \int_0^\infty e^{-t} \,\mathrm{d}t = \varepsilon.\]
		 
	\end{proof} 

\begin{proof} [Proof of Theorem \ref{4UE}]
	Applying Lemma \ref{riemann} to $u_t$, for each $\varepsilon > 0$ we obtain a partition of $]0,\infty[$ (up to a set of Lebesgue measure $0$) into disjoint closed intervals $I_k = \overline{B(t_k, r_k)}$, $k \in \mathbb{N}$ such that
	\begin{equation} \label{tk_good} 
		u_t(t_k) \in - \partial_{D^{-1}} TV(u(t_k)) \quad \text{for } k \in \mathbb{N}
	\end{equation} 
	and 
		\[\int_0^\infty \|u_t - v^\varepsilon\|_{D^{-1}} \leq \varepsilon, \quad \text{where } v^\varepsilon(t) = u_t(t_k) \quad \text{for } t \in I_k, \ k \in \mathbb{N}.\] 
	By \eqref{tk_good} and Theorem \ref{Ch}, for $k \in \mathbb{Z}$
	 there exist $Z_k \in L^\infty(\mathbb{R}^n)$ such that 
	\[|Z_k| \leq 1,\quad u_t(t_k) = - \Delta\operatorname{div} Z_k, \quad - \langle u(t_k), \operatorname{div} Z_k \rangle = TV(u(t_k)). \]
	Further denoting 
	\[u^\varepsilon(t) = u(t_k), \quad Z^\varepsilon(t) = Z_k \quad \text{for } t \in I_k, \ k \in \mathbb{N},\]
	we have  
	\[|Z^\varepsilon| \leq 1,\quad v^\varepsilon = - \Delta\operatorname{div} Z^\varepsilon, \quad - \langle v^\varepsilon,\operatorname{div} Z^\varepsilon \rangle = TV(u^\varepsilon) \quad \text{a.\,e.\ in } [0, \infty[. \]  
	We immediately deduce that there exists $Z \in L^\infty([0,\infty[\times \mathbb{R}^n)$ with $|Z|\leq 1$ a.\,e.\ and a subsequence $Z^{\varepsilon_j}$ such that $Z^{\varepsilon_j}$ converges to $Z$ weakly-* in $Z \in L^\infty([0,\infty[\times \mathbb{R}^n)$. Moreover, for any $\varphi \in C^\infty_c([0, \infty[\times \mathbb{R}^n)$,  
	\[ \int_0^\infty \langle v^{\varepsilon_j}, \varphi\rangle = \int_0^\infty \int_{\mathbb{R}^n} \nabla \operatorname{div} Z^{\varepsilon_j} \cdot \nabla \varphi  = \int_0^\infty \int_{\mathbb{R}^n}  Z^{\varepsilon_j} \cdot \nabla \Delta \varphi \to \int_0^\infty \int_{\mathbb{R}^n}  Z \cdot \nabla \Delta \varphi, \]
	at least along a subsequence. Since on the other hand 
	\[\int_0^\infty \langle v^\varepsilon, \varphi\rangle \to \int_0^\infty \langle u_t, \varphi\rangle,\]
	we infer that $u_t = - \Delta \operatorname{div} Z$ and in particular $\operatorname{div} Z^{\varepsilon_j} \to \operatorname{div} Z$ in $L^1_{loc}(]0, \infty[, D^1_0)$. Finally, we observe that $u^\varepsilon \to u$ in $L^\infty_{loc}(]0, \infty[, D^{-1})$. Moreover, since $t \to TV(u(t))$ is non-increasing, we have $TV(u^\varepsilon) \to TV(u)$ in $L^1_{loc}(]0, \infty[)$. Therefore, for any $\varphi \in C_c(]0, \infty[)$, 
	\[\int_0^\infty \langle u^{\varepsilon_j},\operatorname{div} Z^{\varepsilon_j}\rangle \varphi \to \int_0^\infty \langle u, \operatorname{div} Z \rangle \varphi, \quad \int_0^\infty TV(u^\varepsilon) \varphi \to \int_0^\infty TV(u) \varphi. \]

\end{proof}

\section{Extension to $E^{-1}$} \label{SSD} 

Unfortunately, in the case $n\leq2$, the characteristic function $\mathbf{1}_A$ of a set $A$ of positive measure is not in $D^{-1}$ since $\int \mathbf{1}_A \neq0$.
 We shall define a new space containing $\mathbf{1}_A$ as follows.
 We take a function $\psi\in L^2(\mathbb{R}^n)$ with compact support such that $\int_{\mathbb{R}^n}\psi=1$.
 We introduce a vector space
\[
	E^{-1}_\psi = \left\{ w+c\psi \bigm|
	w \in D^{-1}(\mathbb{R}^n),\ c \in \mathbb{R} \right\}.
\]
This space is independent of the choice of $\psi$.
 Indeed, let $\psi_i \in L^2(\mathbb{R}^n)$ be compactly supported and $\int \psi_i =1$ ($i=1,2$).
 An element $w+c\psi_1\in E^{-1}_{\psi_1}$ can be rewritten as
\[
	w + c\psi_1 = w + c(\psi_1-\psi_2) + c\psi_2.
\]
The next lemma implies $q=c(\psi_1-\psi_2)\in D^{-1}(\mathbb{R}^n)$ since $\int q =0$.
 We then conclude that $w+c\psi_1\in E^{-1}_{\psi_2}$.
\begin{lem} \label{4DS}
Assume that $n\leq2$.
 A compactly supported function $q\in L^2(\mathbb{R}^n)$ belongs to $D^{-1}(\mathbb{R}^n)$ if and only if $\int_{\mathbb{R}^n}q=0$.
\end{lem}
\begin{proof}
If $q\in D^{-1} \cap L^1$, then
\[
\int q = \langle q, [1] \rangle = 0,
\]
where $[1]$ stands for the element of $D_0^1$ whose representatives are $1$ as well as $0$. 	

Now suppose that a compactly supported function $q\in L^2(\mathbb{R}^n)$ satisfies $\int_{\mathbb{R}^n}q=0$. Given a $[\varphi] \in D_0^1$, by the Poincar\'e inequality, we have for any $R>0$, independently of the representative $\varphi \in D^1$, 
 \[ \left\|\varphi - \frac{1}{|B_R|} \int_{B_R} \varphi\right\|_{L^2(B_R)} \leq C_R \|\nabla \varphi \|_{L^2(B_R)} \ \leq C_R \|[\varphi]\|_{D_0^1}.\]
In particular, $\varphi \in L^2_{loc}$. Taking into account this and the assumption $\int q = 0$, we see that the linear functional 
\begin{equation} \label{q_func_def}
	\langle q, [\varphi] \rangle =  \int q\,\varphi
\end{equation} 
on $D_0^1$ is well defined. Moreover, if $R>0$ is large enough that $\mathrm{supp}\, q \subset B_R$,   
  \[\left|\int q\,\varphi\right| = \left|\int_{B_R} q\left(\varphi - \frac{1}{|B_R|} \int_{B_R} \varphi\right)\right|\leq \|q\|_{L^2(B_R)} \left\|\varphi - \frac{1}{|B_R|} \int_{B_R} \varphi\right\|_{L^2(B_R)} \leq C_R \|q\|_{L^2} \|[\varphi] \|_{D_0^1}.\]
 Thus, the functional defined by \eqref{q_func_def} is bounded, i.\,e.\ $q \in D^{-1}$. 

\end{proof}
\noindent
Since $E^{-1}_\psi$ is independent of the choice of $\psi$, we suppress $\psi$ and denote this space by $E^{-1}$. In case $n \geq 3$, we will use notation $E^{-1} = D^{-1}$. We also denote $E^1_0 = D^1$ if $n \leq 2$, $E^1_0 = D^1_0$ if $n\geq 3$. For $u \in E^{-1}$, $v \in E^1_0$, we denote 
\begin{equation} \label{E_lin_func} 
	\langle u, v \rangle_E = \left\{\begin{array}{l} 
	\langle w, v\rangle \text{ if } n \geq 3,  \\
	\langle w, [v]\rangle + c \int \psi v \text{ if } n \leq 2,   
\end{array}\right.
\end{equation}
where $u = w + c \psi$, $w \in D^{-1}$, $\psi \in L^2_c$, $ \int \psi = 1$. As before, we check that the value of $\langle u, v \rangle_E$ does not depend on the choice of this decomposition. 

We recall that if $n \geq 3$, $E_0^1 = D_0^1$ and $E^{-1} = D^{-1}$ come with a Hilbert space structure. We also define inner products on $E_0^1$, $E^{-1}$ in case $n \leq 2$ by 
\[
(v_1,v_2)_{E^1_0} := ([v_1], [v_2])_{D^1_0} + \int \psi v_1 \int \psi v_2, 
\]
\[
(u_1,u_2)_{E^{-1}} := (w_1,w_2)_{D^{-1}} + c_1c_2 
\]
for $u_i=w_i+c_i\psi$, $w_i\in D^{-1}$, $c_i\in\mathbb{R}$ ($i=1,2$).
This gives an orthogonal decomposition
\[
E^{-1} = D^{-1} \oplus \mathbb{R}.
\]
We note that although the values of those products may depend on the choice of $\psi$, the topologies they induce on $E^1_0$, $E^{-1}$ do not. Formula \eqref{E_lin_func} associates to any $u \in E^{-1}$ a continuous linear functional on $E^1_0$. The resulting mapping is an isometric isomorphism between $E^{-1}$ and the continuous dual to $E^1_0$.

We extend $TV$ onto $E^{-1}$ by defining
\[
	TV(u) := \sup_{\psi \in X_1}  \langle u, -\operatorname{div}\psi \rangle_E .
\]
As usual, we check that $TV$ is a convex, weakly-* (and strongly) lower semicontinuous functional. In particular, for a fixed $g\in E^{-1}$, the functional $w\mapsto TV(w+g)$ is convex and lower semicontinuous on $D^{-1}$. 
We next give a definition of a solution of $u_t\in-\partial TV(u)$ in the space $E^{-1}$.
 It turns out the idea of evolutionary variational inequality is very convenient since it is a flow in an affine space $g+D^{-1}$ for some $g\in E^{-1}$.
\begin{definition} \label{4HS}
Assume that $u_0 \in E^{-1}$.
 We say that $u:[0,\infty[\,\to E^{-1}$ is a solution to
\begin{equation} \label{EG}
	u_t \in -\partial_{D^{-1}} TV(u) 
\end{equation}
in the sense of EVI (evolutionary variational inequality) with initial datum $u_0$ if
\begin{enumerate}
\item[(i)] $u-g$ is absolutely continuous on $[\delta,T]$ (for any $0<\delta<T<\infty$) with values in $D^{-1}$ and continuous up to zero with $u(0) = u_0$ and
\item[(i\hspace{-1.5pt}i)] $u-g$ satisfies the evolutionary variational inequality, i.\,e.
\[
	\frac12 \frac{\mathrm{d}}{\mathrm{d}t} \left\| u(t)-g \right\|^2_{D^{-1}}
	\leq TV(g) - TV \left(u(t)\right) \quad
\]
holds for a.\,e.\ $t>0$, provided that $g\in E^{-1}$ is such that $u_0-g\in D^{-1}$.
\end{enumerate}
\end{definition}
\begin{thm} \label{4UEH} 
For any $u_0\in E^{-1}$, there exists a unique solution $u$ of \eqref{EG} in the sense of EVI. 
	 Moreover, if $u^i$ is the solution to \eqref{EG} in the sense of EVI with $u^i(0)=u^i_0 \in E^{-1}$ for $i=1,2$, then
	\begin{equation} \label{contract} 
	\left\| u^1(t) - u^2(t) \right\|_{D^{-1}}
	\leq \left\| u^1_0 - u^2_0\right\|_{D^{-1}} \quad\text{for all}\quad t\geq 0,
	\end{equation}
	provided that $u^1_0 - u^2_0 \in D^{-1}$.
\end{thm}
\begin{proof}
Uniqueness follows from contractivity \eqref{contract}, which is established by a standard reasoning \cite{AGS}.
We give a short proof for the reader's convenience and for completeness.
 Let $u^i$ ($i=1,2$) be a solution to \eqref{EG} in the sense of EVI with initial datum $u^i_0$ such that $u^1_0 - u^2_0\in D^{-1}$.
 Since $u^i$ are solutions, we also have $u^i(t)-u^i_0\in D^{-1}$ for $t\geq0$ and so $u^1_0 - u^2(t)\in D^{-1}$, $u^2_0 - u^1(t)\in D^{-1}$.
 Thus, EVI yields  
\begin{multline*}
	\frac12\frac{\mathrm{d}}{\mathrm{d}t} \left\| u^1-u^2 \right\|^2_{D^{-1}} = \left(u^1_t, u^1 - u^2\right)_{D^{-1}} + \left(u^2_t, u^2 - u^1\right)_{D^{-1}} \\ =  \left.\frac12\frac{\mathrm{d}}{\mathrm{d}t} \left\| u^1(t)-g \right\|^2_{D^{-1}}\right|_{g = u^2(t)} + \left.\frac12\frac{\mathrm{d}}{\mathrm{d}t} \left\| u^2(t)-g \right\|^2_{D^{-1}}\right|_{g = u^1(t)} 
	\\ \leq TV\left(u^2(t)\right)-TV\left(u^1(t)\right) + TV\left(u^1(t)\right)-TV\left(u^2(t)\right) = 0
\end{multline*}
for a.\,e.\ $t>0$. We conclude that $\left\| u^1(t)-u^2(t) \right\|^2$ is non-increasing, in particular \eqref{contract} holds.

The existence is a bit more involved.
 For $u_0=w_0+g_0\in E^{-1}$ with $w_0\in D^{-1}$, we consider the gradient flow of the form
\begin{equation} \label{EW}
	w_t \in -\partial_{D^{-1}} TV(w+g_0),\quad w(0) = w_0.
\end{equation}
Applying Proposition \ref{4GT}, there is a unique solution $w$ to \eqref{EW} for $w_0\in D^{-1}$.
 This solution satisfies the evolutionary variational inequality
\[
	\frac12\frac{\mathrm{d}}{\mathrm{d}t} \| w-f \|^2_{D^{-1}}
	\leq TV(f+g_0) - TV(w+g_0)
	\quad\text{for a.\,e.\ } t>0
\]
for any $f\in D^{-1}$.
 Setting $u=w+g_0$, $g=f+g_0$, we end up with
\[
	\frac12\frac{\mathrm{d}}{\mathrm{d}t} \| w-g \|^2_{D^{-1}}
	\leq TV(g) - TV\left(u(t)\right).
\]
Since $g$ can be taken arbitrary such that $u_0-g\in D^{-1}$, this shows that $u$ is the solution of \eqref{EG} in the sense of EVI; condition (i) follows easily from Proposition \ref{4GT}.
\end{proof}

It is non-trivial to characterize the subdifferential $\partial_{D^{-1}}TV$. For this purpose, we introduce a mapping $I$ which plays a role analogous to $-\Delta$ in $n \geq 3$.  
\begin{lem} \label{HD}
Let $n \leq 2$. The mapping $I\colon E^1_0 \to E^{-1}$ defined by
\[
	I(f) = (-\Delta) [f] + \left( \int_{\mathbb{R}^n} f\psi \right) \psi
\]
is an isometric isomorphism.
\end{lem}
\begin{proof}
It is clear that $I(f)\in E^{-1}$ and $I$ is linear.
 For given $u=w+c\psi\in E^{-1}$ with $w\in D^{-1}$, $c\in\mathbb{R}$, there is $\overline{f}\in D^1_0$ such that $(-\Delta)\overline{f}=w$. Since a representative $f$ of $\overline{f}$ is determined up to an additive constant, there is a unique representative $f$ such that
\[
	\int_{\mathbb{R}^n} f\psi = c.
\]
Thus, the mapping $I$ is surjective.
 If $I(f)=0$, then $(-\Delta) [f] = 0$ so $[f]=0$. Thus $f$ is a constant. Since $\int f\,\psi=0$, this constant must be zero, so $f=0$. Thus, $I$ is injective. Recalling our definitions of inner products on $E^1_0$, $E^{-1}$, it is easy to check that $I$ is an isometry.
\end{proof}

We have a characterization of the polar of $TV$ in $E^{-1}$ as in Theorem \ref{DC}.
\begin{thm} \label{4HP}
Let $n \leq 2$. Let $\Psi$ be given by
\[
	\Psi(v) = \inf \left\{ \|Z\|_\infty \bigm|
	v = I(-\operatorname{div} Z),\ Z \in L^\infty(\mathbb{R}^n),\ 
	\operatorname{div}Z \in E_0^1 \right\}
\]
for $v\in E^{-1}$.
 Then $(TV)^0=\Psi$.
\end{thm}
Admitting this fact, we are able to give a characterization of the subdifferential.
\begin{thm} \label{4HSS}
Let $n \leq2$. An element $v \in E^{-1}$ belongs to $\partial_{E^{-1}}TV(u)$ if and only if there is $Z\in L^\infty(\mathbb{R}^n)$ with $\operatorname{div}Z\in E_0^1$ such that
\begin{enumerate}
\item[(i)] $|Z| \leq 1$
\item[(i\hspace{-1.5pt}i)] $v = I(-\operatorname{div}Z)$
\item[(i\hspace{-1.5pt}i\hspace{-1.5pt}i)] $\displaystyle -\langle u, \operatorname{div}Z \rangle_E  = TV(u).$
\end{enumerate}
\end{thm}
\begin{proof}[Proof of Theorem \ref{4HSS}]
The proof parallels that of Theorem \ref{Ch}.
 By Lemma \ref{ACM} and Theorem \ref{4HP}
\[
	v \in \partial TV(u)\ \iff\
	\Psi(v) \leq 1 \text{ and } (u,v)_{E^{-1}}=TV(u).
\] 
The properties (i), (i\hspace{-1pt}i) together with $\operatorname{div}Z\in E^1_0$ are equivalent to $\Psi(v)\leq1$.
 Since
\begin{equation} \label{4II}
	(u,v)_{E^{-1}} = \left(w, (-\Delta)[v] \right)_{D^{-1}}
	+ c \int_{\mathbb{R}^n} v\,\psi
	= \left\langle w, [v] \right\rangle + c \int_{\mathbb{R}^n} v\,\psi,
\end{equation}
the Euler equation $(u,v)_{E^{-1}}=TV(u)$ is equivalent to (i\hspace{-1pt}i\hspace{-1pt}i).
\end{proof}

\begin{proof}[Proof of Theorem \ref{4HP}]
The proof parallels that of Theorem \ref{DC}.
 We first prove that
\[
	(u, v)_{E^{-1}} \leq \|Z\|_\infty \quad\text{for all}\quad
	u \in E^{-1} \quad\text{with}\quad TV(u) \leq 1
\]
for $v=I(-\operatorname{div}Z)$.
 This implies $TV^0\leq\Psi$.
 The estimate $(u,v)_{E^{-1}}\leq \|Z\|_\infty$ formally follows from the identity \eqref{4II}.
 Indeed, by \eqref{4II}, we see
\[
	(u, v)_{E^{-1}} = - \langle w, [\operatorname{div}Z]\rangle
	- c \int_{\mathbb{R}^n} \psi \operatorname{div}Z.
\]
If $u$ is in $C^\infty_c({\mathbb{R}^n} )$, then, by this formula, we obtain
\[
	(u, v)_{E^{-1}} = -\int_{\mathbb{R}^n} u\operatorname{div}Z
	= \int_{\mathbb{R}^n} \nabla u \cdot Z \leq \|Z\|_\infty TV(u).
\]
By approximation, as in the proof of Theorem \ref{DC}, we conclude the desired estimate.

The other inequality $\Psi\leq TV^0$ follows from $TV\leq\Psi^0$.
 The proof of $TV\leq\Psi^0$ is parallel to that of Theorem \ref{DC} by replacing $\Delta\operatorname{div}Z$ by $I(-\operatorname{div}Z)$ and the $D^{-1}$ inner product by the $E^{-1}$ inner product, respectively, if one notes the identity \eqref{4II}.
 Since $I$ is an isometry, $\Psi$ is lower semicontinuous, and we conclude that $\Psi=TV^0$ by Remark \ref{R00}.

\end{proof}
We have to be careful, since the $E^{-1}$ gradient flow
\[
	u_t \in -\partial_{E^{-1}} TV(u)
\]
does not correspond to the total variation flow $u_t=(-\Delta)\operatorname{div}\left(\nabla u/|\nabla u|\right)$.
 By Theorem \ref{4HSS}(i\hspace{-1pt}i\hspace{-1pt}i), $Z=\nabla u/|\nabla u|$ if $\nabla u\neq0$.
 Thus the $E^{-1}$ gradient flow is formally of the form
\[
	u_t = (-\Delta)\operatorname{div}\left(\nabla u/|\nabla u|\right)
	+ \psi \int_{\mathbb{R}^n} \psi \operatorname{div}\left(\nabla u/|\nabla u|\right).
\]
To recover the original total variation flow, we consider ``partial'' subdifferential in the direction of $D^{-1}$.
 Let $P$ be the orthogonal projection from $E^{-1}$ to $D^{-1}$.
 Then, by definition,
\[
	\partial_{D^{-1}} TV(w+c\psi) = P \partial_{E^{-1}} TV(u).
\]
The equation
\[
	w_t \in -\partial_{D^{-1}} TV(w+c\psi)
\]
is now formally of the form
\[
	u_t = (-\Delta) \operatorname{div}\left(\nabla u/|\nabla u|\right)
\]
since $c\psi$ is time-independent. Here is a precise statement.
\begin{thm} \label{4HPS} Let $n \leq 2$. 
 Consider the functional $\mathcal{F}:w\mapsto TV(w+c\psi)$ on $D^{-1}$ for a fixed $c\in\mathbb{R}$ and $\psi$.
 Then, $\partial_{D^{-1}}\mathcal{F}(w)=P\partial_{E^{-1}}TV(u)$ for $u=w+c\psi$.
 In particular, an element $v\in D^{-1}$ belongs to $\partial_{D^{-1}}\mathcal{F}(w)$ if and only if there is $Z\in L^\infty(\mathbb{R}^n)$ with $\operatorname{div}Z\in E^1_0$ such that
\begin{enumerate}
\item[(i)] $|Z| \leq 1$,
\item[(i\hspace{-1.5pt}i)] $v = \Delta\operatorname{div}Z$, 
\item[(i\hspace{-1.5pt}i\hspace{-1.5pt}i)] $\displaystyle - \langle u, \operatorname{div}Z \rangle_E =TV(u)$.
\end{enumerate}
\end{thm}
(In case $n\leq 2$, by $\Delta\operatorname{div}Z$ we understand $\Delta[\operatorname{div}Z]$.) This characterization is important to calculate the solution of $u_t=(-\Delta)\operatorname{div}\left(\nabla u/|\nabla u|\right)$ for $n\leq2$ explicitly. In fact, we can recover the characterization of a solution in the sense of EVI analogous to the one in Theorem \ref{4UE}, amounting to Theorem \ref{MAIN0}.
\begin{proof}[Proof of Theorem \ref{MAIN0}] 	
This is almost immediate. However, like in the case of Theorem \ref{4UE}, we need to justify that the vector field $Z$ can be chosen to be jointly measurable with respect to $(t,x)$. We proceed as in the proof of Theorem \ref{4UE}. The difference is, now we need to pass to the limit with 
\[- \langle u^\varepsilon, \operatorname{div} Z^\varepsilon \rangle_E = TV(u^\varepsilon), \]
but convergence of $v^\varepsilon$ in $L^1_{loc}(]0,\infty[, D^{-1})$ gives only $[\operatorname{div} Z^\varepsilon] \to [\operatorname{div} Z] \in L^1_{loc}(]0,\infty[, D^1_0)$. To deal with this problem, let us choose $\psi \in C^1_c(\mathbb{R}^n)$ and let $c \in \mathbb{R}$ be such that $u(t) - c \psi \in D^{-1}$ for $t>0$. Then we also have $u^\varepsilon(t) - c \psi \in D^{-1}$ for $t>0$, $\varepsilon >0$ and 
\[\langle u^\varepsilon, \operatorname{div} Z^\varepsilon \rangle_E = \langle u^\varepsilon - c \psi, [\operatorname{div} Z^\varepsilon] \rangle + c \int_{\mathbb{R}^n} \psi \operatorname{div} Z^\varepsilon = \langle u^\varepsilon - c \psi, [\operatorname{div} Z^\varepsilon] \rangle - c \int_{\mathbb{R}^n} \nabla \psi \cdot Z^\varepsilon. \]
Testing with $\varphi \in C_c(]0, \infty[)$ and using weak-* convergence of $Z^\varepsilon$ in $L^\infty(]0,\infty[\times \mathbb{R}^n)$ we get 
\begin{multline*}\int_0^\infty \langle u^\varepsilon, \operatorname{div} Z^\varepsilon \rangle_E \, \varphi = \int_0^\infty \langle u^\varepsilon - c \psi, [\operatorname{div} Z^\varepsilon] \rangle \varphi - c \int_0^\infty \int_{\mathbb{R}^n} \nabla \psi \cdot Z^\varepsilon \varphi \\ \to \int_0^\infty \langle u - c \psi, [\operatorname{div} Z] \rangle \varphi - c \int_0^\infty \int_{\mathbb{R}^n} \nabla \psi \cdot Z \varphi = \int_0^\infty \langle u, \operatorname{div} Z \rangle_E \, \varphi,
\end{multline*}
at which point we conclude as in the proof of Theorem \ref{4UE}.
\end{proof}

\section{The notion of calibrability} \label{SCA}

We are interested in sets where the speed of solution $u_t$ is spatially constant. The speed is given as minus the minimal section of the subdifferential, i.\,e.
\[
	\partial^0_{D^{-1}} TV(u) := \argmin \left\{ \|v\|_{D^{-1}} \bigm|
	v \in \partial_{D^{-1}} TV(u) \right\}.
\]
Since $\partial_{D^{-1}} TV(u)$ is closed and convex, $\partial^0_{D^{-1}} TV(u)$ is uniquely determined if $\partial_{D^{-1}} TV(u)\neq\emptyset$.
 Since we have characterized the subdifferential, we end up with 
\begin{multline*}
	\partial^0_{D^{-1}} TV(u) = \argmin \bigl\{ \|v\|_{D^{-1}} \bigm|
	v = \Delta \operatorname{div}Z,\ Z \in L^\infty(\mathbb{R}^n, \mathbb{R}^n),\ |Z| \leq 1,\\
	\operatorname{div}Z \in E^1_0(\mathbb{R}^n), \ -\langle u, \operatorname{div}Z \rangle_E = TV(u)\bigr\}.
\end{multline*}
Although the minimizer $v$ is unique, the corresponding $Z$ may not be unique.
 Let $U$ be a smooth open set in $\mathbb{R}^n$.
 We consider a smooth function $u$ such that 
\[
	\overline{U} = \left\{ x \in \mathbb{R}^n \bigm|
	u(x) = 0 \right\}
\]
and $\partial_{D^{-1}} TV(u)\neq\emptyset$.
 Such a closed set is often called a facet.
 Assume further that $\nabla u\neq0$ outside $\overline{U}$.
 Let $Z$ be a vector field satisfying $v=\Delta\operatorname{div}Z$ for $v\in\partial TV(u)$.
 It is easy to see that outside the facet $\overline{U}$,
\[
	Z(x) = \nabla u(x) \bigm/ \left| \nabla u(x) \right| 
\]
by $- \langle u, \operatorname{div}Z \rangle_E =TV(u)$.
 Since $\|v\|_{D^{-1}}=\|\operatorname{div}Z\|_{D^1_0}$, we see that
\[
	\partial^0_{D^{-1}} TV(u) = \argmin \left\{ \|\operatorname{div}Z\|_{D^1_0} \bigm|
	|Z| \leq 1\ \text{in}\ U,\ Z = \nabla u / |\nabla u|\ \text{in}\ \overline{U}^c,\ 
	\operatorname{div}Z \in E^1_0 \right\}.
\]
Since $\operatorname{div}Z$ is locally integrable, the normal trace is well-defined from inside as an element of $L^\infty(\partial U)$ \cite{ACM} and it must agree with that from outside, i.\,e.
\begin{equation} \label{Znuchi}
	\nu \cdot Z = \nu \cdot \nabla u \bigm/ |\nabla u| = \nu \cdot \chi\nu = \chi,
\end{equation} 
where $\nu(x)$ is the exterior unit normal of $\partial U$ and
\begin{equation*}
	\chi(x)= \left \{
\begin{array}{cl}
	1 &\text{if}\ u>0\ \text{outside}\ \overline{U}\ \text{near}\ x \in \partial U, \\
	-1 &\text{otherwise.}
\end{array}
\right.
\end{equation*}
Since $\operatorname{div}Z$ is in $E^1_0$, its trace from inside and outside must agree, i.\,e.,
\[
	\operatorname{div}Z = \operatorname{div} \left(\nabla u/|\nabla u|\right)=\chi \operatorname{div}\nu =: \chi \kappa \quad \text{on }\partial U. 
\]
Note that $\kappa(x)$ is the sum of all principal curvatures, equal to $n-1$ times the (inward) mean curvature, of $\partial U$ at $x$.

Let $Z_0$ be a minimizer corresponding to $v=\partial^0_{D^{-1}} TV(u)$.
 Since the value $Z_0$ outside $\overline{U}$ is always the same, we consider its restriction on $U$ and still denote by $Z_0$.
 Then, $Z_0$ is a minimizer of
\begin{equation} \label{mini_cahn}
	\left\{ \int_U |\nabla\operatorname{div}Z|^2 \bigm|
	|Z| \leq 1\ \text{in}\ U, \
	\nu\cdot Z= \chi\ \text{on}\ \partial U,\ 
	\operatorname{div}Z = \chi \kappa\ \text{on}\ \partial U \right\}.
\end{equation}

Although $\operatorname{div}Z_0 \in D^1_0(\mathbb{R}^n)$ so that $\nabla\operatorname{div}Z_0 \in L^2(\mathbb{R}^n, \mathbb{R}^n)$, the quantity $\nabla\operatorname{div}Z_0$ may jump across $\partial U$.
 Thus $\Delta\operatorname{div}Z_0$ 
may contain singular part which is a driving force to move the facet boundary ``horizontally'' during its evolution under the fourth-order total variation equation as observed in the previous section and earlier in \cite{GG}.
 In the second-order problem, the speed does not contain any singular part so the jump discontinuity does not move.

We are interested in a situation where $\Delta\operatorname{div}Z_0$ is constant over $U$. In the spirit of \cite{LMM}, we call any continuous function $\chi \colon \partial U \to \{-1,1\}$ a \emph{signature} for $U$. 
\begin{definition} \label{CA1}
Let $U$ be a smooth open set in $\mathbb{R}^n$ with signature $\chi$.
 We say that $U$ is \linebreak 
($D^{-1}$-)\emph{calibrable} (with signature $\chi$) if there exists $Z_0$ satisfying the constraint
\begin{equation} \label{CA_CR} 
	|Z_0|\leq 1 \quad \text{on}\quad  U
\end{equation} 
with boundary conditions
\begin{equation} \label{CA_BC} 
	\nu\cdot Z_0= \chi, \quad
	\operatorname{div}Z_0 = \chi \kappa\ \quad\text{on}\quad \partial U,
\end{equation} 
with the property that 
\begin{equation} \label{CA_CT}
	\Delta\operatorname{div}Z_0 \text{ is constant over } U.
\end{equation} 	
 We call any such $Z_0$ a ($D^{-1}$-)\emph{calibration} for $U$ (with signature $\chi$). 
\end{definition}

From the definition of calibration, we easily deduce  
\begin{prop} \label{5CA}
Let $U$ be a smooth bounded domain in $\mathbb{R}^n$.
 Assume that $Z_0$ is a calibration for $U$ with signature $\chi$. Then, $w_0=\operatorname{div}Z_0$ is a solution to the Saint-Venant problem
\begin{numcases}{}
   -\Delta w = \lambda\quad\text{in}\quad U \label{SV1} \\
   w = \chi \kappa \quad\text{on}\quad \partial U \label{SV2}
\end{numcases}
with the constraint
\begin{equation} \label{AV}
	\int_U w \,\mathrm{d} \mathcal{L}^n = \int_{\partial U}\chi \,\mathrm{d} \mathcal{H}^{n-1},
\end{equation}
where $\lambda$ is some constant.
\end{prop}
Appealing to this relation between calibrability and the Saint-Venant problem, we can prove the following 
\begin{thm} \label{SCA2}
	Let $U$ be a smooth bounded domain in $\mathbb{R}^n$. Suppose that $Z^*$ is a calibration for $U$ with signature $\chi$. Then $Z^*$ is a minimizer of \eqref{mini_cahn}. 
\end{thm}
\begin{proof}
	We first note that $w_*=\operatorname{div}Z_*$ must satisfy \eqref{AV}.	We consider the minimization problem for 
	\[e(w) = \int_U |\nabla w|^2\]
	under the Dirichlet condition \eqref{SV2} and the constraint \eqref{AV}.
	Since the problem is strictly convex, there is a unique minimizer $\overline{w}$ in $D^1_0(U)$.
	By Lagrange's multiplier method, $\overline{w}$ must satisfy \eqref{SV1} because of the constraint \eqref{AV}.
	(Actually, a weak solution $\overline{w}$ of \eqref{SV1} is a smooth solution of \eqref{SV1}, \eqref{SV2} by the standard regularity theory of linear elliptic partial differential equations \cite[Chapter 6]{GT}.)
	As we see in Lemma \ref{5SV} below, the constant $\lambda$ is uniquely determined by \eqref{SV1} and \eqref{AV}.
		For $\overline{w}$, there always exists $Z\in C^\infty(\overline{U})$ such that
		\begin{equation} \label{Div}
			\operatorname{div}Z = \overline{w} \quad\text{in}\quad U, \quad
			\nu \cdot Z = \chi \quad\text{on}\quad \partial U.
		\end{equation}
		Indeed, let $p$ be a solution of the Neumann problem
		\[
		\Delta p = \overline{w} \quad\text{in}\quad U, \quad
		\nu \cdot \nabla p = \chi \quad\text{on}\quad \partial U.
		\]
		Such a solution $p$ always exists since $\overline{w}$ satisfies the compatibility condition \eqref{AV} and it is smooth up to $\overline{U}$; see e.\,g.\ \cite{Gal, GT}.
		If we set $Z=\nabla p$, then $Z$ satisfies the desired property \eqref{Div}.
		Thus, the minimum $e(\overline{w})$ of the Dirichlet energy under the constraint \eqref{AV} agrees with
		\[
		\min \left\{ \int_U |\nabla\operatorname{div}Z|^2 \biggm|
		\nu \cdot Z = \chi,\ \operatorname{div}Z=\chi \kappa\ \text{on}\ \partial U \right\}.
		\]
		Since $w_*=\overline{w}$ and $|Z_*|\leq1$, this shows that $Z_*$ is a minimizer of \eqref{mini_cahn}.
\end{proof}
\begin{lem} \label{5SV}
Let $U$ be a smooth bounded domain in $\mathbb{R}^n$.
 Let $w$ solve
\begin{align*}
	-\Delta w &= \lambda \quad{in}\quad U \\
	w &= f \quad{on}\quad \partial U	
\end{align*}
for $\lambda\in\mathbb{R}$, $f\in C(\partial U)$.
 This solution can be written as
\[
	w = \lambda w_{\mathrm{sv}} + h_f
\]
where $w_{\mathrm{sv}}$ solves the Saint-Venant problem
\begin{align} \label{sv1}
	-\Delta w_{\mathrm{sv}} &= 1 \quad{in}\quad U \\
	w_{\mathrm{sv}} &= 0 \quad{on}\quad \partial U	
\end{align}
and $h_f$ is the harmonic extension of $f$ to $U$.
 In particular, if $\int_U w =c$ is given then $\lambda$ is uniquely determined by
\begin{equation} \label{lambda_uniq} 
	\lambda \int_U w_{\mathrm{sv}} + \int_U h_f = c,
\end{equation}
since $w_{\mathrm{sv}}>0$ in $U$.
\end{lem}
\begin{proof}
The decomposition $w=\lambda w_{\mathrm{sv}}+h_f$ is rather clear.
 The property $w_{\mathrm{sv}}>0$ in $U$ follows from the maximum principle \cite{GT}.
\end{proof} 

From Lemma \ref{5SV} we also deduce the following formula for vertical speed of calibrable facets in terms of the solution to the Saint-Venant problem.  
\begin{prop} \label{5SValt}
	Let $U$ be a smooth bounded domain in $\mathbb{R}^n$ and $\lambda \in \mathbb{R}$. If $Z$ is a calibration for $U$ with signature $\chi$ satisfying 
	\[-\Delta \operatorname{div} Z = \lambda,\]
	then 
	\[\lambda \int_U w_{\mathrm{sv}} = \int_{\partial U} \chi\kappa\, \nu \cdot \nabla w_{\mathrm{sv}} + \int_{\partial U} \chi, \]
	where $w_{\mathrm{sv}}$ is the solution to the Saint-Venant problem \eqref{sv1}. 
\end{prop} 
\begin{proof} 
	We recall \eqref{lambda_uniq} and calculate 
	\[\int_U h_f = - \int_U \Delta w_{\mathrm{sv}} h_f = - \int_{\partial U} \nu \cdot \nabla w_{\mathrm{sv}}\, \chi \operatorname{div} \nu + \int_U \nabla w_{\mathrm{sv}} \cdot \nabla h_f, \]
	\[\int_U \nabla w_{\mathrm{sv}} \cdot \nabla h_f = \int_{\partial U}  w_{\mathrm{sv}} \nu \cdot \nabla h_f - \int_U w_{\mathrm{sv}} \Delta \nabla h_f = 0,\]
	\[c = \int_U \operatorname{div} Z = \int_{\partial U} \chi . \]
\end{proof} 

We now compare the definition of calibrability for the second-order problem.
\begin{definition} \label{CA2}
Let $U$ be a smooth open set in $\mathbb{R}^n$ with signature $\chi$.
 We say that $\overline{U}$ is \linebreak 
($L^2$-)\emph{calibrable} if there is $Z_0$ satisfying the constraint $|Z_0|\leq1$ in $U$ and the boundary condition $\nu\cdot Z_0=\chi$ with the property that $\operatorname{div}Z_0$ is a constant over $U$.
\end{definition} 
This definition is slightly weaker than the calibrability used in \cite{L}, where $a(t)\mathbf{1}_U$ is a solution of the total variation flow in $\mathbb{R}^n$ with some function $a(t)$ of $t$; see also \cite{ACM}.
 This requires that $\partial^0_{L^2}TV(u)$ is constant not only on $U$ but also $U^c$.
 Our definition follows from that of \cite{BNP01c}.
\section{Calibrability of rotationally symmetric sets} \label{examples}
\begin{definition} \label{annulus} 
	We say that a Lebesgue measurable subset $U$ (defined up to a set of measure zero) of $\mathbb{R}^n$ is a \emph{generalized annulus} if $U$ is non-empty, open, connected and rotationally symmetric, i.\,e.\ invariant under the linear action of $SO(n)$ on $\mathbb{R}^n$. 
\end{definition} 
	It is easy to see that any generalized annulus is a ball, an annulus, the complement of a ball or the whole space $\mathbb{R}^n$. In other words, any generalized annulus is of form 
	\[A_{R_0}^{R_1} = \{x \in \mathbb{R}^n \colon R_0 < |x| < R_1\} \text{ with } 0 < R_0 < R_1 \leq \infty \quad \text{or}\quad  A_0^R = B_R \text{ with }R>0. \]
	In this section we will settle the question which generalized annuli are calibrable. 
\begin{lem} \label{symmetry} 
	Let $U$ be a generalized annulus. Suppose that $U$ is calibrable with signature $\chi$. Then there exists a calibration $\overline{Z}$ for $(U, \chi)$ of form $\overline{Z}(x) = z(|x|)\frac{x}{|x|}$. 
\end{lem} 
\begin{proof} 
	Let $Z$ be any calibration for $(U, \chi)$. Let $\mu_n$ be the Haar measure on $S\!O(n)$. We define $\overline{Z}$ as the average 
	\[ \overline{Z}(x) = \int L Z(L^{-1} x) \,\mathrm{d} \mu_n(L). \]
	It is an exercise in vector calculus to check that $\overline{Z}$ satisfies boundary conditions \eqref{CA_BC} and that $\Delta \operatorname{div} \overline{Z}$ is a constant (equal to $\Delta \operatorname{div} Z$) on $U$. By convexity, $|\overline{Z}|\leq 1$. 
	Thus $\overline{Z}$ is a calibration for $(U, \chi)$. By definition, it is invariant under rotations, i.\,e.
	\[L\overline{Z}(L^{-1}x) = \overline{Z}(x)\]
	for $L \in S\!O(n)$, $x \in \mathbb{R}^n$. In the case $n=1$ this already shows that $\overline{Z}$ is in the desired form. In higher dimensions, we consider the orthogonal decomposition 
	\[ \overline{Z}(x) = \overline{Z}\vphantom{Z}^\perp(x) + \overline{Z}\vphantom{Z}^T(x) : = \frac{x}{|x|} \otimes \frac{x}{|x|} \, \overline{Z}(x) + \left(I - \frac{x}{|x|} \otimes \frac{x}{|x|}\right) \overline{Z}(x).\] 
	Both $\overline{Z}\vphantom{Z}^\perp$ and $\overline{Z}\vphantom{Z}^T$ are invariant under rotations. In particular, for any given $R>0$, the restriction of $\overline{Z}\vphantom{Z}^T$ to $\mathbb{S}^{n-1}_R$ is an invariant tangent vector field on $\mathbb{S}^{n-1}_R$. Note that any such vector field is smooth. If $n=3$, it follows by the hedgehog uncombability theorem \cite[Proposition 7.15]{Ful} that $\overline{Z}\vphantom{Z}^T \equiv 0$. If $n>3$, any vector field invariant on $\mathbb{S}^{n-1}$ is in particular invariant on a sphere $\mathbb{S}^2$ containing any given point in $\mathbb{S}^{n-1}$, so the same conclusion follows. Thus, we have
	\[\overline{Z}(x) = \overline{Z}\vphantom{Z}^\perp(x) = \frac{x}{|x|} \otimes \frac{x}{|x|} \, \overline{Z}(x) = \frac{x}{|x|} \cdot \overline{Z}(x)\, \frac{x}{|x|}=: \overline{z}(x) \frac{x}{|x|}.\] 
	By rotational invariance, we have $\overline{z}(x) = z(|x|)$, which concludes the proof. 
	 
	We are left with the case $n=2$ in which there exists a one-dimensional space of invariant tangent fields on $\mathbb{S}^{n-1} = \mathbb{S}^1$ spanned by $e^T(x) : = (x_2, - x_1)$. Thus, we have 
	\[\overline{Z}\vphantom{Z}^T(x) = z^T(|x|) e^T(x).\]
	We calculate 
	\[ \operatorname{div} \overline{Z}\vphantom{Z}^T(x) = z^T(|x|) \operatorname{div} e^T(x) + (z^T)'(|x|) \frac{x}{|x|} \cdot e^T(x) = 0. \]
	Thus, we can disregard $\overline{Z}\vphantom{Z}^T$ and choose $\overline{Z}\vphantom{Z}^\perp$ as our calibration, since it satisfies conditions \eqref{CA_CR}-\eqref{CA_CT} (recall that $\overline{Z}\vphantom{Z}^T$ and $\overline{Z}\vphantom{Z}^\perp$ are orthogonal). 
	As before, we see that 
	\[\overline{Z}\vphantom{Z}^\perp(x) = z(|x|) \frac{x}{|x|}.\] 
\end{proof} 
Let $U$ be a generalized annulus. By Lemma \ref{symmetry}, if $U$ is calibrable, then there exists a calibration $Z$ for $U$ of form $Z = z(|x|)\frac{x}{|x|}$. It follows from \eqref{SV1} that $z$ needs to satisfy the ODE
\begin{equation} \label{ode_z}
	- r^{1-n} \left(r^{n-1} \left(r^{1-n} \left(r^{n-1} z\right)'\right)'\right)' = \lambda.
\end{equation} 
The general solution to this ODE is 
\begin{equation} \label{gen_soln_z} 
	z(r) = c_0 r^3 + c_1 r^{3-n} + c_2 r + c_3 r^{1-n} 
\end{equation} 
where $c_0 = - \frac{\lambda}{2n(n+2)}$ if $n\neq 2$ and 
\begin{equation} \label{gen_soln_z2}  
	z(r) = c_0 r^3 + c_1 r \log r  + c_2 r + c_3 r^{-1}. 
\end{equation} 	
where $c_0 = - \frac{\lambda}{16}$ if $n=2$. We will now try to find a calibration for $U$ by solving a suitable boundary value problem for \eqref{ode_z}.  

\subsection{Balls} 
Let $U = B_R(0)$. To focus attention, we choose $\chi = -1$ on $\partial U$. In this case, boundary conditions \eqref{CA_BC} lead to 
\begin{equation} \label{z_ball_bc}  
	z(R) = -1, \quad z'(R)=0. 
\end{equation} 
If $n\geq 2$, in order to satisfy the requirements $|Z|\leq 1$ and $\nabla \operatorname{div} Z$, we need to restrict to $c_1 = c_3 = 0$ in \eqref{gen_soln_z}. We make the same choice also in case $n=1$, as it leads to the right result. Then, applying \eqref{z_ball_bc} in \eqref{gen_soln_z} or \eqref{gen_soln_z2}, we obtain a system of two affine equations for two unknowns $\lambda, c_2$. We solve it obtaining 
\begin{equation} \label{z_ball} 
	z(r) =  \frac{1}{2} \left(\frac{r}{R}\right)^3  -\frac{3}{2} \frac{r}{R},  
\end{equation} 
\begin{equation} \label{lambda_ball} 
	\lambda = - \frac{n(n+2)}{R^3}. 
\end{equation}
We check that $z$ satisfies $|z|\leq 1$ on $[0,R]$, so $Z$ is a calibration for $B_R$. Thus, all balls are calibrable in any dimension.  

\subsection{Complements of balls} 

Let $U = \mathbb{R}^n \setminus B_R$. For consistency with the previous case, we choose $\chi = 1$ on $\partial U$. In this case, boundary conditions \eqref{CA_BC} also lead to 
\begin{equation} \label{z_cball_bc}  
	z(R) = -1, \quad z'(R)=0. 
\end{equation} 
Let us first assume that $n \geq 3$. In order to satisfy the requirement $|Z|\leq 1$, we need to restrict to $\lambda = c_2 = 0$ in \eqref{gen_soln_z}. Again, applying \eqref{z_ball_bc} in \eqref{gen_soln_z} leads to a system of two affine equations for two unknowns $c_1, c_3$. We solve it obtaining 
\begin{equation} \label{z_cball} 
	z(r) =  - \frac{n-1}{2} \left(\frac{r}{R}\right)^{3-n} + \frac{n-3}{2} \left(\frac{r}{R}\right)^{1-n}. 
\end{equation} 
Again, we easily check that $z$ satisfies $|z|\leq 1$ on $[0,R]$, so $Z$ is a calibration for $B_R$. 

In the omitted cases $n=1,2$, requirement $|Z|\leq 1$ implies $\lambda = c_1 = c_2 = 0$ in \eqref{gen_soln_z}. If $n=1$, there exists $z$ of such form satisfying \eqref{z_cball_bc}: $z(r) \equiv -1$, consistently with \eqref{z_cball}. On the other hand, if $n=2$, applying \eqref{z_cball_bc} to \eqref{gen_soln_z} with $\lambda = c_1 = c_2 = 0$ leads to a contradiction. 

Summing up, all complements of balls are calibrable if $n\neq 2$. On the other hand, if $n=2$ all complements of balls turn out not to be calibrable. 

\subsection{Annuli} 

Let now $U = A_{R_0}^{R_1} = B_{R_1}\setminus B_{R_0}$, $0<R_0 < R_1$. In this case $\partial U$ has two connected components, so there exist two distinct choices of signature: constant and non-constant. Let us first consider the former. To focus attention, we choose $\chi \equiv -1$. Then, boundary conditions \eqref{CA_BC} take form
\begin{equation} \label{z_annulus_bc}  
	z(R_0) = 1, \quad z(R_1) = -1, \quad z'(R_0) = z'(R_1) = 0. 
\end{equation} 
Applying \eqref{z_annulus_bc} to \eqref{gen_soln_z} or \eqref{gen_soln_z2} leads to a system of four affine equations with four unknowns. In the case $n \neq 2$, the solution is 
\begin{align} \begin{split} \label{annulus_soln}
	c_0 &= \tfrac{- (R_1-R_0) R_1^n R_0^n \left( (n-1)(n-2) (R_1+R_0)^2 - 2 R_0 R_1\right)+2 R_1^3 R_0^{2 n}-2 R_0^3 R_1^{2 n}}{R_1 R_0 \left(R_1^n R_0^n \left(n^2 \left(R_1^2-R_0^2\right)^2+8 R_1^2 R_0^2\right)-4 R_1^2 R_0^{2 n+2}-4 R_0^2 R_1^{2 n+2}\right)}, \\
	c_1 &= \tfrac{(R_1+R_0) \left(R_1^n \left(2 (n-1) R_1^2 + (n+2) R_1 R_0 - (n+2) R_0^2 \right) + R_0^n \left((n+2)R_1^2 - (n+2)R_1 R_0 - 2(n-1)R_0^2 \right)\right)}{2 \left(n^2-4\right) R_1^3 R_0^3 + 4 R_1^{3-n} R_0^{n+3} + 4 R_1^{n+3} R_0^{3-n} - n^2 R_1^5 R_0 - n^2 R_1 R_0^5},\\
	c_2 &= \tfrac{(R_1-R_0) R_1^n R_0^n \left(6 R_1^2 R_0^2 + (n-1) n R_1^3 R_0 + (n-1) n R_1^4 + (n-1) n R_1 R_0^3 + (n-1) n R_0^4 \right) - 6 R_1^3 R_0^{2 n+2} + 6 R_0^3 R_1^{2 n+2}}{R_1 R_0 \left(R_1^n R_0^n \left(n^2 \left(R_1^2-R_0^2\right)^2+8 R_1^2 R_0^2\right)-4 R_1^2 R_0^{2 n+2}-4 R_0^2 R_1^{2 n+2}\right)},\\
	c_3 &= \tfrac{(R_1 + R_0) \left(R_0^2 R_1^n \left(-2 (n-3) R_1^2 - n R_1 R_0 + n R_0^2\right)- R_1^2 R_0^n \left(n R_1^2 - R_0 (n R_1 + 2(n-3)R_0)\right)\right)}{2 \left(n^2-4\right) R_1^3 R_0^3 + 4 R_1^{3-n} R_0^{n+3} + 4 R_1^{n+3} R_0^{3-n} - n^2 R_1^5 R_0 - n^2 R_1 R_0^5}.
	\end{split} 
\end{align} 
This can be rewritten in a form emphasizing homogeneity: 
\begin{align} \begin{split} 
	c_0 &= \tfrac{- (Q-1) Q^n \left((n-1)(n-2) (Q+1)^2 - 2Q\right)+2 Q^3 - 2  Q^{2 n}}{ Q^{n-2}  \left(n^2 \left(Q^2-1\right)^2+8 Q^2 \right)-4 - 4 Q^{2 n}} R_1^{-3}, \\
	c_1 &= \tfrac{(Q+1) \left(Q^n \left(2 (n-1) Q^2 + (n+2) Q  - (n+2) \right) + \left((n+2)Q^2 - (n+2)Q - 2(n-1) \right)\right)}{2 \left(n^2-4\right) Q^n + 4 + 4 Q^{2n} - n^2 Q^{n+2} - n^2 Q^{n-2} } R_1^{n-3},\\
	c_2 &= \tfrac{(Q-1) Q^n \left(6 Q^2 + (n-1) n Q^3 + (n-1) n Q^4 + (n-1) n Q + (n-1) n \right) - 6 Q^3 + 6 Q^{2 n+2}}{Q^n \left(n^2 \left(Q^2-1\right)^2+8 Q^2 \right)-4 Q^2 -4  Q^{2 n+2}} R_1^{-1},\\
	c_3 &= \tfrac{(Q + 1) \left( Q^n \left(-2 (n-3) Q^2 - n Q + n \right)- Q^2 \left(n Q^2 - (n Q + 2(n-3))\right)\right)}{2 \left(n^2-4\right) Q^{n+2} + 4 Q^{2} + 4 Q^{2n+2} - n^2 Q^{n+4} - n^2 Q^n} R_1^{n-1},
\end{split} 
\end{align} 
where we denoted $Q = R_1/R_0$. We can further simplify it to 
\begin{align} \begin{split} \label{annulus_soln_homo}
		c_0 &= \tfrac{2Q^3(Q^{2n-3}-1) + (Q-1)Q^n ( (n-1)(n-2)(Q+1)^2 - 2Q )}{4(Q^n-1)^2 - n^2(Q^2-1)^2Q^{n-2}} R_1^{-3}, \\
		c_1 &= \tfrac{(Q+1) ( 2(n-1)(Q^{n+2}-1)+(n+2)Q(Q-1)(Q^{n-1}+1) )}{4(Q^n-1)^2 - n^2(Q^2-1)^2Q^{n-2}} R_1^{n-3}, \\
		c_2 &= -\tfrac{6Q(Q^{2n-1}-1) + (Q-1)Q^{n-2} ( 6Q^2 + n(n-1)(1+Q)(1+Q^3) )}{4(Q^n-1)^2 - n^2(Q^2-1)^2Q^{n-2}} R_1^{-1}, \\
		c_3 &= -\tfrac{(Q+1) ( 2(n-3)(Q^n-1)+nQ(Q-1)(Q^{n-3}+1) )}{4(Q^n-1)^2 - n^2(Q^2-1)^2Q^{n-2}} R_1^{n-1}.
	\end{split} 
\end{align}
We need to check whether condition $|Z|\leq 1$ is satisfied. We calculate 
\begin{multline}\label{z''} 
	z''(r) = 6 c_0 r + (n-3)(n-2) c_1 r^{1-n} + n(n-1) c_3 r^{-n-1}\\ = r^{-n-1} (6 c_0 r^{n+2} + (n-3)(n-2) c_1 r^2 + n(n-1) c_3) =: r^{-n-1} w(r).
\end{multline} 
Using the form \eqref{annulus_soln_homo}, we can check that $c_0>0$, $c_1>0$ for all $Q>1$. Therefore, $w$ has at most one zero on the half-line $r>0$. Consequently, $z''$ has at most one zero, so $z$ has at most one inflection point. Taking into account \eqref{z_annulus_bc}, $z$ cannot have a local extremum on $]R_0, R_1[$. Thus, $|z|\leq 1$ on $]R_0, R_1[$ and $Z$ is a valid calibration.  

\begin{figure}
	\centering
	\begin{subfigure}{0.45\textwidth}
		\includegraphics[width=\textwidth]{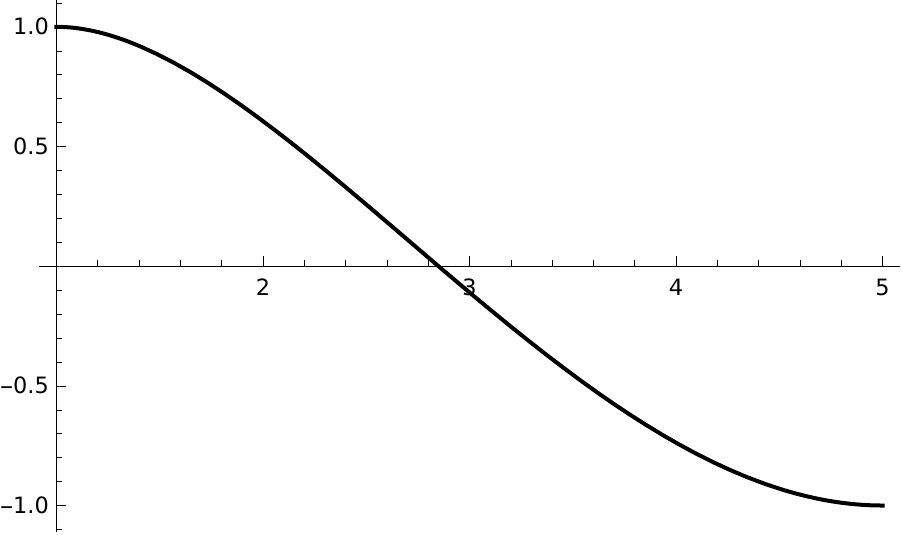}
		\caption{$R_0=1$, $R_1=5$. }
	\end{subfigure}
	\hfill
	\begin{subfigure}{0.45\textwidth}
		\includegraphics[width=\textwidth]{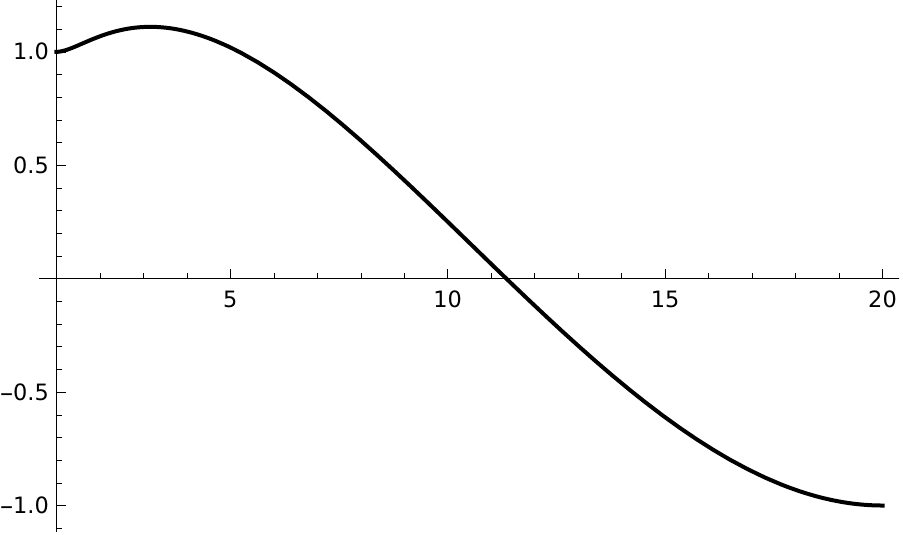}
		\caption{$R_0=1$, $R_1=20$.}
	\end{subfigure}
	
	\caption{Plots of $z$ for an annulus with constant signature for two different values of $Q$ in case $n=2$.}
	\label{plot_z_n2} 
\end{figure}

In the case $n=2$, the solution is 
\begin{align*} 
c_0 &= \tfrac{R_1^2 - R_0^2 + 2 R_1 R_0 \log(R_1/R_0)}{4 R_1 (R_1 - R_0) R_0 (-R_1^2 + R_0^2 + (R_1^2 + R_0^2) \log(R_1/R_0))} \\
c_1 &= \tfrac{-R_1^3 - 3 R_1^2 R_0 - 3 R_1 R_0^2 - R_0^3}{2 R_1 R_0 (-R_1^2 + R_0^2 + (R_1^2 + R_0^2) \log(R_1/R_0))} \\
c_2 &= \tfrac{-3 R_1^4 + 3 R_0^4 + 2 (R_1^4 - R_1^3 R_0 + R_1^2 R_0^2 - 3 R_1 R_0^3) \log(R_1) + 2(3 R_1^3 R_0 - R_1^2 R_0^2 + R_1 R_0^3 - R_0^4) \log(R_0)}{4 R_1 (R_1 - R_0) R_0 (-R_1^2 + R_0^2 + (R_1^2 + R_0^2) \log(R_1/R_0))} \\
c_3 &= \tfrac{R_1 (-3 R_1^2 R_0 + 3 R_0^3 + 2 (R_1^2 R_0 - R_1 R_0^2 + R_0^3) \log(R_1/R_0) )}{4 (R_1 - R_0) (-R_1^2 + R_0^2 + (R_1^2 + R_0^2) \log(R_1/R_0))}
\end{align*} 
which can be rewritten (again, denoting $Q = R_1/R_0$) as 
\begin{align} \begin{split} 
		c_0 &= \tfrac{Q^2\left(Q^2 - 1 + 2 Q \log Q\right)}{4 (Q-1)\left(-Q^2 + 1 +(Q^2 + 1) \log Q\right)} R_1^{-3}, \\
		c_1 &= \tfrac{-(Q+1)^3}{2 \left(-Q^2 + 1 + (Q^2 + 1) \log Q\right)} R_1^{-1}, \\
		c_2 &= \tfrac{-3(Q^4 - 1) - 2(3 Q^3 - Q^2 + Q - 1)\log Q}{4(Q-1)\left(-Q^2 + 1 +(Q^2 +1) \log Q\right)}R_1^{-1} + \tfrac{(Q+1)^3}{2 \left(-Q^2 + 1 + (Q^2 + 1) \log Q\right)}R_1^{-1} \log R_1, \\
		c_3 &= \tfrac{-3 Q^2 + 3 + 2(Q^2 - Q + 1) \log Q}{4 (Q-1) \left( -Q^2 + 1 +(Q^2 + 1) \log Q \right)}R_1.
\end{split} \end{align} 
As before, we calculate the second derivative of $z$: 
\begin{equation*} 
	z''(r) = 6 c_0 r + c_1 r^{-1} +2 c_3 r^{-3} = r^{-3} (6 c_0 r^4 + c_1 r^2 + 2 c_3) =: r^{-3} w(r). 
\end{equation*} 
The polynomial $w$ has at most $2$ positive roots, and so does $z''$. By \eqref{z_annulus_bc}, at least one of them belongs to $]R_0, R_1[$. Furthermore, since $c_0 > 0$ for $Q > 1$, $z''(r)$ is positive for large values of $r$. Taking into account these observations, we deduce that $|z|\leq 1$ on $[R_0, R_1]$ if and only if $z''(R_0) \leq 0$. This inequality is equivalent to 
\[m(Q):=\log Q - \frac{(Q^2 - 1)(2Q - 1)}{Q(Q^2 - 2Q + 3)} \leq 0. \]
We compute
\begin{equation} \label{m_comp} 
	m(1) = 0, \qquad \lim_{Q \to +\infty} m(Q) = +\infty, \qquad m'(Q) = \frac{(Q-3)(Q-1)(Q+1)^3}{Q^2(Q^2- 2Q + 3)^2}.
\end{equation}  
We observe that $m$ has exactly one zero $Q_*$ on $]1, +\infty[$, and $m(Q) \leq 0$ if and only if $Q \leq Q_*$. Therefore, $Z$ is a valid calibration for $A_{R_0}^{R_1}$ with signature $-1$  if and only if $R_1/R_0 \leq Q_*$. By \eqref{m_comp} it is evident that $Q_* > 3$. Numerical computation using Wolfram Mathematica shows that $Q_* \approx 9.7$. Thus, $A_{R_0}^{R_1}$ with constant signature is calibrable if and only if $R_1/R_0 \leq Q_*$. This concludes the proof of Theorem \ref{MAIN1}. 

Now, let us consider non-constant signature. We assume that $\chi=1$ on $\partial B_{R_0}$ and $\chi = -1$ on $\partial B_{R_1}$. This choice leads to 
\begin{equation} \label{z_annulus_bc2}  
	z(R_0) = -1, \quad z(R_1) = -1, \quad z'(R_0) = z'(R_1) = 0. 
\end{equation} 
If $n\neq 2$, the solution to the resulting affine system is 
\begin{align} \begin{split} \label{annulus_soln_ncs}
		c_0 &= \tfrac{-2 R_1^3 R_0^{2 n}-2 R_0^3 R_1^{2 n}+R_0^n R_1^n (R_0+R_1) \left((n-2) (n-1) R_0^2-2 ((n-3) n+1) R_0 R_1+(n-2) (n-1) R_1^2\right)}{R_0 R_1 \left(R_0^n R_1^n \left(n^2 \left(R_0^2-R_1^2\right)^2+8 R_0^2 R_1^2\right)-4 R_1^2 R_0^{2 n+2}-4 R_0^2 R_1^{2 n+2}\right)} \\
		c_1 &=-\tfrac{-3 n R_1 R_0^{n+2}+2 (n-1) R_0^{n+3}+(n+2) R_1^3 R_0^n+(n+2) R_0^3 R_1^n-3 n R_0 R_1^{n+2}+2 (n-1) R_1^{n+3}}{4 R_0^{3-n} R_1^{3-n} \left(R_0^n-R_1^n\right)^2-n^2 R_0 R_1 \left(R_0^2-R_1^2\right)^2} \\
		c_2 &= \tfrac{6 R_1^3 R_0^{2 n+2}+6 R_0^3 R_1^{2 n+2}-R_0^n R_1^n (R_0+R_1) \left((n-1) n R_0^4-(n-1) n R_0^3 R_1-(n-1) n R_0 R_1^3+(n-1) n R_1^4+6 R_0^2 R_1^2\right)}{R_0 R_1 \left(R_0^n R_1^n \left(n^2 \left(R_0^2-R_1^2\right)^2+8 R_0^2 R_1^2\right)-4 R_1^2 R_0^{2 n+2}-4 R_0^2 R_1^{2 n+2}\right)} \\
		c_3 &= \tfrac{(R_0-R_1) \left(R_0^2 R_1^n \left(n (R_0-R_1) (R_0+2 R_1)+6 R_1^2\right)-R_1^2 R_0^n \left(-2 (n-3) R_0^2+n R_0 R_1+n R_1^2\right)\right)}{4 R_0^{3-n} R_1^{3-n} \left(R_0^n-R_1^n\right)^2-n^2 R_0 R_1 \left(R_0^2-R_1^2\right)^2}
\end{split} \end{align} 
which we rewrite as 
\begin{align} \begin{split} \label{annulus_soln_ncs_homo} 
		c_0 &= \tfrac{-2 Q^3-2 Q^{2 n}+ Q^n (1+Q) \left((n-2) (n-1) -2 ((n-3) n+1) Q+(n-2) (n-1) Q^2\right)}{Q^{n-2} \left(n^2 \left(1-Q^2\right)^2+8 Q^2\right)-4 -4 Q^{2 n}} R_1^{-3} \\
		c_1 &=-\tfrac{-3 n Q+2 (n-1) +(n+2) Q^3 +(n+2) Q^n-3 n Q^{n+2}+2 (n-1) Q^{n+3}}{4  \left(1-Q^n\right)^2-n^2 Q^{n-2} \left(1-Q^2\right)^2} R_1^{n-3}\\
		c_2 &= \tfrac{6 Q^3 +6 Q^{2 n+2}-Q^n (1+Q) \left((n-1) n -(n-1) n Q-(n-1) n Q^3+(n-1) n Q^4+6 Q^2\right)}{ Q^n \left(n^2 \left(1-Q^2\right)^2+8 Q^2\right)-4 Q^2 -4 Q^{2 n+2}} R_1^{-1}\\
		c_3 &= \tfrac{(1-Q) \left( Q^n \left(n (1-Q) (1+2 Q)+6 Q^2\right)-Q^2 \left(-2 (n-3) +n Q+n Q^2\right)\right)}{4 Q^2 \left(1-Q^n\right)^2-n^2 Q^n \left(1-Q^2\right)^2} R_1^{n-1}
\end{split} \end{align}
We note that in case $n=1$ the solution reduces to 
\begin{equation*} 
	c_0 = 0, \quad c_1 = 0, \quad c_2 = 0, \quad c_3 =-1,
\end{equation*} 
while in case $n=3$ it reduces to 
\begin{equation*} 
	c_0 = 0, \quad c_1 = -1, \quad c_2 = 0, \quad c_3 =0.
\end{equation*} 
In both of these cases $z$ is constant and we have $\lambda = c_0 = 0$. On the other hand, if $n \geq 4$, we can check that $c_0 > 0$ for $Q >1$. Recalling \eqref{z''}, we observe that $z''$ has at most two zeros on the positive half-line and $z''(r)>0$ for large values of $r$. On the other hand, by \eqref{z_annulus_bc2}, if $z$ has $N$ local extrema on $]R_0,R_1[$, it needs to have at least $N+1$ inflection points. We deduce from these conditions that $z$ has exactly one local maximum and no local minima, and therefore $z \geq -1$ on $]R_0, R_1[$. It remains to check whether $z\leq 1$ on $]R_0, R_1[$. Let now 
\[f(r) = r^{1-n}(r^{n-1}z(r))' = f'(r) + (n-1)\frac{f(r)}{r}.\]
Then, by \eqref{ode_z}, \eqref{z_annulus_bc2}, $f$ is a solution to the second-order elliptic problem
\[\mathcal A f =  \lambda, \quad f(R_0)=-\frac{n-1}{R_0}, \quad f(R_1)=-\frac{n-1}{R_1},\]
where 
\[\mathcal A f = - r^{1-n}(r^{n-1}f'(r))' = - f''(r) - (n-1)\frac{f'(r)}{r}.\]
Since $c_0 \geq 0$, we have $\lambda < 0$ for $Q >1$.  By the classical weak maximum principle \cite[Theorem 3.1.]{GT}, 
\[\max_{[R_0, R_1]}f = \max_{\{R_0, R_1\}}f = -\frac{n-1}{R_1}.\]
Now, if $z$ has a local maximum at $r_0$, then $z'(r_0)=0$, so $f(r_0) = \frac{z(r_0)}{r_0}$. Consequently, 
 \[\frac{z(r_0)}{r_0} \leq -\frac{n-1}{R_1} < 0,\]
 so $z<0$ on $[R_0, R_1]$. Thus, if $n\neq 2$, all annuli with non-constant signature are calibrable. 

\begin{figure}
	\centering
	\begin{subfigure}{0.45\textwidth}
		\includegraphics[width=\textwidth]{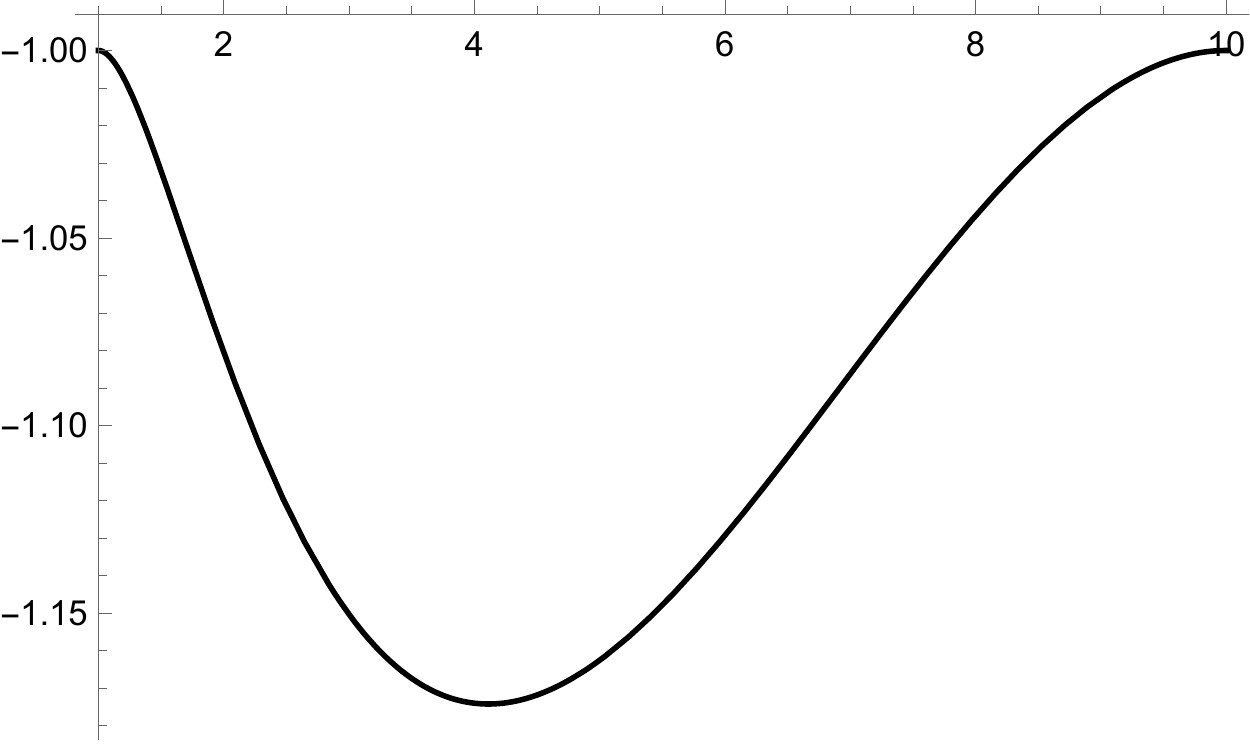}
		\caption{$n=2$. }
	\end{subfigure}
	\hfill
	\begin{subfigure}{0.45\textwidth}
		\includegraphics[width=\textwidth]{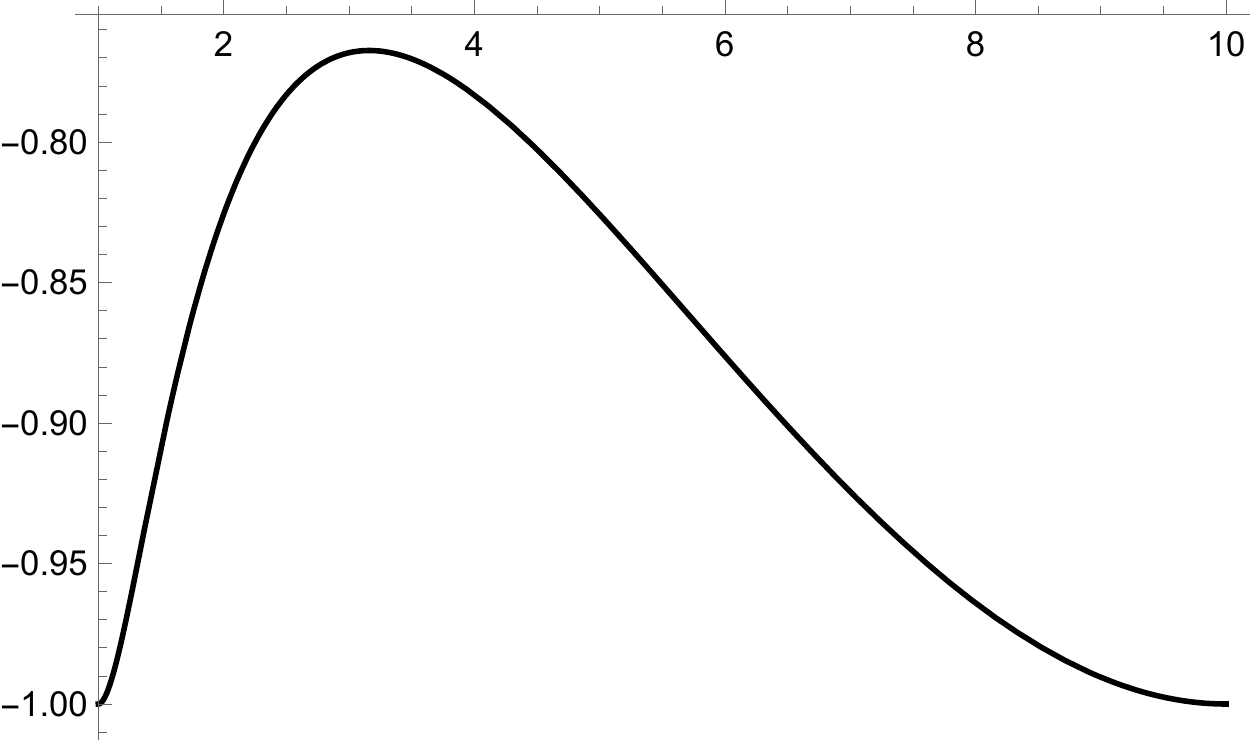}
		\caption{$n=4$.}
	\end{subfigure}
	
	\caption{Plots of $z$ for an annulus with non-constant signature for two values of $n$, with $R_0 = 1$, $R_1 = 10$.}
	\label{plot_z_ncs} 
\end{figure}

We move to the case $n=2$. Now, the solution to the affine system for coefficients of $z$ is 
\begin{align*}\begin{split} 
		c_0 &= \tfrac{-R_0^2 - 2 R_0 R_1 \log (R_1/R_0)+R_1^2}{4 R_0 R_1 (R_0+R_1) \left(-R_0^2+R_1^2- \left(R_0^2+R_1^2\right)\log (R_1/R_0)\right)}, \\
		c_1 &= \tfrac{(R_0-R_1)^3}{2 R_0 R_1 \left(-R_0^2+R_1^2- \left(R_0^2+R_1^2\right)\log (R_1/R_0)\right)},\\ 
		c_2 &= \tfrac{(R_0-R_1) (R_0+R_1) \left(R_0^2+4 R_0 R_1+R_1^2\right)-2 R_0 \left(R_0^3+R_0^2 R_1+R_0 R_1^2+3 R_1^3\right) \log (R_0)+2 R_1 \left(3 R_0^3+R_0^2 R_1+R_0 R_1^2+R_1^3\right) \log (R_1)}{4 R_0 R_1 (R_0+R_1) \left(-R_0^2+R_1^2- \left(R_0^2+R_1^2\right)\log (R_1/R_0)\right)}, \\
		c_3 &= \tfrac{R_0 R_1 \left(2 \left(R_0^2+R_0 R_1+R_1^2\right) \log (R_1/R_0))+3 (R_0-R_1) (R_0+R_1)\right)}{4 (R_0+R_1) \left(-R_0^2+R_1^2- \left(R_0^2+R_1^2\right)\log (R_1/R_0)\right)}
\end{split} \end{align*} 
or equivalently 
\begin{align*}\begin{split} 
		c_0 &= \tfrac{Q^2\left(-1 - 2 Q \log Q+Q^2\right)}{4 (1+Q) \left(-1 + Q^2- \left(1+Q^2\right)\log Q \right)}R_1^{-3}, \\
		c_1 &= \tfrac{(1-Q)^3}{2 \left(-1 + Q^2- \left(1 + Q^2\right)\log Q\right)}R_1^{-1},\\ 
		c_2 &= \tfrac{(1-Q) (1+Q) \left(1 + 4 Q + Q^2\right) + 2 \left(1 + Q + Q^2 + 3 Q^3\right) \log Q}{4 (1 + Q) \left(-1 + Q^2 - \left(1 + Q^2\right)\log Q \right)}R_1^{-1} + \tfrac{ (Q-1)^3}{2 \left(-1 + Q^2 - \left(1 + Q^2\right)\log Q \right)}R_1^{-1}\log R_1, \\
		c_3 &= \tfrac{2 \left(1 + Q + Q^2\right) \log Q +3 (1 - Q)(1 + Q)}{4(1 + Q) \left(-1 + Q^2 - \left(1 + Q^2\right)\log Q\right)}R_1.
\end{split} \end{align*} 
We can check that in this case $c_0 < 0$ for $Q>1$. By the same argument as in the previous case, we show that $z < -1$ in $]R_0, R_1[$, so it does not define a valid calibration. Thus, in the case $n=2$ all annuli with non-constant signature are not calibrable.   
\section{Explicit solutions} \label{SDS}
\subsection{Balls} 
In this section, our goal is to provide explicit description of solutions to \eqref{E1} emanating from the characteristic function of a ball 
\begin{equation}\label{char_ball} 
	u_0 = a_0 \mathbf 1_{B_{R_0}}.
\end{equation} 
In the case of second-order total variation flow, the solutions with initial datum \eqref{char_ball} are known to be of form
\[u(t) = a(t)\mathbf 1_{B_{R_0}} \]
with finite extinction time, i.\,e.\ there exists $t_*>0$ such that $a(t) = 0$ for $t \geq t_*$. In the fourth order case, based on the treatment of case $n=1$ in \cite{GG}, we would expect the solutions to have the form
\begin{equation} \label{soln_form}
	u(t) = a(t)\mathbf 1_{B_{R(t)}},
\end{equation}  
at least until an extinction time beyond which $u(t,\cdot) \equiv 0$. This intuition turns out to be correct in every dimension except $n = 2$. 

Let first $n \geq 3$.
 As we have checked in section \ref{examples}, in this case both balls and complements of balls are calibrable. Thus, as long as the solution is of form \eqref{soln_form} in time instance $t \geq 0$, we expect a valid Cahn-Hoffman vector field $Z$ to be given by 
\begin{equation} \label{Zin_Zout} 
	Z(x) = \left\{\begin{array}{l} 
		Z_{in}(x) \text{ if } |x| \in [0,R[ \\
		Z_{out}(x) \text{ if } |x| > R,
	\end{array}  \right.
\end{equation} 
where $Z_{in}$ is the calibration w constructed for a ball $B_R$ and $Z_{out}$ is the calibration we constructed for the complement of that ball, recall:  
\begin{equation} \label{lambda_soln}
	\lambda = - \frac{n(n+2)}{R^3},  
\end{equation} 
\begin{equation} \label{ode_z_soln} 
	Z_{in}(x) =  \frac{1}{2} \left(\frac{|x|}{R}\right)^3 \frac{x}{|x|}  -\frac{3}{2} \frac{x}{R}, \qquad Z_{out}(x) = - \frac{n-1}{2} \left(\frac{|x|}{R}\right)^{3-n}\frac{x}{|x|} + \frac{n-3}{2} \left(\frac{|x|}{R}\right)^{1-n}\frac{x}{|x|}. 
\end{equation}
We further calculate: 
\begin{equation*}   
	\operatorname{div} Z_{in}(x) = \frac{n+2}{2} \frac{|x|^2}{R^3} - \frac{3n}{2} \frac{1}{R}, \qquad \operatorname{div} Z_{out}(x) = - (n-1) \frac{|x|^{2-n}}{R^{3-n}},
\end{equation*} 
\begin{equation*} 
	\nabla \operatorname{div} Z_{in}(x) = (n+2) \frac{x}{R^3}, \qquad \nabla \operatorname{div} Z_{out}(x) = (n-1)(n-2) \frac{|x|^{-n} x}{R^{3-n}}.
\end{equation*}
It is straightforward to check that $\operatorname{div} Z \in D^1(\mathbb{R}^n) \cap L^{2^*}(\mathbb{R}^n) = D_0^1(\mathbb{R}^n)$. Next, we deduce 
\begin{multline} 
	u_t = -\Delta \operatorname{div} Z \\ = - \Delta \operatorname{div} Z_{in}\, \mathcal{L}^n \,\raisebox{-.127ex}{\reflectbox{\rotatebox[origin=br]{-90}{$\lnot$}}}\,_{B_R} - \Delta \operatorname{div} Z_{out}\, \mathcal{L}^n \,\raisebox{-.127ex}{\reflectbox{\rotatebox[origin=br]{-90}{$\lnot$}}}\,_{\mathbb{R}^n\setminus B_R} + \frac{x}{|x|} \cdot(\nabla \operatorname{div} Z_{in} - \nabla \operatorname{div} Z_{out})\,\mathcal{H}^{n-1} \,\raisebox{-.127ex}{\reflectbox{\rotatebox[origin=br]{-90}{$\lnot$}}}\,_{\partial B_R} \\ = - \frac{n(n+2)}{R^3}\, \mathcal{L}^n \,\raisebox{-.127ex}{\reflectbox{\rotatebox[origin=br]{-90}{$\lnot$}}}\,_{B_R} - \frac{n(n-4)}{R^2}\, \mathcal{H}^{n-1} \,\raisebox{-.127ex}{\reflectbox{\rotatebox[origin=br]{-90}{$\lnot$}}}\,_{\partial B_R}.   
\end{multline}    
Then, using the identity $\frac{\mathrm{d}}{\mathrm{d} t} \int_{\mathbb{R}^n} u = \int_{\mathbb{R}^n} u_t$, we obtain (recall notation \eqref{soln_form})
\begin{equation*} 
	a(t) \mathcal{H}^{n-1}\left(\partial B_{R(t)}\right) \frac{\mathrm{d} R}{\mathrm{d} t} = a(t) \frac{\mathrm{d}}{\mathrm{d} t} \mathcal{L}^n\left(B_{R(t)}\right) = - \frac{n(n-4)}{R^2} \mathcal{H}^{n-1}\left(\partial B_{R(t)}\right).
\end{equation*} 
Summing up, evolution of initial datum \eqref{char_ball} is given by \eqref{soln_form} with $a$, $R$ satisfying 
\begin{equation}\label{ode_aR} 
	\frac{\mathrm{d} a}{\mathrm{d} t} = - \frac{n(n+2)}{R^3}, \qquad \frac{\mathrm{d} R}{\mathrm{d} t} = - \frac{n(n-4)}{R^2 a}.
\end{equation} 
This system can be explicitly solved by noticing that 
\begin{equation*} 
	\frac{\mathrm{d}}{\mathrm{d} t} (a R^3) = -n(n+2) - 3n(n-4) = -n (4n - 10)
\end{equation*} 
and therefore 
\begin{equation*} 
	a R^3 = a_0 R_0^3 - n (4n - 10) t
\end{equation*} 
along trajectories. The solution is 
\begin{equation} \label{final_soln} 
	a(t) = a_0 \left(1 - \frac{n(4n-10)}{a_0 R_0^3} t\right)^{\frac{n+2}{4n-10}}, \qquad R(t) = R_0 \left(1 - \frac{n(4n-10)}{a_0 R_0^3} t\right)^{\frac{n-4}{4n-10}}. 
\end{equation}
We note that the solution satisfies
\[\left(\frac{a}{a_0}\right)^{n-4} = \left(\frac{R}{R_0}\right)^{n+2}\]
along trajectories. (This "first integral" could also have been used to solve the system \eqref{ode_aR}.) Let us point out a few observations concerning the solutions (compare Figure \ref{plot_ball}): 
\begin{itemize} 
	\item the extinction time is equal to $t_* = \frac{a_0 R_0^3}{n(4n-10)}$, 
	\item if $n=3$, $R(t)$ is increasing and $R(t) \to + \infty$ as $t \to t_*^-$, 
	\item if $n=4$, $R(t) = R_0$ is constant, 
	\item in higher dimensions, $R(t)$ is decreasing and $R(t) \to 0$ as $t \to t_*^-$.
\end{itemize} 

\begin{figure}
	\centering
	\begin{subfigure}{0.45\textwidth}
		\includegraphics[width=\textwidth]{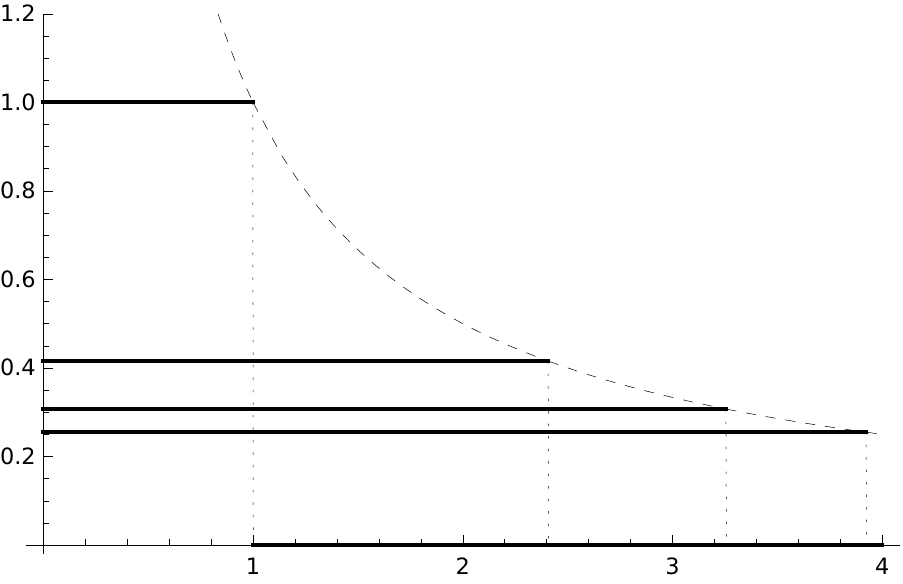}
		\caption{Case $n=1$. Solid lines: plots of $u(t,\cdot)$ for $t=0$, $t=0.8$, $t=1.6$, $t=2.4$. }
	\end{subfigure}
	\hfill
	\begin{subfigure}{0.45\textwidth}
		\includegraphics[width=\textwidth]{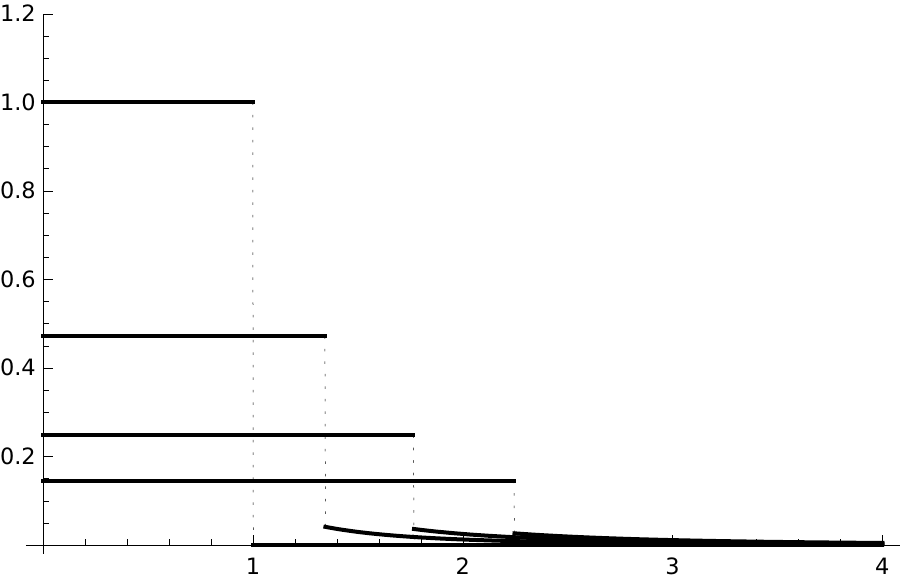}
		\caption{Case $n=2$. Solid lines: plots of $u(t,\cdot)$ for $t=0$, $t=0.1$, $t=0.2$, $t=0.3$.}
	\end{subfigure}
	
	\vspace{25pt}
	
	\begin{subfigure}{0.45\textwidth}
		\includegraphics[width=\textwidth]{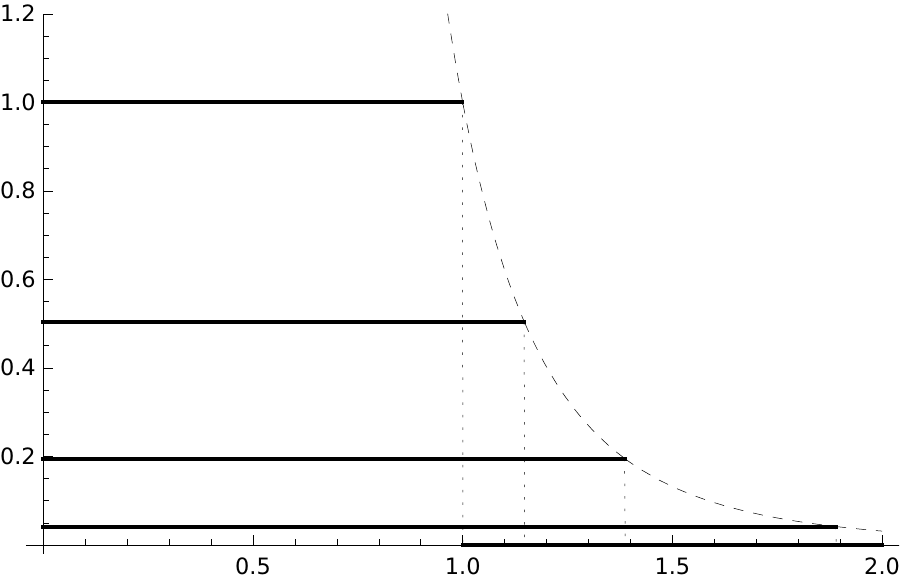}
		\caption{Case $n=3$. Solid lines: plots of $u(t,\cdot)$ for $t=0$, $t=0.04$, $t=0.08$, $t=0.12$. }
	\end{subfigure}
	\hfill
	\begin{subfigure}{0.45\textwidth}
		\includegraphics[width=\textwidth]{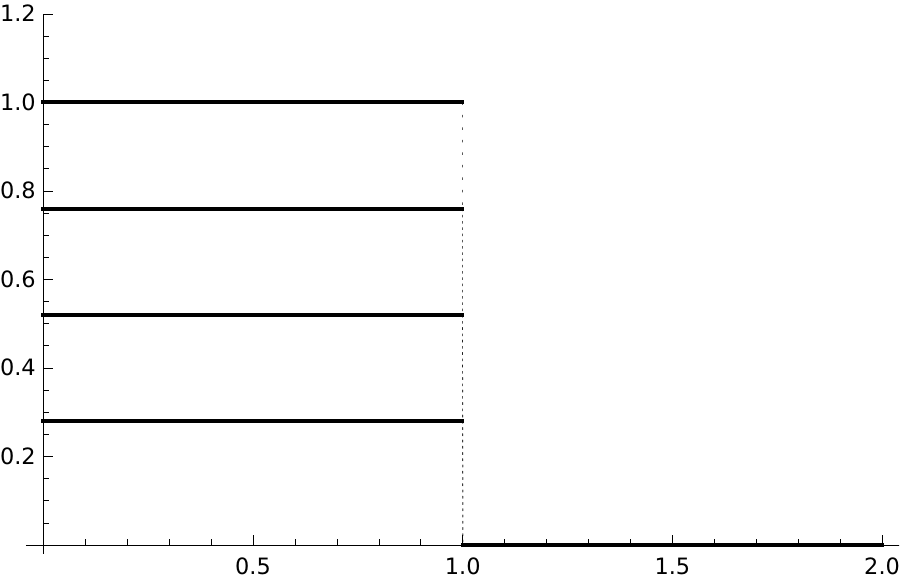}
		\caption{Case $n=4$. Solid lines: plots of $u(t,\cdot)$ for $t=0$, $t=0.01$, $t=0.02$, $t=0.03$.}
	\end{subfigure}
	
	\vspace{25pt}
	
	\begin{subfigure}{0.45\textwidth}
		\includegraphics[width=\textwidth]{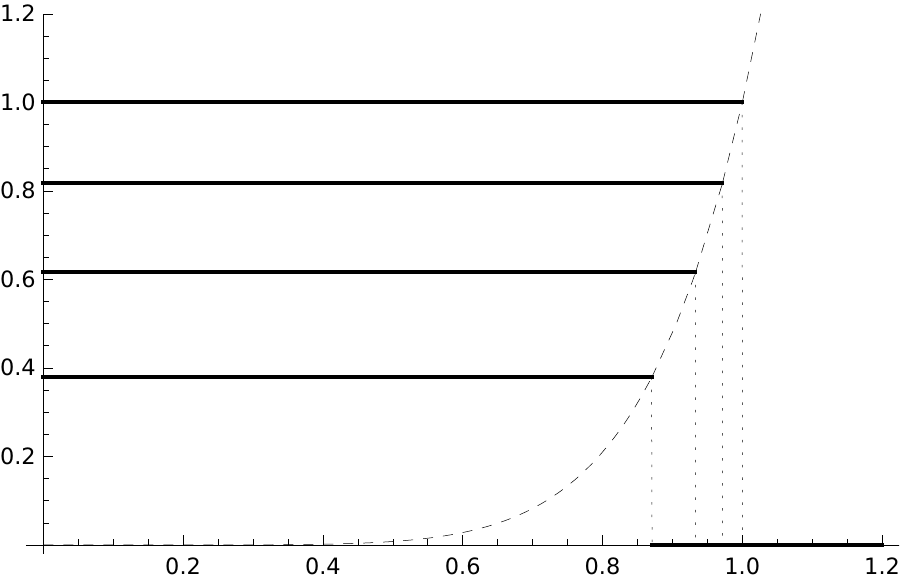}
		\caption{Case $n=5$. Solid lines: plots of $u(t,\cdot)$ for $t=0$, $t=0.005$, $t=0.01$, $t=0.015$. }
	\end{subfigure}
	\hfill
	\begin{subfigure}{0.45\textwidth}
		\includegraphics[width=\textwidth]{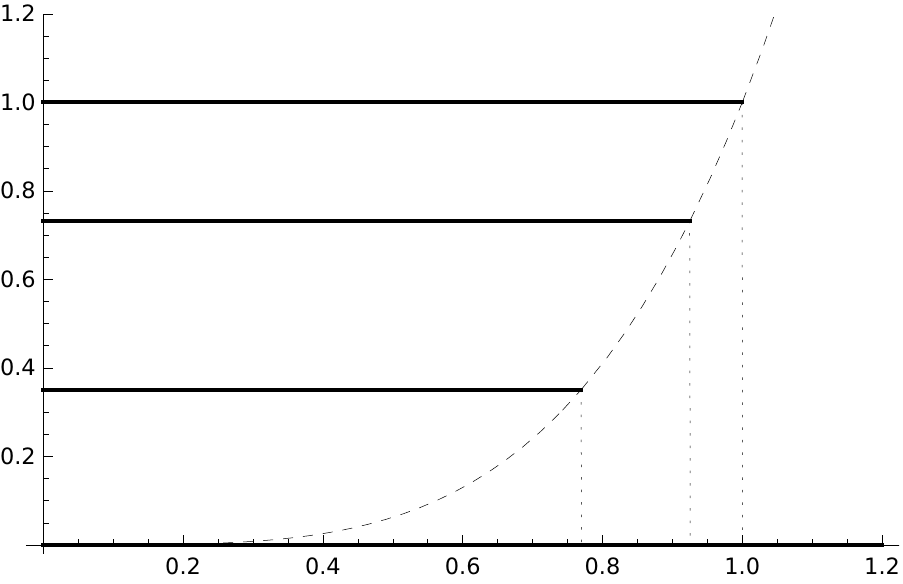}
		\caption{Case $n=6$. Solid lines: plots of $u(t,\cdot)$ for $t=0$, $t=0.005$, $t=0.01$, $t=0.015$.}
	\end{subfigure}
	
	\caption{Plots of the solution $u(t,x)$ emanating from the characteristic function of the unit ball as a function of $|x|$ for chosen values of $t$.}
	\label{plot_ball}
\end{figure}

In the case $n=2$ we were able to exhibit a calibration for the ball $B_R$, but not for its complement. Another possible ansatz on the Cahn-Hoffman vector field of form \eqref{Zin_Zout} is one where $Z_{in}$ is the calibration we constructed for $B_R$ and $Z_{out}$ is the choice
considered in \cite{GKM}: 
\begin{equation} \label{ode_z_soln_2}
	Z_{in}(x) =  \frac{1}{2} \left(\frac{|x|}{R}\right)^3 \frac{x}{|x|}  -\frac{3}{2} \frac{x}{R}, \qquad Z_{out}(x) = - \frac{x}{|x|}. 
\end{equation}
We calculate 
\begin{equation} \label{Zout2} 
	\operatorname{div} Z_{out} = - \frac{(n-1)}{|x|}, \qquad \nabla \operatorname{div} Z_{out} = \frac{(n-1)x}{|x|^3},
\end{equation}
hence 
\begin{equation} \label{ut_out} 
	u_t(t,x) = - \frac{(n-1)(n-3)}{|x|^3} \quad \text{in } \mathcal{D}'\left(\mathbb{R}^n \setminus \overline{B}_{R(t)}\right). 
\end{equation} 
If $n \geq 4$, this would lead to $u(t)$ being radially strictly increasing for positive $t$ and large values of $|x|$, which would be at odds with our choice of $Z_{out}$. In fact, if $n \geq 4$, $\operatorname{div} Z \not \in D_0^1(\mathbb{R}^n)$ for any $Z$ of this form. However, in smaller dimensions this ansatz remains a viable option. If $n=3$, it leads to the same solution as before. On the other hand, if $n=2$, we obtain a solution which is not of form \eqref{soln_form}. Instead, we are led to assume 
\begin{equation}\label{soln_form2} 
	u(t,x) = a(t) \mathbf 1_{B_{R(t)}} + \frac{t}{|x|^3}\mathbf 1_{\mathbb{R}^2\setminus B_{R(t)}}. 
\end{equation}  
We have: 
\begin{equation*}   
	\operatorname{div} Z_{in}(x) = 2 \frac{|x|^2}{R^3} - 3 \frac{1}{R}, \qquad \operatorname{div} Z_{out}(x) = - \frac{1}{|x|},
\end{equation*} 
\begin{equation*} 
	\nabla \operatorname{div} Z_{in}(x) = 4 \frac{x}{R^3}, \qquad \nabla \operatorname{div} Z_{out}(x) = \frac{x}{|x|^3}, 
\end{equation*}
\begin{equation*} 
	u_t = -\Delta \operatorname{div} Z = - \frac{8}{R^3}\, \mathcal{L}^2 \,\raisebox{-.127ex}{\reflectbox{\rotatebox[origin=br]{-90}{$\lnot$}}}\,_{B_R} + \frac{1}{|x|^3}\, \mathcal{L}^2 \,\raisebox{-.127ex}{\reflectbox{\rotatebox[origin=br]{-90}{$\lnot$}}}\,_{\mathbb{R}^2 \setminus B_R} + \frac{3}{R^2}\, \mathcal{H}^1 \,\raisebox{-.127ex}{\reflectbox{\rotatebox[origin=br]{-90}{$\lnot$}}}\,_{\partial B_R} 
\end{equation*} 
and, recalling \eqref{soln_form2},  
\begin{equation*}
	\left(a(t) - \frac{t}{R(t)^3}\right) \mathcal{H}^1(\partial B_{R(t)}) \frac{\mathrm{d} R}{\mathrm{d} t} = \left(a(t) - \frac{t}{R(t)^3}\right) \frac{\mathrm{d} }{\mathrm{d} t} \mathcal{L}^2(B_{R(t)}) = \frac{3}{R(t)^2} \mathcal{H}^1(\partial B_{R(t)}).  
\end{equation*}  
Thus, we arrive at ODE system
\begin{equation} \label{ode_n2}
	\frac{\mathrm{d} a}{\mathrm{d} t} = -\frac{8}{R^3}, \qquad \frac{\mathrm{d} R}{\mathrm{d} t} = \frac{3R}{aR^3 - t}.
\end{equation} 
This system is not autonomous, but it can be integrated by noticing that along trajectories
\begin{equation*} 
	\frac{\mathrm{d}}{\mathrm{d} t} \left(aR^3 - t\right) = \frac{9aR^3}{aR^3 - t} - 8 - 1 = \frac{9t}{aR^3 - t} 
\end{equation*}   
and so
\begin{equation*} 
	aR^3 = \sqrt{a_0^2 R_0^6 + 9 t^2} + t. 
\end{equation*} 
This implies, first of all, that 
\begin{equation*} 
	a(t) > \frac{t}{R(t)^3}
\end{equation*} 
for all $t>0$ and the form of solution \eqref{soln_form2} is preserved as long as the solution does not vanish. Furthermore, we can rewrite the system \eqref{ode_n2} in decoupled form 
\begin{equation}\label{ode_n2_ex} 
	\frac{\mathrm{d}}{\mathrm{d} t} \log a = \frac{- 8}{\sqrt{a_0^2 R_0^6 + 9 t^2} + t}, \qquad \frac{\mathrm{d}}{\mathrm{d} t} \log R^3 = \frac{9}{\sqrt{a_0^2 R_0^6 + 9 t^2}}.
\end{equation}
These equations can be explicitly integrated: 
\begin{equation*} 
a(t) = a_0 \frac{a_0^4 R_0^{12} + 8 a_0^2 R_0^{12}t^2}{\left(\sqrt{a_0^2 R_0^6 + 9 t^2} + 3 t \right)^2 \left(a_0^2 R_0^6 + 6 t^2 + 2 t \sqrt{a_0^2 R_0^6 + 9 t^2}\right)},   
\end{equation*} 
\begin{equation*} 
	 R(t) = R_0 \sqrt{1 + 6t \frac{3 t + \sqrt{a_0^2 R_0^6 + 9 t^2}}{a_0^2 R_0^6}}.
\end{equation*} 
We observe that the solutions exist globally and 
\begin{equation*} 
	\lim_{t \to \infty} a(t) = 0, \qquad \lim_{t \to \infty} R(t) = \infty. 
\end{equation*} 
In particular $u$ stays in the form \eqref{soln_form2} for all $t>0$.  

Finally we consider $n=1$. In this case, both ans\" atze considered before lead to the same solution: 
\begin{equation} 
	Z_{in}(x) =  \frac{1}{2} \left(\frac{x}{R}\right)^3  -\frac{3}{2} \frac{x}{R}, \qquad Z_{out}(x) = - \mathrm{sgn}\, x 
\end{equation}
which coincides with \eqref{ode_z_soln}. Repeating the calculations following \eqref{ode_z_soln}, we obtain a solution of form \eqref{soln_form} satisfying \eqref{final_soln}, i.\,e.
\begin{equation} 
	u(t) = a(t)\mathbf 1_{B_{R(t)}}, \qquad a(t) = a_0 \left(1 + \frac{6}{a_0 R_0^3} t\right)^{-\frac{1}{2}}, \qquad R(t) = R_0 \left(1 + \frac{6}{a_0 R_0^3} t\right)^{\frac{1}{2}}. 
\end{equation}
Note that now, as opposed to the case $n \geq 3$, the coefficient multiplying $t$ is positive. Like in $n=2$, the extinction time is infinite and we have 
\begin{equation*} 
	\lim_{t \to \infty} a(t) = 0, \qquad \lim_{t \to \infty} R(t) = \infty.  
\end{equation*} 
This concludes the proof of Theorem \ref{MAIN2}. 

\subsection{Stacks} 

\begin{figure}
	\centering
	\begin{subfigure}{0.45\textwidth}
		\includegraphics[width=\textwidth]{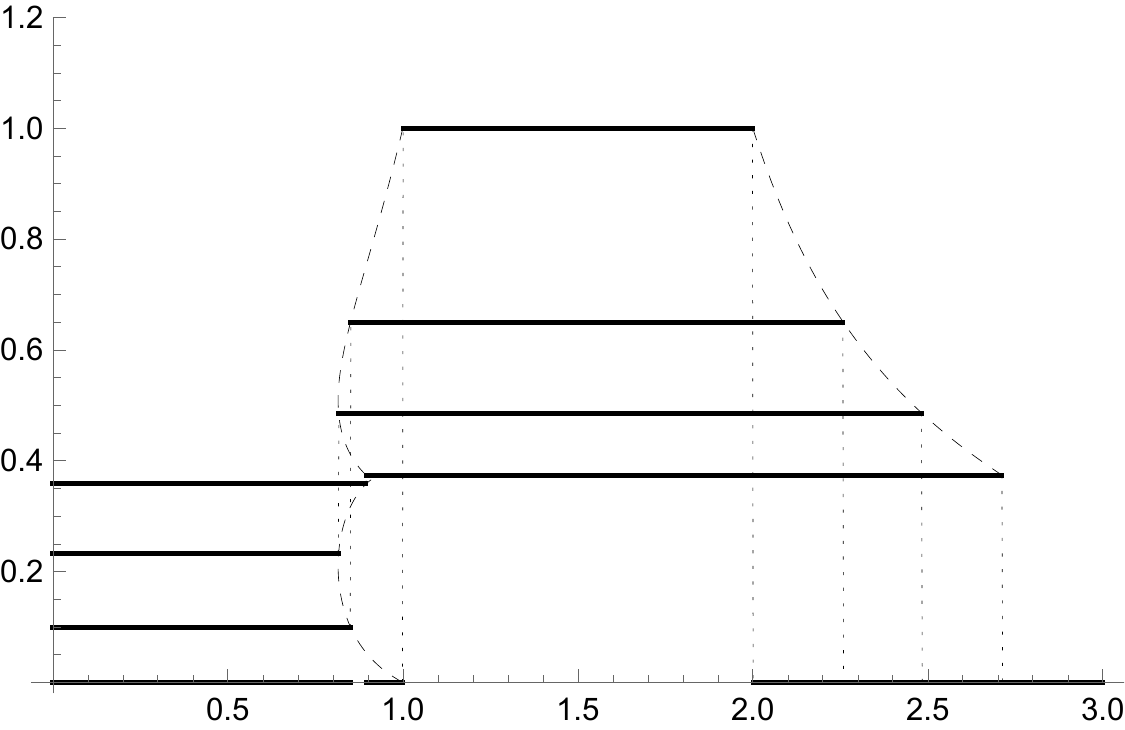}
		\caption{Case $n=1$. Solid lines: plots of $u(t,\cdot)$ for $t=0$, $t=0.025$, $t=0.5$, $t=0.075$. }
	\end{subfigure}
	\hfill
	\begin{subfigure}{0.45\textwidth}
		\includegraphics[width=\textwidth]{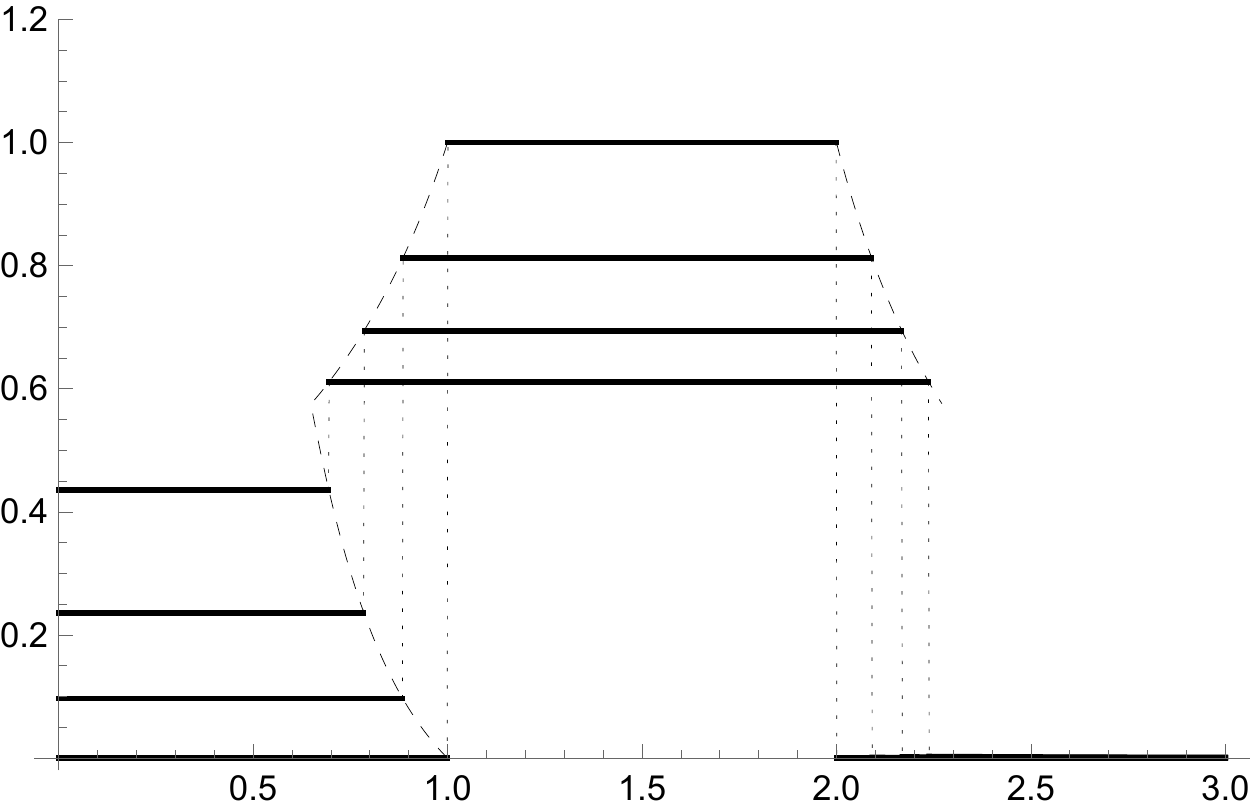}
		\caption{Case $n=2$. Solid lines: plots of $u(t,\cdot)$ for $t=0$, $t=0.01$, $t=0.02$, $t=0.03$.}
	\end{subfigure}
	
	\vspace{25pt}
	
	\begin{subfigure}{0.45\textwidth}
		\includegraphics[width=\textwidth]{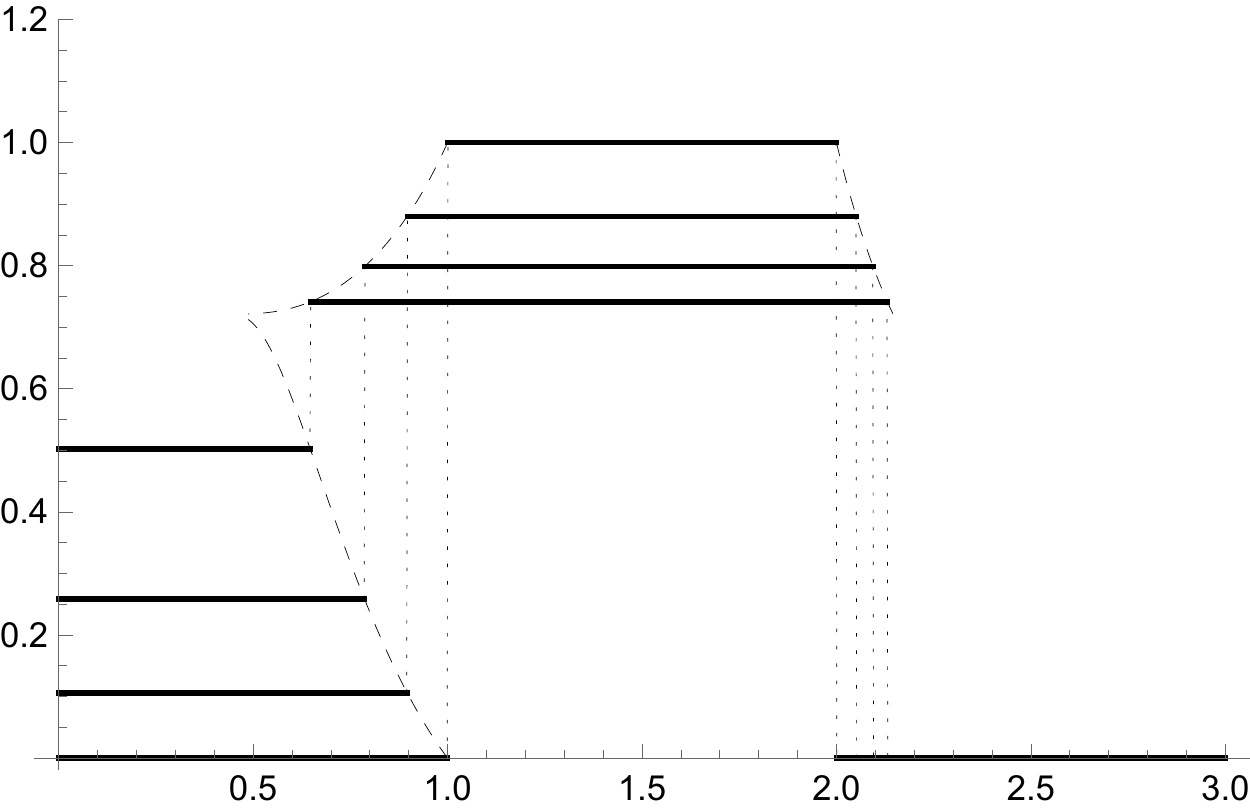}
		\caption{Case $n=3$. Solid lines: plots of $u(t,\cdot)$ for $t=0$, $t=0.006$, $t=0.012$, $t=0.018$. }
	\end{subfigure}
	\hfill
	\begin{subfigure}{0.45\textwidth}
		\includegraphics[width=\textwidth]{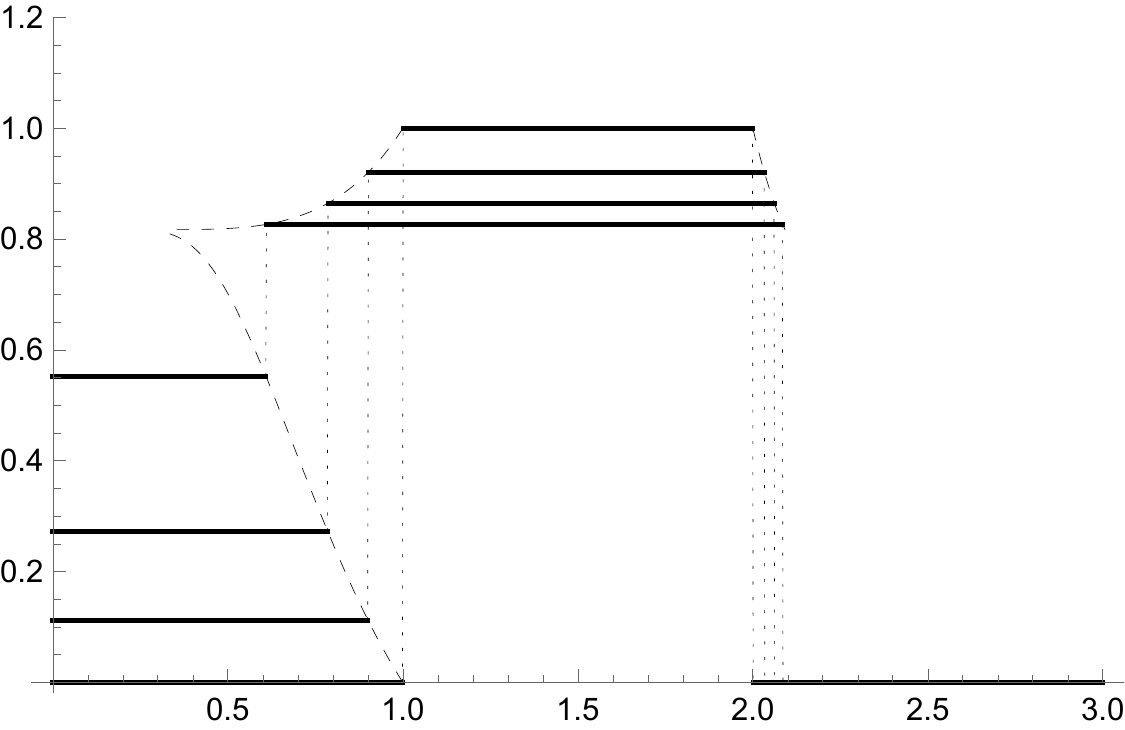}
		\caption{Case $n=4$. Solid lines: plots of $u(t,\cdot)$ for $t=0$, $t=0.004$, $t=0.008$, $t=0.012$.}
	\end{subfigure}
	
	\vspace{25pt}
	
	\begin{subfigure}{0.45\textwidth}
		\includegraphics[width=\textwidth]{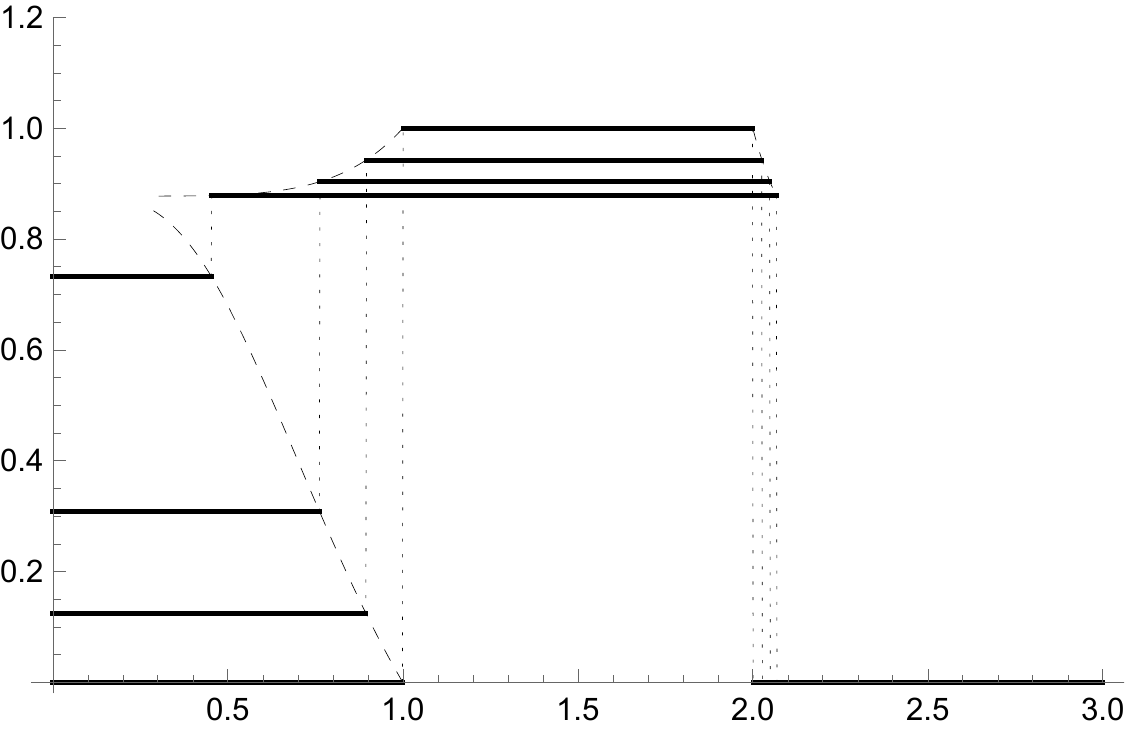}
		\caption{Case $n=5$. Solid lines: plots of $u(t,\cdot)$ for $t=0$, $t=0.003$, $t=0.006$, $t=0.009$. }
	\end{subfigure}
	\hfill
	\begin{subfigure}{0.45\textwidth}
		\includegraphics[width=\textwidth]{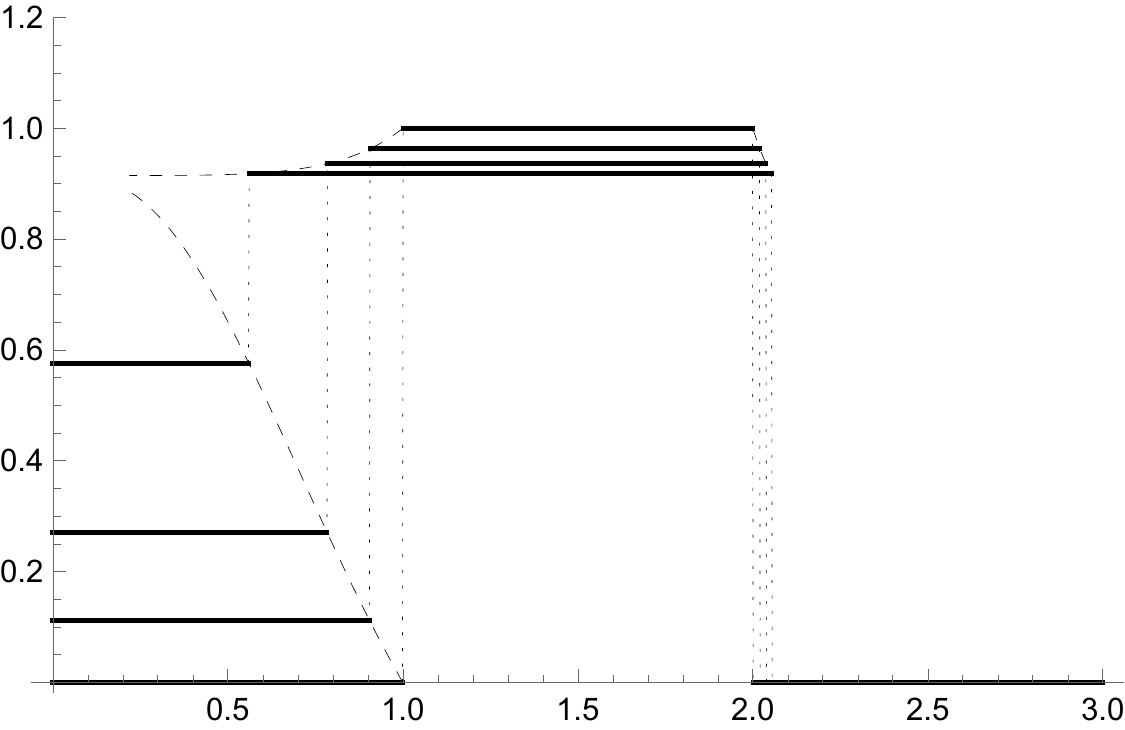}
		\caption{Case $n=6$. Solid lines: plots of $u(t,\cdot)$ for $t=0$, $t=0.002$, $t=0.004$, $t=0.006$.}
	\end{subfigure}
	
	\caption{Plots of the solution $u(t,x)$ emanating from the characteristic function of annulus $A_{R^1}^{R^2}$ as a function of $|x|$ for chosen values of $t$ with $R^1=1$, $R^2=2$.}
	\label{plot_annulus}
\end{figure}

Using the calibrations we constructed for generalized annuli, we will now derive a system of ODEs locally prescribing the solution emanating from any piecewise constant, radially symmetric datum (a \emph{stack}).
\begin{definition} 
	Let $w \in D(TV)$. We say that $w$ is a stack if there exists a number $N \in \mathbb{N}$ and sequences $0<R^0<R^1<\ldots<R^{N-1}$, $a^0, a^1, \ldots, a^N$ with $a^k \in \mathbb{R}$ such that 
	\[w = a^0 \mathbf{1}_{B_{R^0}} + a^1 \mathbf{1}_{A_{R^0}^{R^1}} +\ldots + a^{N-1} \mathbf{1}_{A_{R^{N-2}}^{R^{N-1}}} + a^N \mathbf{1}_{\mathbb{R}^n \setminus B_{R^{N-1}}}.\]
\end{definition}  
Suppose first that $n\neq 2$, in which case all connected components of level sets of any stack $w$ are calibrable. Let $u_0$ be a stack
\begin{equation} \label{datum_stack}
	u_0 = a^0_0 \mathbf{1}_{B_{R^0_0}} + a^1_0 \mathbf{1}_{A_{R^0_0}^{R^1_0}} +\ldots + a^{N-1}_0 \mathbf{1}_{A_{R^{N-2}_0}^{R^{N-1}_0}} + a^N_0 \mathbf{1}_{\mathbb{R}^n \setminus B_{R^{N-1}_0}}, 
\end{equation} 
where $a^{k-1}\neq a^k$ for $k=1,\ldots, N$, $a^N_0 = 0$. We expect that if $u$ is the solution emanating from $u_0$, then $u(t,\cdot)$ is a stack of form 
\begin{equation} \label{ut_stack}
	u(t,\cdot) = a^0(t) \mathbf{1}_{B_{R^0(t)}} + a^1(t) \mathbf{1}_{A_{R^0(t)}^{R^1(t)}} +\ldots + a^{N-1}(t) \mathbf{1}_{A_{R^{N-2}(t)}^{R^{N-1}(t)}} + a^N(t) \mathbf{1}_{\mathbb{R}^n \setminus B_{R^{N-1}(t)}}, 
\end{equation} 
with $a^N(t)=0$ for all $t>0$, and that $a^{k-1}\neq a^k$, $k=1,\ldots, N$ for small $t$. We construct a Cahn-Hoffman vector field $Z(t,\cdot)$ for $u(t,\cdot)$ by pasting together calibrations $Z^k$ for $B_{R^0(t)}$, $A_{R^k(t)}^{R^{k+1}(t)}$, $\mathbb{R}^n \setminus B_{R^{N-1}(t)}$ with suitable choice of signatures. We have 
\begin{multline} 
u_t = -\Delta \operatorname{div} Z = - \Delta \operatorname{div} Z^0\, \mathcal{L}^n \,\raisebox{-.127ex}{\reflectbox{\rotatebox[origin=br]{-90}{$\lnot$}}}\,_{B_{R^0}} - \sum_{k=1}^{n}  \Delta \operatorname{div} Z^k\, \mathcal{L}^n \,\raisebox{-.127ex}{\reflectbox{\rotatebox[origin=br]{-90}{$\lnot$}}}\,_{A_{R^{k-1}}^{R^k}} - \Delta \operatorname{div} Z^N\, \mathcal{L}^n \,\raisebox{-.127ex}{\reflectbox{\rotatebox[origin=br]{-90}{$\lnot$}}}\,_{\mathbb{R}^n\setminus B_{R^N}} \\ + \sum_{k=0}^{n} \frac{x}{|x|} \cdot(\nabla \operatorname{div} Z^k - \nabla \operatorname{div} Z^{k+1})\,\mathcal{H}^{n-1} \,\raisebox{-.127ex}{\reflectbox{\rotatebox[origin=br]{-90}{$\lnot$}}}\,_{S_{R^k}}.
\end{multline}
We denote 
\begin{multline*}\left. \frac{x}{|x|} \cdot(\nabla \operatorname{div} Z^k - \nabla \operatorname{div} Z^{k+1}) \right|_{S_{R^k}} \\= z^k_{rr}(R^k) - z^{k+1}_{rr}(R^k) + \frac{n-1}{R^k}(z^k_{r}(R^k)- z^{k+1}_{r}(R^k))-\frac{n-1}{(R^k)^2}(z^k(R^k)- z^{k+1}(R^k))\\= z^k_{rr}(R^k) - z^{k+1}_{rr}(R^k) =: d^k.
\end{multline*}
The values of $d^k$ are functions of $R^0, \ldots, R^{N-1}$.  Assuming that $R^k$ are regular enough and $\varepsilon$, $|t-s|$ are small enough, we have 
\[\frac{\mathrm{d}}{\mathrm{d} t}\int_{A_{R^k(s)-\varepsilon}^{R^k(s)+\varepsilon}} u = \int_{A_{R^k(s)-\varepsilon}^{R^k(s)+\varepsilon}} u_t,\]
whence 
\[(a^k(t) - a^{k-1}(t)) \mathcal{H}^{n-1}(S_{R^k(t)}) \frac{\mathrm{d} R^k}{\mathrm{d} t} = (a^k(t) - a^{k-1}(t)) \frac{\mathrm{d}}{\mathrm{d} t} \mathcal{L}^n(A_{R^k(s)-\varepsilon}^{R^k(t)}) = d^k(t) \mathcal{H}^{n-1}(S_{R^k(t)}).\]
Further, for $k=0,\ldots,N$, we denote by $\lambda^k$ the value of $-\Delta \operatorname{div} Z^k(t,\cdot)$ which is constant since $Z^k$ is a calibration. Then, we can write down the system of ODEs for $a^k$ and $R^k$:
\begin{equation} \label{stack_ode} 
	\frac{\mathrm{d} a^k}{\mathrm{d} t} = \lambda^k \text{ for } k=0,\ldots N, \qquad \frac{\mathrm{d} R^k}{\mathrm{d} t} = \frac{d^k}{a^k - a^{k+1}} \text{ for } k=0,\ldots N-1. 
\end{equation}
Let $c_0^k$ denote $c_0$ given by \eqref{annulus_soln_ncs} if $\mathrm{sgn}\,(a^{k+1} - a^k) = \mathrm{sgn}\, (a^k - a^{k-1})$ or by \eqref{annulus_soln} if $\mathrm{sgn}\,(a^{k+1} - a^k) \neq \mathrm{sgn}\, (a^k - a^{k-1})$, with $R^{k+1}$ and $R^k$ in place of $R_1$ and $R_0$. Then, we have
\begin{equation}\label{ak_ode} 
	\lambda^0 = \mathrm{sgn}\,(a^1-a^0)\frac{n(n+2)}{(R^0)^3}, \quad \lambda^k = 2n(n+2)\mathrm{sgn}\,(a^{k+1} - a^k) c_0^k \text{ for } k=1,\ldots,N-1, \quad \lambda^N = 0. 
\end{equation}
We observe that in a neighborhood of any initial datum $R^0_0, \ldots, R^{N-1}_0$,  $a^0_0, \ldots, a^N_0$, $R^k_0 < R^{k+1}_0$, $a^k_0 \neq a^{k+1}_0$, the r.\,h.\,s.\ of \eqref{stack_ode} is regular in $R^0, \ldots, R^{N-1}$,  $a^0, \ldots, a^N$, so locally the system has a unique solution. Unique solvability fails when a time instance $t>0$ is reached such that $a^k(t) = a^{k+1}(t)$, $R^k(t) = R^{k+1}(t)$ or $R_0 = 0$. In such case $u(t,\cdot)$ is again a stack with a smaller $N$, and we can restart our procedure. This concludes the proof of Theorem \ref{MAIN3}.

\begin{figure}
	\begin{subfigure}{0.45\textwidth}
		\includegraphics[width=\textwidth]{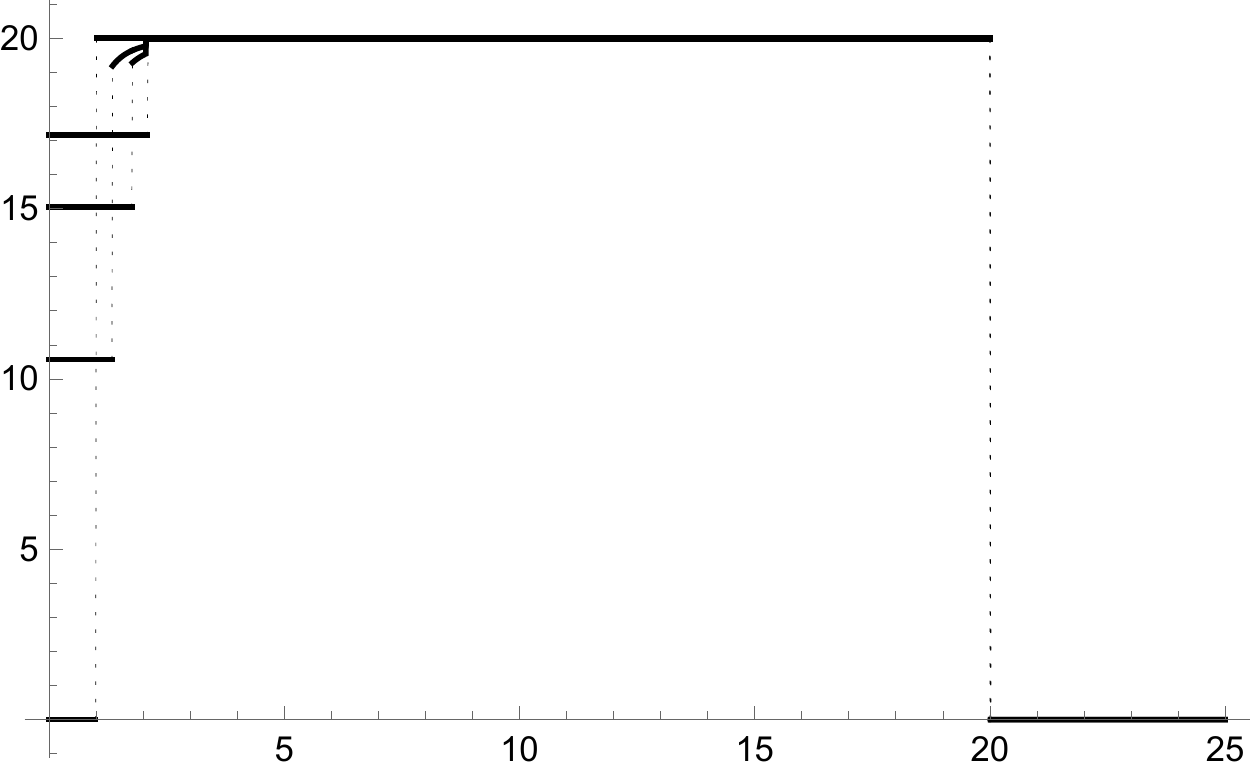}
		\caption{Plot over $[0,25]$.}
	\end{subfigure}
	\hfill
	\begin{subfigure}{0.45\textwidth}
		\includegraphics[width=\textwidth]{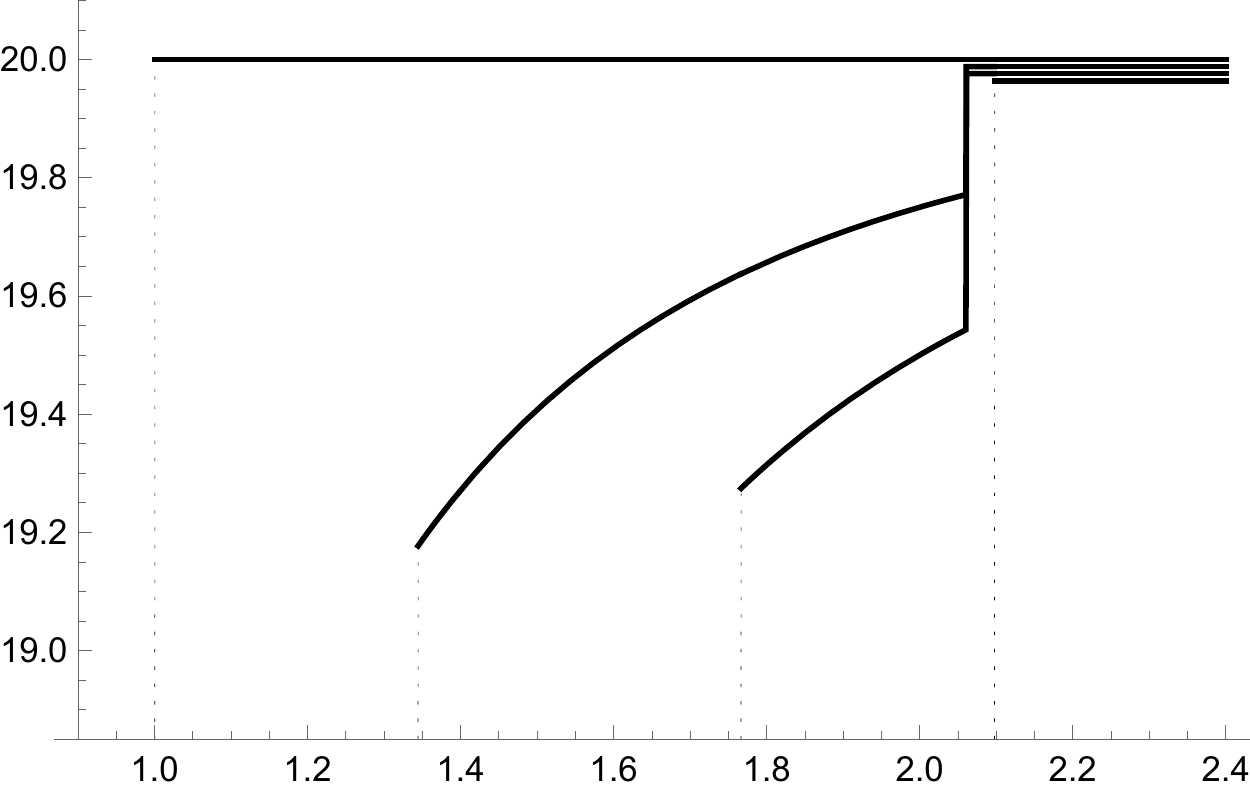}
		\caption{Magnified plot over $[0.9,2.4]$.}
	\end{subfigure}
	\caption{Plots of the solution $u(t,x)$ emanating from the characteristic function of annulus $A_{R^1}^{R^2}$ with $R_0=1$, $R_1=20$ in $n=2$ as a function of $|x|$ for $t=0$, $t=2$, $t=4$, $t=6$.}
	\label{plot_annulus2_bending}
	
	\vspace{40pt} 
	
	\begin{centering}
		\includegraphics[width=0.45\textwidth]{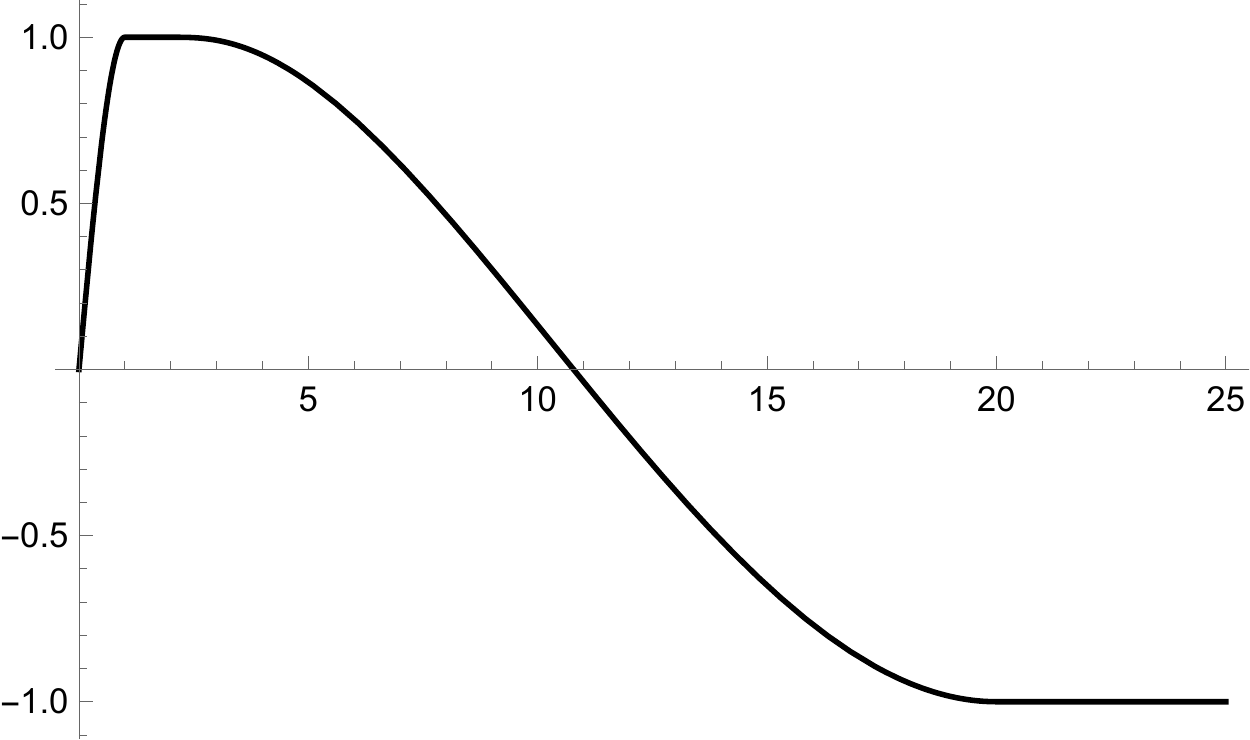}

	\caption{Plot of the Cahn-Hoffman vector field $Z(t,x)$ for the characteristic function of annulus $A_{R^1}^{R^2}$ with $R_0=1$, $R_1=20$ in $n=2$ as a function of $|x|$.}
	\label{plot_annulus3_bending}
	\end{centering}
\end{figure}

Next we deal with the remaining case of dimension $n=2$. In this case, our attempt to obtain a radial calibration failed for complements of balls and for some annuli. Again, let $u_0$ be a stack of form \eqref{datum_stack}. For $k=1,\ldots,N$, let $\sigma^k = \mathrm{sgn}\,(a^k_0 - a^{k-1}_0)$. We assume the following ansatz on the solution $u$ and the associated field $Z$ for small $t>0$: 
\begin{equation} \label{stack_ansatz}
	u(t,\cdot)= a^0(t) \text{ on } B_{R^0(t)}, u(t,\cdot) = a^k(t) \text{ on } A_{\max(R^{k-1}(t), R^k(t)/Q_*)}^{R^k(t)} \text{ if } \sigma^{k+1} \neq \sigma^k, \ k=1,\ldots,{N-1},
\end{equation}
\begin{multline} \label{stack_ansatz_Z} 
	Z(t,x)= \sigma^k \tfrac{x}{|x|} \text{ on } A_{R^{k-1}(t)}^{R^k(t)} \text{ if } \sigma^{k+1} \neq \sigma^k \text{ or on } A_{R^{k-1}(t)}^{R^k(t)/Q_*} \text{ if }\sigma^{k+1} = \sigma^k,\ k=1,\ldots,{N-1}, \\ Z(t,x) = \sigma^{k+1} \tfrac{x}{|x|} \text{ on } \mathbb{R}^n \setminus B_{R^{N-1}(t)}.
\end{multline}
We complete the definition of a Cahn-Hoffman field $Z$ consistent with \eqref{stack_ansatz}, \eqref{stack_ansatz_Z} by pasting the calibrations $Z^k$ with suitable choice of signatures into the gaps left in \eqref{stack_ansatz_Z}. This leads to 
\[u_t(t,\cdot) = \lambda^0(t) \text{ in } \mathcal{D}'(B_{R^0(t)}),\]
\begin{align*} 
	u_t(t,x) &= \lambda^k(t) \text{ in } \mathcal{D}'\left(A_{\max(R^{k-1}(t), R^k(t)/Q_*)}^{R_k(t)}\right), \\ u_t(t,x) &= \frac{\sigma^k}{|x|^3} \text{ in } \mathcal{D}'\left(A_{R^{k-1}(t)}^{R^k(t)/Q_*}\right) \text{ if } \sigma^k \neq \sigma^{k+1} \text{ or in } \mathcal{D}'\left(A_{R^{k-1}(t)}^{R^k(t)}\right) \text{ if } \sigma^k = \sigma^{k+1},\ k=1, \ldots, N-1, \\    
	u_t(t,x) &= \frac{\sigma^N}{|x|^3} \text{ in } \mathcal{D}'(\mathbb{R}^2 \setminus B_{R^N(t)}).
\end{align*} 
Moreover, $u_t(t,\cdot) \in M(\mathbb{R}^2)$ and 
\[u_t \,\raisebox{-.127ex}{\reflectbox{\rotatebox[origin=br]{-90}{$\lnot$}}}\,_{S_{R^k}} = \frac{x}{|x|}\cdot ((\nabla \operatorname{div} Z)^- - (\nabla \operatorname{div} Z)^+) \mathcal{H}^{1} \,\raisebox{-.127ex}{\reflectbox{\rotatebox[origin=br]{-90}{$\lnot$}}}\,_{S_{R^k}} =: d^k\]
for $k = 0, \ldots, N-1$, where $(\nabla \operatorname{div} Z)^\pm$ are the one sided limits as $|x| \to (R^k)^\pm$. The values of $d^k$ are functions of $R^0, \ldots, R^{N-1}$. Reasoning as in the case $n \neq 2$, the evolution of $R^k$ is governed by equations
\begin{equation}\label{ode_stack_R2}
	\frac{\mathrm{d} R^k}{\mathrm{d} t} = \frac{d^k}{u(t,x)\big|_{|x|=(R^k)^-} - u(t,x)\big|_{|x|=(R^k)^+}}.
\end{equation} 
The values $u(t,x)\big|_{|x|=(R^k)^+}$ are either prescribed by ODEs 
\begin{equation} \label{ode_stack_a2} 
	\frac{\mathrm{d} a^k}{\mathrm{d} t} = \lambda^k
\end{equation} 
with $\lambda^k$ functions of $R^0, \ldots, R^{N-1}$ in calibrable regions where $u(t,x) = a^k(t)$, or explicitly determined by $u_t(t,x) = \sigma^k /|x|^3$ in bending regions. It is important to note that in the case $\sigma^k \neq \sigma^{k+1}$, $R^{k-1} \leq R^k/Q_*$ the functions $d^k, \lambda^k$ do not depend on $R^{k-1}$. Thus, one can first solve a part of the system \eqref{ode_stack_R2}, \eqref{ode_stack_a2} for the outer annuli, then calculate $u$ in the bending region (without knowing a priori its inner boundary) and move on to solving innermore parts of \eqref{ode_stack_R2}, \eqref{ode_stack_a2}. This way, finding the solution is indeed again reduced to solving a system of ODEs. 

The qualitative behavior resulting from this procedure is rather intricate and may be surprising. To showcase this, we include Figure \ref{plot_annulus2_bending} picturing the evolution emanating from the characteristic function of a thick annulus. 
	
	Let us explain the evolution in a few words. The inner part of the initial facet corresponding to the annulus instantaneously bends downwards. Meanwhile, the outer boundary of the facet expands outwards, at relatively low speed (practically invisible in the picture). Since the ratio of outer to inner radius of the facet is constant, this means that the whole facet slowly moves outwards. The combined effect of this and the bending results in the very steep (but continuous!) part of the graph between the facet and the bending part. 

At the same time, the facet corresponding to the inside ball also expands outwards, gradually consuming the bending part. In the final pictured time instance, the whole bending part has disappeared and the solution is a stack again. 


\section*{Acknowledgement} 
The authors would like to thank the anonymous reviewers for their helpful comments.

\noindent
Y.\ Giga: \\
Graduate School of Mathematical Sciences, The University of Tokyo, 3-8-1 Komaba, Meguro-ku, Tokyo 153-8914, Japan, labgiga@ms.u-tokyo.ac.jp \\
H.\ Kuroda: \\
Department of Mathematics, Hokkaido University, 
Kita 10, Nishi 8, Kita-ku, Sapporo, Hokkaido 060-0810, Japan, kuro@math.sci.hokudai.ac.jp \\
M.\ {\L}asica: \\
Graduate School of Mathematical Sciences, The University of Tokyo, 3-8-1 Komaba, Meguro-ku, Tokyo 153-8914, Japan; \\
Institute of Mathematics, Polish Academy of Sciences, ul.\ \'Sniadeckich 8, 00-656 Warszawa, Poland, mlasica@impan.pl

\end{document}